\documentclass[notitlepage,reqno,11pt]{amsart}
\usepackage{latexsym,amssymb, epsfig, amsmath,amsfonts, subfigure,amsthm,mathrsfs}

\usepackage{rotating}
\usepackage[toc,page]{appendix}
\usepackage{color}
\usepackage{multirow}
\usepackage{relsize}
\usepackage{microtype}
\usepackage[foot]{amsaddr}
\usepackage{dsfont}

\usepackage[colorlinks, allcolors=blue]{hyperref}

\usepackage[margin=1in]{geometry}
\newcommand{\eps} {\varepsilon}

\numberwithin{equation}{section}
 \newtheorem{assumption}{Assumption}[section]
\newtheorem{lemma}{Lemma}[section]
\newtheorem{theorem}{Theorem}[section]
\newtheorem{definition}{Definition}[section]
\newtheorem{coro}{Corollary}[section]
\newtheorem{prop}{Proposition}[section]
\newtheorem{remark}{Remark}[section]

\newlength{\defbaselineskip}
\newcommand{\setlinespacing}[1]%
           {\setlength{\baselineskip}{#1 \defbaselineskip}}

\newcommand{\RR}{{\mathbb R}}
\newcommand{\ZZ}{{\mathbb Z}}
\newcommand{\NN}{{\mathbb N}}

\def\E{\mathbb{E}}
\def\P{\mathbb{P}}
\def\R{\mathbb{R}}

\newcommand{\sG}{{\mathcal{G}}}
\newcommand{\sF}{{\mathcal{F}}}

\newcommand{\sM}{{\mathcal{M}}}

\newcommand{\beql}[1]{\begin{equation}\label{#1}}
\newcommand{\eeq}{\end{equation}}

\newcommand{\beqal}[1]{\begin{eqnarray}\label{#1}}
\newcommand{\eeqa}{\end{eqnarray}}
\newcommand{\beq}{\begin{displaymath}}
\newcommand{\eeqno}{\end{displaymath}}
\newcommand{\bali}[1]{\begin{align}\label{#1}}
\newcommand{\eali}{\begin{align}}
\newcommand{\balino}{\begin{align*}}
\newcommand{\ealino}{\begin{align*}}

\newcommand{\ep}{\varepsilon}

\newcommand{\Var}{\text{\rm Var}}
\newcommand{\Cov}{\text{\rm Cov}}
\newcommand{\wt}{\widetilde}

\newcommand{\mfF}{\mathfrak{F}}

\newcommand{\mfa}{\mathfrak{a}}

\newcommand{\bone}{{\mathbf 1}}

\newcommand{\qandq}{\quad\mbox{and}\quad}

\newcommand{\qasq}{\quad\mbox{as}\quad}
\newcommand{\qinq}{\quad\mbox{in}\quad}

\newcommand{\non}{\nonumber}
\newcommand{\RA}{\Rightarrow}

\newcommand{\baa}{\begin{eqnarray*}}
\newcommand{\eaa}{\end{eqnarray*}}

\newcommand{\mds}{\mathds{1}}

\newcommand{\ttl}{\Large  SPDE  for stochastic SIR epidemic models \\[5pt] with infection-age  dependent infectivity}

\begin{document}

\title[]{\ttl}

\author[Guodong \ Pang]{Guodong Pang$^*$}
\address{$^*$Department of Computational Applied Mathematics and Operations Research,
George R. Brown School of Engineering and Computing, 
Rice University,
Houston, TX 77005}
\email{gdpang@rice.edu}

\author[{\'E}tienne \ Pardoux]{{\'E}tienne Pardoux$^\dag$}
\address{$^\dag$Aix--Marseille Univ, CNRS, I2M, Marseille, France}
\email{etienne.pardoux@univ-amu.fr}

\date{\today}

\begin{abstract} 
We study the stochastic SIR epidemic model with infection-age dependent infectivity for which a measure-valued process is used to describe the ages of infection for each individual. 
We establish a functional law of large numbers (FLLN) and a functional central limit theorem (FCLT) for the properly scaled measure-valued processes together with the other epidemic processes to describe the evolution dynamics. 
In the FLLN, assuming that the hazard rate function of the infection periods is bounded and the ages at time $0$ of the infections of the initially infected individuals are bounded, we obtain a PDE limit for the LLN-scaled measure-valued process, for which we characterize its solution explicitly. 
The PDE is linear with a boundary condition given by the unique solution to a set of Volterra-type nonlinear integral equations.
In the FCLT, we obtain an SPDE for the CLT-scaled measure-valued process, driven by two independent white noises coming from the infection and recovery processes. The SPDE is also linear and coupled with the solution to a system of stochastic Volterra-type linear integral equations driven by three independent Gaussian noises, one from the random infection functions in addition to the two white noises mentioned above. The solution to the SPDE can be also explicitly characterized, given this auxiliary process. The uniqueness of the SPDE solution is established under stronger assumptions  on the distribution function of the infectious duration.
\end{abstract}


\maketitle



\allowdisplaybreaks

\section{Introduction}

Kermack and MacKendrick introduced in their seminal paper \cite{KM27} an integral equation model to describe the SIR epidemic dynamics with infection-age dependent infectivity and recovery rate. 
Various extensions of this model have  been developed to study more realistic epidemic dynamics with dependence on infection ages  (see, e.g., \cite{thieme1993may,inaba2004mathematical,ZhangPeng2007,magal2013two,YChen2014} and Chapter 13 in \cite{martcheva2015introduction}).  An individual-based stochastic SIR epidemic model has been recently developed in \cite{FPP2020b} and further studied in \cite{PP2020-FCLT-VI,PP2021-PDE,FPP2022-MPMG,mougabe2023extinction,FPPZ2024} where each individual has an independent random infectivity function so that the infection and recovery processes depend on the age of infection.  It is shown in \cite{FPP2020b} (see also section 2 in \cite{FPP2022-MPMG} for the same result under weaker assumptions) that the Kermack - McKendrick model in \cite{KM27} is the law of large numbers limit, as the size of the population tends to infinity, of that  individual-based stochastic SIR epidemic model.
In \cite{PP2021-PDE}, the authors show that, by introducing the age of infection as a new variable, one can turn the limiting LLN integral equation model into a PDE model. 

In the present paper, we continue studying the individual-based stochastic SIR model by investigating the fluctuations of the epidemic dynamics around the PDE limit. 
We take a different approach from our previous work \cite{PP2021-PDE}, by using a measure-valued process to describe the infection dynamics instead of a two-parameter process. In particular, the measure-valued process gives at each time a ``mass'' for each individual's infection age if still infected (see \eqref{eqn-muN-rep}). To study its dynamics, we also come up with a novel representation using two independent Poisson random measures (PRMs), resembling the ``birth" and ``death" processes in the continuous-time branching process setting. 
so that one dictates the new infections and the other dictates the recoveries, see \eqref{muN-eq}. The measure-valued process with this representation can be regarded as an infinite-dimensional ``birth" and ``death" process. The novelty in this representation lies in the way that determines which individual will recover next based on their infection age. 
To completely describe the epidemic dynamics, we also need to include the number of susceptible individuals and the total force of infection (which is the aggregate of the random infectivity functions of all infected individuals evaluated with their infection ages at each time). 
The number of infected individuals can be obtained from the measure-valued process. 

We first establish the FLLN  for the LLN-scaled measure-valued processes, which results in a linear PDE limit with a boundary condition which is the solution to a set of Volterra-type integral equations. It recovers the PDE result in our previous work \cite{PP2021-PDE}. Because of the measure-valued representation for the PDE limit, we provide a new proof for the uniqueness of its solution, which requires that the hazard rate function for the distribution of the infection period is bounded.  This new proof uses a duality argument with an associated backward PDE.
 In addition, we also provide a new proof for the tightness of the LLN-scaled processes by exploiting the evolution dynamics mentioned above.  

We then establish a functional central limit theorem (FCLT) for the CLT-scaled measure-valued processes, whose limit is an SPDE (see \eqref{SPDE-CLT}). The SPDE is driven by two white noises, one from the infection process and the other from the recovery process (that is, the ``birth" and ``death" processes mentioned above, respectively). 
The SPDE is linear, coupled with the solution to a system of stochastic Volterra-type linear integral equations driven by three independent Gaussian noises (in addition to the two white noises, a third coming from the randomness of the random infectivity functions). 
We are able to characterize the solution to the SPDE explicitly, see \eqref{eqn-SPDE-solution}. Under the additional conditions on the distribution function of the infection period (the density and its derivative being locally bounded),  we also show the uniqueness of the solution (using a similar duality argument as for the PDE). 
The convergence of the other associated processes follows from a slight modification of the proofs in 
 our previous work  \cite{PP2020-FCLT-VI}, by taking into account the initial conditions with infection-age dependence. 

The convergence of  the CLT-scaled measure-valued processes is proved in the space $D(\RR_+,  H^{1}(\RR_+)'_w )$ (see the notations below). 
We exploit some useful properties of stochastic integrals with respect to PRMs, particularly, the moment formulas (see, e.g.,  \cite{pardoux-PRM-book})
and fourth moment estimates in Lemma \ref{le:4thmoment}.  With these tools, we are able to prove tightness of the stochastic integral terms with respect to the compensated PRMs for the infection and recovery processes, by verifying the moment criterion for tightness of processes in $D$ in Billingsley \cite{billingsley1999convergence}.  In these calculations, we have to handle some challenges caused by the  functional that induces the recovery process based on the infection ages. 
Moreover, for the CLT-scaled measure-valued processes, we are also able to derive a prelimit ``SPDE" driven by the PRMs so that the convergence to the SPDE limit can be established by verifying the convergence for the driving stochastic components. 

Finally, we remark that although there have been a few studies on establishing PDE limits for individual-based stochastic epidemic models, see for example \cite{reinert1995asymptotic,foutel2020individual,duchamps2023general,delmas2023individual,foutel2023optimal,PP2023spatially,PP2024-MP-PDE}, very little work exists for the study of the fluctuations of the stochastic dynamics around the PDE limits. The work in \cite{clemenccon2008stochastic} establishes both the PDE and SPDE limits for stochastic epidemic model with contract-tracing, tracking the infection duration since detection for each individual, but the model itself is Markovian, unlike our model. 
To our best knowledge, this is the first work to establish an SPDE limit for non-Markovian epidemic models in the literature, so this work will lay the foundation for future work on this subject. 
On the other hand, FCLT results have been established for stochastic non-Markovian epidemic models, for example, our earlier works on non--Markov epidemic models with constant infection rate \cite{PP-2020,PP-2020b}, and some previous works in the literature \cite{wang1977gaussian,wang1975limit,wang1978asymptotic}. 
Our work is also somewhat related to the SPDE limits in the queueing context, see for example \cite{decreusefond2008functional,kaspi2013spde}, but our approach of establishing the convergence of the measure-valued processes is very different since we exploit properties of PRMs and the measure-valued process is not necessarily Markov and hence no martingale properties can be exploited.

\subsection{Organization of the paper}
The paper is organized as follows. 
We summarize the notation  in the end of this section. 
In Section \ref{sec-model}, we describe the model in detail and use a measure-valued process and other associated processes to describe the epidemic dynamics. 
In Section \ref{sec-generic1orderPDE}, we present a linear first-order PDE and the associated properties that will be used in the discussions of the PDE and SPDE limits. 
In Section \ref{sec-FLLN}, we state the FLLN and present the PDE limit. In Section \ref{sec-FCLT}, we state the FCLT results and present the SPDE limit and its solution. 
In Section \ref{sec-FLLN-proof}, we prove the FLLN in Theorem \ref{thm-FLLN}, focusing on a new proof for the uniqueness of the solution to the PDE.
In Section \ref{sec-wh-muN-proof}, we prove the weak convergence of the CLT-scaled measure-valued process and prove the uniqueness of the solution to the SPDE.
In Sections \ref{sec-SF-proof} and \ref{sec-IR-proof}, we provide sketch proofs for the other associated epidemic processes. 

\subsection{Notation}
The following notation will be used throughout the paper. $\RR$ is the set of real numbers and $\RR_+$ the set of non-negative real numbers, $\NN$ the set of natural numbers, and $\ZZ$ the set of integers. For $a,b\in \RR$, $a\vee b = \max\{a,b\}$ and $a\wedge b = \min\{a,b\}$.  We use ${\bf1}_{A}$ to denote the indicator function of a set $A$. We use $\mds$ to denote the constant function $\mds(t)\equiv 1$ for any $t\in \RR$. 

Given any metric space $S$, $C_b(S)$ is the space of bounded and continuous real-valued functions on $S$, and $C_c(S)$ the space of continuous functions with compact support.  $C^1(S)$ is the space of real-valued, once continuously differential functions on $S$, and $C^1_c(S)$ is the subspace of 
functions in $C^1(S)$  with compact support, while $C^1_b(S)$ is the subspace of $C^1(S)$ of functions that are bounded and have bounded first order derivative. Let $C^2(S)$ be the space of  real-valued, twice continuously differential functions on $S$ and $C^2_c(S)$ be the subspace of $C^2(S)$ of functions with compact support. 
We will mostly use $S=\RR_+$ or $\RR_+^2$. Let $L^p(\RR_+)$, $p\ge 1$,  be the space  of measurable functions $f$ on $\RR_+$ such that $\int_{\RR_+} |f(x)|^pdx<\infty$. Let $L^p_{loc}(\RR_+)$ be the corresponding space in which the associated property holds only locally.

The space of Radon measures on a Polish space $S$, endowed with the Borel $\sigma$-algebra, is denoted by $\sM(S)$, and $\sM_F(S)$ is the subspace of finite and nonnegative measures. The space $\sM(S)$ is equipped with the weak
 topology, that is, a sequence of measures $\{\mu^n\}$ in $\sM(S)$ is said to converge to $\mu$ in the weak topology (denoted by $\mu^n\to^w \mu$)
if and only if for every $\varphi\in C_b(S)$, 
\begin{equation}\label{eqn-conv-meas}
\int_S \varphi(x) \mu_n(dx) \to \int_S \varphi(x) \mu(dx), \quad \text{as} \quad n\to \infty.
\end{equation}

The Sobolev space $H^1(\R_+)$ is the Hilbert space consisting of continuous functions  $u:\R_+\mapsto\R$ which are such that $u\in L^2(\R_+)$ and there exists a function $u'\in L^2(\R_+)$ such that for all $t>0$, 
$u(t)=u(0)+\int_0^t u'(s)ds$. Recall that  $H^1(\R_+)\subset C_b(\R_+)$.
We shall consider the dual space $(H^1(\R_+))'$ of $H^1(\R_+)$, which we equip with its weak topology.  Let $H^1_c(\R_+)$ be a subspace of $H^1(\R_+)$ consisting of functions with compact support. Further, let $H^2(\R_+)$ be 
the Hilbert space consisting of continuous functions  $u:\R_+\mapsto\R$ such that $u, u', u'' \in L^2(\R_+)$, and $H^2_c(\R_+)$ be the corresponding subspace of functions with compact support. 

We write $\langle \mu, \varphi\rangle = \int_S\varphi(x) \mu(dx)$ for a Borel measurable function $\varphi:S\to \RR$ that is integrable with respect to a measure $\mu\in\sM(S)$. 
In our case $S = \RR_+$. 
The symbol $\delta_x$ is used to denote the measure with unite mass at the point $x\in S$. 

We use $D:=D(\RR_+, \RR)$ to denote the space of $R$-valued, c{\`a}dl{\`a}g functions on $\RR_+$, endowed with the usual Skorohod $J_1$ topology; see \cite{billingsley1999convergence}. Let $C$ be the subspace of $D$ of continuous functions.

We shall consider elements of $D(\R+;\sM_{F}(\R_+))$ and $D(\R_+;(H^{1}(\R_+))')$. It follows from the results in \cite{mitoma1983tightness} that tightness and weak convergence of a sequence $\{X^N,\ N\ge1\}$ in $D(\R_+;\sM_{F}(\R_+))$ (resp. in $D(\R_+;(H^{1}(\R_+))')$)
will follow from tightness and weak convergence of the sequence $\{\langle X^N,\varphi\rangle,\ N\ge1\}$ for any $\varphi\in C_b$ (resp.  $H^1(\R_+)$).

\medskip

\section{Model and Results} \label{sec-main}

\subsection{Model} \label{sec-model} 
We consider an SIR model for a homogeneous population, each of which has a varying infectivity depending on the age (elapsed time) of infection. 
Let $N$ be the population size, and the numbers of susceptible, infectious and recovered individuals at each time $t$ are denoted by $S^N(t),  I^N(t), R^N(t)$, respectively.  Then we have the balance equation
$N= S^N(t) +  I^N(t) + R^N(t)$, $t\ge 0.$
Assume that $I^N(0)>0$. 
Let $\mu^{N}(t,d\mfa)$ be the measure-valued process describing  the infection-ages of the infected individuals at time $t$. For each $t\ge 0$, $\mu^{N}(t,d\mfa)$ is a finite measure over $\RR_+$. For convenience, we write $\mu^N_t$ or $\mu^N(t)$. 
It is clear that $I^N(t) = \langle \mu^{N}(t), \mds \rangle$ for each $t\ge0$, where $\mds$ is the constant function $\mds(t)\equiv1$ for each $t\ge0$.

Let  $\{\lambda_i(\cdot)\}_{i\in \NN}$ and
 $\{\lambda_{-j}(\cdot)\}_{j=1,\dots, I^N(0)}$ be the nonnegative random infectivity functions taking values in $D$
 for the newly infected and initially infected individuals, respectively. They  are non zero only during the infectious period.  We assume that they are mutually independent and have the same law. 
 Let $\bar\lambda(t) = \E[\lambda_1(t)]$ for $t\ge 0$.  Note that $\bar\lambda(\cdot)$ also takes values in $D$. 

Let $A^N(t)$ be the cumulative number of newly infected/exposed individuals in $(0,t]$, with event times $\{\tau^N_i: i \in \NN\}$.
For each individual $i$, let $\eta_i$ be the associated infected period, i.e., 
\begin{equation} \label{eqn-eta}
\eta_i := \sup\{t>0: \lambda_i(t)>0\}\,. 
\end{equation}
The variables $\{\eta_i\}$ are i.i.d. with a distribution function $F(\cdot)$.
 Let $F^c=1-F$. We shall assume that $F$ has a density $f$ w.r.t. the Lebesgue measure, and denote by $h=f/F^c$ the associated hazard rate function. Similar to \eqref{eqn-eta}, we also have $\eta_{-j}= \inf\{t>0: \lambda_{-j}(t)>0\}\,$ for $j=1,\dots, I^N(0)$, representing the infection duration for the initially infected individuals. The variables $\eta_{-j}$ have the same distribution function $F$. 

Let  $\tau_{j,0}, j=1,\dots, I^N(0)$ be the times of being infected for the initially infected individuals at time zero.  
We construct them from the corresponding $\lambda_{-j}$ as follows. 
$\tilde{\tau}_{j,0} = -\tau_{j,0}$,  $j=1,\dots, I^N(0)$,  which are the corresponding ages of infection at time zero,
are given as $\tilde{\tau}_{j,0}=(U_j\eta_{-j})\wedge\bar\mfa$, 
where $\{U_j,\ j\ge1\}$ is a sequence of i.i.d. $\mathcal{U}([0,1])$ r.v.'s, which are supposed to be globally independent of all other given random inputs, and $\bar\mfa>0$ is arbitrarily fixed. As a result, the
$\{\tilde{\tau}_{j,0},\ 1\le j\le I^N(0)\}$ are i.i.d., and we denote their common law by $\bar\mu_0(\mds)^{-1}\bar\mu_0$, which is a probability measure on $\R_+$, where $\bar\mu_0$ is a  measure which satisfies the following assumption:
\begin{assumption} \label{AS-g}
$\bar{\mu}_0(\{0\})=0$, $\mu_0$ has a compact support included in $[0,\bar\mfa]$, for some $\bar\mfa>0$, and its total mass satisfies $\bar\mu_0(\mds)<1$.  
\end{assumption}
 Let  $ \eta^{0}_{j}$, $j=1,\dots, I^N(0)$, be the variables representing the corresponding remaining infected periods, where 
$$\eta^0_{j}:= \sup\big\{t>0: \lambda_{-j}(\tilde{\tau}_{j,0}+t)>0\big\}>0\,.
$$


The aggregate infectivity process  $\mathfrak{F}^N(t)$  at time $t$ is given by
 \begin{align} \label{eqn-mfF}
 \mathfrak{F}^N(t) & =  \sum_{j=1}^{I^N(0)} \lambda_{-j}(\tilde{\tau}_{j,0}+t)  +\sum_{i=1}^{A^N(t)} \lambda_i(t- \tau^N_i ) \,, \quad t\ge 0.  
 \end{align}
 The instantaneous infection rate at time $t$ can be written as 
  \begin{align}  \label{eqn-Upsilon}
 \Upsilon^N(t) & =  \frac{S^N(t)}{N} \mathfrak{F}^N(t), \quad t \ge 0. 
 \end{align}
 Then the infection process $A^N(t)$ can be written as 
\begin{align} \label{eqn-An-PRM}
A^N(t) = \int_0^t \int_0^\infty \bone_{v \le \Upsilon^N(s^-) } Q_{inf}(ds,dv)\,, 
\end{align}
where $Q_{inf}$ is a standard Poisson random measure (PRM) on $\R_+^2$. 

 We assume that the three following sources of randomness are mutually independent:  generation of new infections via the PRM $Q_{inf}$, the random varying infectivity functions $\{\lambda_j(\cdot)\}_{j\ge 1}$ of the newly infected individuals, and the pairs $\{(\lambda_{-j},\tau_{j,0})\}_{j\ge1}$ infectivity functions and infection ages of the initially infected individuals. 



The number of susceptible individuals satisfies
\begin{align} \label{eqn-Sn-rep}
S^N(t) & = S^N(0) - A^N(t) = N -  I^N(t) - R^N(t)\,.
\end{align}

The measure-valued process $\mu^N_t(d\mfa)$ can be expressed as
\begin{align} \label{eqn-muN-rep}
\mu^N_t(d\mfa) = \sum_{j=1}^{I^N(0)}{\bf1}_{\eta^0_{j} >t} \delta_{\tilde{\tau}_{j,0}+t}(d\mfa) + \sum_{i=1}^{A^N(t)}{\bf1}_{\tau^N_i + \eta_i>t} \delta_{t-\tau^N_i}(d\mfa). 
\end{align}
Or equivalently, for any  function $\varphi \in C_b(\RR_+)$, 
\begin{align*} 
\langle \mu^N_t, \varphi\rangle = \sum_{j=1}^{I^N(0)}{\bf1}_{\eta^0_{j} >t}\varphi(\tilde{\tau}_{j,0}+t) + \sum_{i=1}^{A^N(t)}{\bf1}_{\tau^N_i + \eta_i>t} \varphi(t-\tau^N_i). 
\end{align*}
For convenience, we also write $\langle \mu^N_t, \varphi\rangle$ as $ \mu^N_t(\varphi)$ sometimes. 

The initial condition $\mu^N_0$  is given by
  \begin{equation} \label{eqn-muN0}
  \mu^N_0=\sum_{j=1}^{I^N(0)}\delta_{\tilde\tau_{j,0}}\,\,.
  \end{equation}

The number of infected individuals at time $t$ can be written as 
\begin{align}\label{eqn-In-rep}
I^N(t) &= \langle \mu^N_t, \mds\rangle  \nonumber \\
&=  \sum_{j=1}^{I^N(0)} \bone_{\eta^0_j > t} + \sum_{i=1}^{A^N(t)} \bone_{\tau^N_i + \eta_i > t}\,,
\end{align}
and the number of recovered individuals at time $t$:
\begin{align} \label{eqn-Rn-rep}
R^N(t) = R^N(0)+  \sum_{j=1}^{I^N(0)} \bone_{\eta^0_j \le t} + \sum_{i=1}^{A^N(t)} \bone_{\tau^N_i + \eta_i \le t}\,. 
\end{align}

We next write an equation to describe the evolution dynamics of $\mu^N_t(\varphi)$. 
Before proceeding, we introduce a function: for any measure $\nu$ on $\R_+$, $w\to H(\nu,w)$ is the right--continuous increasing function such that (recalling  that $h=f/F^c$ is the hazard rate function of $\eta$): 
\begin{equation}\label{eqn-H-def}
 H(\nu,w)\le\mfa \Leftrightarrow w\le\frac{\nu(h{\bf1}_{[0,\mfa]})}{\nu(h)}\,,
\end{equation}
in other words, $H(\nu,\cdot)$ is the ``inverse of the distribution function of the normalized $h$--biased $\nu$''.
In particular, with $G(\mfa) := \frac{\nu(h {\bf1}_{[0,\mfa]})}{\nu(h)}$, we have $H(\nu, v) = G^{-1}(v)$ and $H(\nu,G(\mfa)) =\mfa$. 
Note that the function $H({\mu}^N_{s^-},v)$ in the last term plays the role of identifying which individual recovers next, through the ages of infection, see the justification below. 

It is then easy to show that the above expression of $\mu^N_t$ in \eqref{eqn-muN-rep} can be equivalently (in distribution) written as: for $\varphi \in C_b(\R_+)$,
\begin{equation}\label{muN-eq}
\begin{split}
{\mu}^N_t(\varphi)&={\mu}^N_0(\varphi(t+\cdot)+\int_0^t\int_0^\infty\varphi(t-s)
{\bf1}_{v\le{\Upsilon}^N(s^-)}Q_{inf}(ds,dv)\\
&\quad-\int_0^t\int_0^\infty\int_0^1\varphi(t-s+H({\mu}^N_{s^-},w)){\bf1}_{v\le\mu^N_{s^-}(h)}Q_{rec}(ds,dv,dw)\,,
\end{split}
\end{equation}
$Q_{inf}$ and $Q_{rec}$ being  two independent standard PRM, respectively on $\R_+^2$ and on $\R_+^2\times[0,1]$, representing the infection and recovery processes. The first term on the right hand side of \eqref{muN-eq} indicates the evolution of the initially infected individuals in terms of their infection ages, the second term indicating the new infection process and the third term indicating the recovery process for both the initially and newly infected individuals. In some sense, the expression in \eqref{muN-eq} can be regarded as a measure-valued birth and death process where the second term as the ``birth" and the third term as the ``death".  Let us justify the form of the ``death term'' in \eqref{muN-eq}. We first note that, since each infected individual, independently of the others and of the infection process, recovers at rate $h(\mfa(t))$ if $\mfa(t)$ is its infection age at time $t$, the total recovery rate in the population at time $t$ is $\mu^N_t(h)$. Moreover, we note that
$w$ is picked uniformly at random in $[0,1]$,  and that if $W$ is a $\mathcal{U}([0,1])$ r.v. independent of the process $\mu^N$, and $\mfa_i(t^-), 1\le i\le I^N(t^-)$ denote the ages of infection at time $t^-$ of the individuals infected at that time, we have that 
\[\P\left(H(\mu^N_{t^-},W)=\mfa_i(t^-) \,|\, \mu^N_{t^-}\right)=\frac{h(\mfa_i(t^-))}{\mu^N_{t^-}(h)},\]
which is the proper choice dictated by our model.


%
We remark that the measure-valued process $({\mu}^N_{t})_{t\ge0}$ itself is not Markov because of the random varying infectivity processes $(\lambda_j(\cdot))_{j\in \ZZ})$. However, in the special case that $\lambda_j(t) = \tilde\lambda(t) {\bf1}_{t\le \eta_j}$ for a deterministic function $\tilde\lambda(t)$, one can show that $({\mu}^N_{t})_{t\ge0}$ is Markov. 

\subsection{A linear first-order PDE} \label{sec-generic1orderPDE}
The PDE which follows will play a central role in this paper.
Suppose we have a continuous function $u_t(\mfa)=u(t,\mfa)$ on $\R_+^2$ which satisfies the following: for any smooth function $\varphi\in C^1_c(\R_+)$, 
\begin{equation}\label{genPDE-weak}
\left\{
\begin{aligned}
 \frac{d}{dt}\langle u_t,\varphi\rangle&=\langle u_t,\varphi'-h\varphi\rangle+\varphi(0)k(t)+\langle g_t,\varphi\rangle\,,\\
u(0,\mfa)&=u_0(\mfa),
\end{aligned}
\right.
\end{equation}
where $g_t(\mfa)=g(t,\mfa)$ is measurable and bounded.

We claim that such a function $u$ satisfies the following PDE
\begin{equation}\label{genPDE-strong}
\left\{
\begin{aligned}
\partial_t u+\partial_\mfa u=-hu+g,\\
u(0,\mfa)=u_0(\mfa),\quad u(t,0)=k(t).
\end{aligned}
\right.
\end{equation}
Moreover, the unique solution of this linear equation is given explicitly as
\begin{equation}\label{genexplicit-strong}
\begin{split}
u(t,\mfa)
&={\bf1}_{t<\mfa}\frac{F^c(\mfa)}{F^c(\mfa-t)}u_0(\mfa-t)+{\bf1}_{\mfa\le t}F^c(\mfa)k(t-\mfa)\\
&\quad+ \int_{(t-\mfa)^+}^t\frac{F^c(\mfa)}{F^c(\mfa-t+s)}g(s,\mfa-t+s)ds\,.
\end{split}
\end{equation}

To go from \eqref{genPDE-weak} to \eqref{genPDE-strong}, we can argue as follows: first choose $\varphi_n(\mfa)=(1-n\mfa)^+$, and let $n\to\infty$, which yields
the boundary condition $u(t,0)=k(t)$. Next, we integrate \eqref{genPDE-weak} on the interval $[0,t]$, with $\varphi$ satisfying $\varphi(0)=0$, and deduce the first line of \eqref{genPDE-strong}. A similar argument allows to go from \eqref{genPDE-strong} to \eqref{genPDE-weak}.
Uniqueness is proved using the formulation \eqref{genPDE-weak} and a duality argument, which will be given below. Hence, in order to show that the above explicit formula is correct, it suffices to show that it satisfies the equation, which is not so hard.
We also deduce from the formula \eqref{genexplicit-strong} the following formula, valid for any $\varphi\in C_c(\R_+)$:
\begin{equation}\label{genexplicit-weak}
\begin{split}
\langle u_t,\varphi\rangle
&=\int_0^\infty \varphi(\mfa+t)\frac{F^c(t+\mfa)}{F^c(\mfa)}u_0(\mfa)d\mfa+\int_0^t \varphi(t-\mfa)F^c(t-\mfa)k(\mfa)d\mfa\\
&\quad+ \int_0^t\int_{0}^\infty\ \varphi(t-s+r)\frac{F^c(t-s+r)}{F^c(r)}g(s,r)drds\,.
\end{split}
\end{equation}
Note that it is easy to recover \eqref{genPDE-weak} from \eqref{genexplicit-weak}. Indeed, if $\varphi \in C^1_c(\R_+)$, 
differentiating the right hand side of \eqref{genexplicit-weak}
w.r.t. $t$ produces $u_t(\varphi'-h\varphi)$, plus the derivative w.r.t. the upper bound of the two integrals $\int_0^t$, which produces the last two terms on the right hand side of the first line of \eqref{genPDE-weak}. 

The FLLN limit will solve a PDE of the above type, with $g=0$, and $u_0$ being an arbitrary measure, see \eqref{eqn-PDE}.

The FCLT limit will be a PDE of the same type, but with $u_0$, $k$ and $g$ being Gaussian random distributions, see \eqref{SPDE-CLT}.

\subsection{FLLN} \label{sec-FLLN}

We give ourselves three numbers $0<\bar{I}(0),\bar{S}(0)<1$ and $0\le\bar{R}(0)<1$ such that $\bar{I}(0) = \langle \bar\mu_0, \mds\rangle$ and
$\bar{I}(0)+\bar{S}(0)+\bar{R}(0)=1$. Moreover, we assume that $S^N(0)=[N\bar{S}(0)]$, $I^N(0)=[N\bar{I}(0)]$ and
$R^N(0)=[N\bar{R}(0)]$ (taking integer parts, or more precisely, setting $S^N(0)=\lfloor N\bar{S}(0)\rfloor$, $I^N(0)=\lfloor N\bar{I}(0)\rfloor$ and
$R^N(0)=\lceil N\bar{R}(0)\rceil $) are such that $S^N(0) + I^N(0) + R^N(0)=N$.

Define the LLN-scaled processes  $\bar{X}^N= N^{-1} X^N$ for any processes $X^N$.  It is clear that
$\bar{I}^N(0)\to\bar{I}(0)$, $\bar{S}^N(0)\to\bar{S}(0)$ and $\bar{R}^N(0)\to\bar{R}(0)$, as $N\to\infty$.
Moreover, by the strong law of large numbers (SLLN),
 $\bar\mu^N_0\Rightarrow\bar\mu_0$ a.s. as $N\to\infty$.  The notation $\Rightarrow$ refers to the fact that the convergence is in the sense of weak convergence of measures, while as random elements, the convergence is in the a.s. sense.

It follows from \eqref{muN-eq} that, $\overline{Q}_{inf}$ (resp. $\overline{Q}_{rec}$) denoting the compensated measure associated to $Q_{inf}$ (resp. $Q_{rec}$), we have for $\varphi \in C_b(\R_+)$,
\begin{equation}\label{eq-mubarN}
\begin{split}
\bar{\mu}^N_t(\varphi)
&=\bar{\mu}^N_0(\varphi(t+\cdot)+\int_0^t\varphi(t-s)\overline{\Upsilon}^N(s)ds-\int_0^t\bar{\mu}^N_s(\varphi(t-s+\cdot)h)ds \\
& \qquad + \bar{\mu}^{inf,N}_t(\varphi) - \bar{\mu}^{rec,N}_t(\varphi) \,,
\end{split}
\end{equation}
where 
\begin{align}
\bar{\mu}^{inf,N}_t(\varphi) &= \frac{1}{N}\int_0^t\int_0^\infty\varphi(t-s){\bf1}_{v\le{\Upsilon}^N(s^-)}\overline{Q}_{inf}(ds,dv)\,, \label{eq-mubarN-inf} \\
\bar{\mu}^{rec,N}_t(\varphi) &= \frac{1}{N}\int_0^t\int_0^\infty\int_0^1\varphi(t-s+H({\mu}^N_{s^-},w)){\bf1}_{v\le\mu^N_{s^-}(h)}\overline{Q}_{rec}(ds,dv,dw)\,,\label{eq-mubarN-rec}
\end{align} 
and we have used the following formula: for any $\psi\in C_b(\R_+)$,
\[ \int_0^1\psi(H(\nu,w))dw=\frac{\nu(\psi h)}{\nu(h)}\,,\]
which follows from the definition of $H$ in \eqref{eqn-H-def}. 

We first prove the following FLLN. It recovers  Theorem 3.1 in \cite{PP2021-PDE}. Its proof is given in Section \ref{sec-FLLN-proof}.

  \begin{theorem} \label{thm-FLLN} 
  Assume that the initial law $\bar{\mu}_0$ satisfies Assumption \ref{AS-g} and that the hazard rate function $h$ is locally bounded. 
   Then, as $N\to\infty$, 
     \begin{equation}
\big(\bar{S}^N,  \overline{\mfF}^N\big) \to \big(\bar{S}, \overline{\mfF}\big) \qinq D^2 
\end{equation}
in probability, where $ \big(\bar{S}, \overline{\mfF}\big) \in C\times D$ is the unique solution to the following set of integral equations, 
\begin{align}
\bar{S}(t) & =  \bar{S}(0)  - \bar{A}(t) = \bar{S}(0) -\int_0^t \overline{\Upsilon}(s) ds\,, \label{eqn-barS}\\
 \overline{\mfF}(t) 
 &= \int_0^{\infty} \bar\lambda(\mfa+t)  \bar{\mu}_0(d\mfa) 
  + \int_0^t  \bar\lambda(t-s)  \overline\Upsilon(s) ds\,, \label{eqn-overline-mfF} 
\end{align}
with  
\begin{equation} \label{eqn-bar-Upsilon}
\overline{\Upsilon}(t) = \bar{S}(t) \overline{\mfF}(t)\,.  
\end{equation}
If $\bar\lambda(\cdot) \in C$, then $ \overline{\mfF}(t)$ and $\overline{\Upsilon}(t)$ are continuous. 
Given $ \big(\bar{S}, \overline{\mfF}\big)$,  we have 
\begin{equation*}
\{\bar{\mu}^N_t\}_{t\ge0} \Rightarrow \{\bar{\mu}_t \}_{t\ge0} \qinq D(\RR_+, \sM_{F}(\RR_+)) \qasq N \to \infty,
\end{equation*}
where the limit solves the PDE: for $\varphi\in C^1_c(\RR_+)$,
\begin{equation}\label{eqn-PDE}
\left\{
\begin{split}
\frac{d}{dt}\bar{\mu}_t(\varphi)&=\varphi(0)\overline{\Upsilon}(t)+\bar{\mu}_t(\varphi'-h\varphi),\ t\ge0\,,\\
\bar{\mu}_0(\varphi)&=\int_0^\infty\varphi(\mfa)\bar{\mu}_0(d\mfa)\,,
\end{split}
\right.
\end{equation}
whose unique solution is given as
\begin{align}  \label{eqn-barmu}
\bar\mu(t,d\mfa) &=  {\bf1}_{\mfa<t}  F^c(\mfa)  \overline\Upsilon(t-\mfa)d \mfa + {\bf1}_{\mfa \ge t}  \frac{F^c(\mfa)}{F^c(\mfa-t)}  \bar\mu_0(d\mfa -t)\,,
 \end{align} 
  or equivalently, for $\varphi\in C_b(\R_+)$,
 \begin{align}
\begin{split}
\langle \bar\mu_t,\varphi\rangle
&=\int_0^\infty \varphi(\mfa+t)\frac{F^c(t+\mfa)}{F^c(\mfa)}\bar{\mu}_0(\mfa)d\mfa+\int_0^t \varphi(t-\mfa)F^c(t-\mfa) \overline\Upsilon(\mfa)d\mfa\,.
\end{split} 
 \end{align}
As a consequence,  we obtain 
$ (\bar{I}^N, \bar{R}^N) \to (\bar{I}, \bar{R})$ in $D^2$ in probability as $N\to \infty$, where 
\begin{align} \label{eqn-barI} 
\bar{I}(t) =   \int_0^{\infty} \frac{F^c(t+\mfa)}{F^c(\mfa)} \bar\mu_0(d\mfa) +  \int_{0}^t F^c(t-s) \overline{\Upsilon}(s) ds\,,
  \end{align}
  and
  \begin{align} \label{eqn-barR} 
\bar{R}(t) = \bar{R}(0)+  \int_0^{\infty} \left(1- \frac{F^c(t+\mfa)}{F^c(\mfa)} \right)   \bar\mu_0(d\mfa)    + \int_0^t F(t-s)  \overline{\Upsilon}(s) ds\,. 
\end{align}
Since $F$ is continuous, $\bar{I}$ and $\bar{R}$ are continuous. 
  \end{theorem}

Note that the argument which leads from \eqref{genPDE-weak} to \eqref{genPDE-strong} allows us to deduce from \eqref{eqn-PDE} 
the following (at least formally):
\begin{equation*}
\left\{
\begin{aligned}
\partial_t \bar\mu(t,\cdot)+\partial_\mfa \bar\mu(t,\cdot)=-h(\cdot)\bar\mu(t,\cdot),\\
\bar\mu(0,d\mfa)=\bar\mu_0(d\mfa),\quad \bar\mu(t,0)=\overline{\Upsilon}(t).
\end{aligned}
\right.
\end{equation*}

By \eqref{eq-mubarN} and Theorem \ref{thm-FLLN}, we also obtain the following expression for the LLN limit: for $\varphi \in C_b(\R_+)$,
\begin{equation}\label{eq-mubar-exp2}
\bar{\mu}_t(\varphi)
=\bar{\mu}_0(\varphi(t+\cdot)+\int_0^t\varphi(t-s)\overline{\Upsilon}(s)ds-\int_0^t\bar{\mu}_s(\varphi(t-s+\cdot)h)ds\,. 
\end{equation}

  \begin{remark} \label{rem-lambda-indep}
Let $\lambda_{-j}(t) = \tilde\lambda(t) \bone_{t \le \eta^0_{j}} $ and $\lambda_i(t)=\tilde\lambda(t)\bone_{ t\le  \eta_i}$, where $\tilde\lambda(t)$ is a deterministic function. 
We obtain the limit 
 \begin{align*}
 \overline{\mfF}(t) 
 &= \int_0^{\infty} \tilde\lambda(\mfa+t) \frac{F^c(\mfa+t)}{F^c(\mfa)}  \bar{\mu}_0(d\mfa)  + \int_0^t  \tilde\lambda(t-s) F^c(t-s)  \overline\Upsilon(s) ds \\
 &= \int_0^{\infty}\tilde\lambda(\mfa)\bar{\mu}_t(d\mfa)\,. 
 \end{align*}
 where  $\bar\mu_t(\mfa)$ is given in \eqref{eqn-barmu}. This second expression has a very intuitive interpretation that the aggregate infectivity is equal to 
 the infectivity function with respect to the distribution of the infection ages of the infectious individuals.  
 This is often assumed in the study of epidemic PDE models (see, e.g., \cite{inaba2001kermack,magal2013two,foutel2020individual}).  
 \end{remark}

\subsection{FCLT}  \label{sec-FCLT}

Define the CLT-scaled process:  for $\varphi\in C_b(\R_+)$,
\begin{align} \label{eqn-wh-muN-def}
\hat{\mu}^N_t(\varphi) = \sqrt{N} (\bar\mu^N_t(\varphi) - \bar\mu_t(\varphi)). 
\end{align}
and similarly for any other process $X^N$, $\widehat{X}^N= \sqrt{N}(\bar{X}^N- \bar{X})$ where $\bar{X}$ is the FLLN  limit of $\bar{X}^N$. 
We  make the following assumptions on the CLT-scaled initial quantities.

Recall that $\mu^N_0(d\mfa)$ is given by  \eqref{eqn-muN0} and that $\bar\mu_0$ satisfies  $ \bar{\mu}_0(\mds)  \in (0,1)$. 
We introduce the notation $\bar{\bar{\mu}}_0=\bar{\mu}_0(\mds)^{-1}\bar{\mu}_0$, and 
define
  \begin{equation} \label{eqn-hat-muN0}
 \hat\mu^N_0(\cdot)= \frac{1}{\sqrt{N}}\sum_{j=1}^{I^N(0)} \Big( \delta_{\tilde\tau_{j,0}} - \bar{\bar{\mu}}_0(\cdot) \Big)\,. 
  \end{equation}
   We have the following result, which is a direct consequence of 
 \cite[Theorem 14.3]{billingsley1999convergence}. 


\begin{prop}\label{prop-initCLT}
Under Assumption \ref{AS-g}, as $N\to\infty$, 
\[\hat{\mu}^N_0([0,\cdot])\Rightarrow \hat{\mu}_0([0,\cdot])  \qinq D([0,\bar\mfa], \RR),
\] 
where 
\[
\hat{\mu}_0([0,\cdot]):= \bar{\mu}_0(\mds)^{1/2} W^0\left(\bar{\bar{\mu}}_0([0,\cdot])\right)
\]
and $\{W^0(s),\ 0\le s\le t\}$ is a Brownian bridge, i.e., a 
centered Gaussian process whose covariance is given as $\E[W^0(s)W^0(s')]=s(1-s')$ for any $0\le s\le s'\le 1$.
\end{prop}

Note that our choices for $I^N(0)$, $S^N(0)$ and $R^N(0)$  imply readily that $\hat{I}(0)=\hat{S}(0)=\hat{R}(0)=0$.
This is consistent with the fact that $\hat{I}(0)=\hat{\mu}_0(\mds)=0$, since $W^0(1)=0$.

\begin{remark}
Since the mapping $\mfa\to W^0(\bar{\bar{\mu}}_0([0,\mfa]))$ is not of bounded variation, $\hat\mu_0$, which we define as the distributional derivative of the function 
$\mfa\to \bar{\mu}_0(\mds)^{1/2}W^0(\bar{\bar{\mu}}_0([0,\mfa]))$ is not a signed measure. However, if the distribution function of $\bar{\mu}_0$
is H\" oder continuous with exponent $\alpha>0$, then $\mfa\to W^0(\bar{\bar{\mu}}_0([0,\mfa]))$ is  an element of  $H^{s}_{loc}(\R_+)$ for any $s<\alpha/2$, see, e.g., formula (2.1) from \cite{di2012}, and its derivative belongs to $H^{-s'}(\R_+)$, for any $s'>1-\alpha/2$. 
As a matter of fact, the distribution
$\frac{d}{d\mfa}W^0\big(\bar{\bar{\mu}}_0([0,\cdot])\big)$ belongs to $(H^1(\RR_+))'$, since for any $\varphi\in H^1(\R_+)$, $\hat{\mu}_0(\varphi)$ can be defined by integration by parts. 
\end{remark}

\medskip


\begin{remark}
Since $W^0(t)$ can be written as $W^0(t)= W(t) -t W(1)$ for a Wiener process $W$, if the distribution $\bar{\bar{\mu}}_0(\cdot)$ has a density function $g_0$ on $[0,\bar\mfa]$,
that is, $d\bar{\bar{\mu}}_0 (d\mfa) = g_0(\mfa)d\mfa$,  then we can write 
\begin{align*}
\hat{\mu}_0([0,\mfa]) & =\bar{\mu}_0(\mds)^{1/2} W^0\left( \int_0^\mfa g_0(a)da  \right)\\
&= \bar{\mu}_0(\mds)^{1/2}\bigg( W \left(\int_0^\mfa g_0(a)da  \right) - \int_0^\mfa g_0(a)da W(1)\bigg) \,.
\end{align*}
 If $\wt{W}$ is another standard Brownian motion, we have that the pair
 $\left(W\left(\int_0^\mfa g_0(a)da\right),\int_0^\mfa g_0(a)da W(1)\right)$ has the same law as
 $\left(\int_0^\mfa\sqrt{g_0(a)}d\wt{W}(a),\int_0^\mfa g_0(a)da\int_0^\infty\sqrt{g_0(a)}d\wt{W}(a)\right)$. 
 Thus, 
 $\hat{\mu}_0([0,\mfa]) $ is equal in distribution to the following expression, where $W$is a standard Wiener process: 
\begin{align} \label{eqn-hat-mu0-Wiener}
\hat{\mu}_0([0,\mfa])= \bar{\mu}_0(\mds)^{1/2} \bigg(  \int_0^\mfa \sqrt{g_0(a)} d W(a)  -\int_0^\mfa g_0(a)da  \ \int_0^\infty \sqrt{g_0(a)} d W(a) \bigg) \,.
\end{align}

\end{remark}

\medskip

We next give a heuristic derivation of the SPDE; the formal proof will be given in Section  \ref{sec-wh-muN-proof}. 
We first deduce from \eqref{eq-mubarN} and \eqref{eq-mubar-exp2}
 that $\hat{\mu}^N_t$ has the following expression:  for any $\varphi \in C_b(\RR_+)$, 
\begin{equation}\label{eqn-hat-muN-rep}
\begin{aligned} 
\hat{\mu}^N_t(\varphi)&=\hat{\mu}^N_0(\varphi(t+\cdot))+\int_0^t\varphi(t-s)\widehat{\Upsilon}^N_sds
-\int_0^t\hat{\mu}^N_s(\varphi(t-s+\cdot)h)ds \\
& \qquad + \hat{\mu}^{inf,N}_t(\varphi) - \hat{\mu}^{rec,N}_t(\varphi)\,,
\end{aligned}
\end{equation}
where 
\begin{align}
\hat{\mu}^{inf,N}_t(\varphi) &=\frac{1}{\sqrt{N}}\int_0^t\int_0^\infty\varphi(t-s){\bf1}_{v\le \Upsilon^N(s^-)}\overline{Q}_{inf}(ds,dv)\,, \label{eqn-hat-muN-rep-inf} \\
\hat{\mu}^{rec,N}_t(\varphi) &= \frac{1}{\sqrt{N}}\int_0^t\int_0^\infty\int_0^1\varphi(t-s+H(\mu^N_{s^-},w)){\bf1}_{v\le\mu^N_{s^-}(h)}\overline{Q}_{rec}(ds,dv,dw)\,.  \label{eqn-hat-muN-rep-rec} 
\end{align} 
We expect that, $W_{inf}$ denoting a standard Gaussian white noise on $\R_+$, and $W_{rec}$ a  centered Gaussian white noise process on $\R_+^2$ such that
\begin{equation} \label{eqn-Wrec-cov}
 \E\left[\left(\int_0^\infty\int_0^\infty g(s,\mfa)W_{rec}(ds,d\mfa)\right)^2\right]=\int_0^\infty\int_0^\infty g^2(s,\mfa)h(\mfa)\bar{\mu}_s(d\mfa)ds,
\end{equation}
 if we define  for $\varphi\in H^1(\RR_+)$, 
\begin{align}
\hat{\mu}^{inf}_t(\varphi)&=\int_0^t\varphi(t-s)\sqrt{\overline{\Upsilon}(s)}W_{inf}(ds), \label{eqn-hat-mu-rep-inf}\\
\hat{\mu}^{rec}_t(\varphi)&=\int_0^t\int_0^\infty \varphi(t-s+\mfa)W_{rec}(ds,d\mfa)\,, \label{eqn-hat-mu-rep-rec}
\end{align}
then  for any $\varphi \in H^1(\R_+)$,  $\hat{\mu}^N_t(\varphi)\Rightarrow\hat{\mu}_t(\varphi)$ in $D$, where 
\begin{align} \label{eqn-SPDE-hat} 
\hat{\mu}_t(\varphi)&=\hat{\mu}_0(\varphi(t+\cdot))+\int_0^t\varphi(t-s)\widehat{\Upsilon}_sds
-\int_0^t\hat{\mu}_s(\varphi(t-s+\cdot)h)ds
+\hat{\mu}^{inf}_t(\varphi)+\hat{\mu}^{rec}_t(\varphi)\,, 
\end{align}
where $\widehat{\Upsilon}_s$ will be defined in \eqref{eqn-hat-Upsilon}. 

We differentiate this equation w.r.t. $t$, yielding (explanation: when we differentiate the right-hand side of the above equation w.r.t. $t$, we get 
$\hat{\mu}_t(\varphi')$, plus the result of differentiating w.r.t. the upper bounds of the integrals), hence integrating that derivative on the interval $[0,t]$, we obtain that for any $\varphi \in H^2_c(\RR_+)$, 
\begin{equation}\label{SPDE-CLT}
\begin{split}
\hat{\mu}_t(\varphi)&=\hat{\mu}_0(\varphi)+\int_0^t\hat{\mu}_s(\varphi'-h\varphi)ds+\varphi(0)\int_0^t[\widehat{\Upsilon}_sds+\sqrt{\overline{\Upsilon}(s)}W_{inf}(ds)]\\ &\quad
+\int_0^t\int_0^\infty \varphi(a)W_{rec}(ds,da)\,. 
\end{split}
\end{equation}

Recall the formula \eqref{genexplicit-weak} for the solution of the PDE \eqref{genPDE-strong}.
%
We hence conjecture that the SPDE \eqref{SPDE-CLT} has the unique solution: for $\varphi\in H^1(\R_+)$,
\begin{align} \label{eqn-SPDE-solution}
\hat{\mu}_t(\varphi)&=\int_0^\infty\varphi(t+\mfa)\frac{F^c(t+\mfa)}{F^c(\mfa)}\hat{\mu}_0(d\mfa)
+\int_0^t\varphi(t-\mfa)F^c(t-\mfa)\left[\widehat{\Upsilon}(\mfa)d\mfa+\sqrt{\overline{\Upsilon}(\mfa)}W_{inf}(d\mfa)\right]\nonumber\\
& \quad + \int_0^t\int_0^\infty\varphi(t-s+\mfa)\frac{F^c(t-s+\mfa)}{F^c(\mfa)}W_{rec}(ds,d\mfa)\,.
\end{align}

In order to check that this expression for $\hat{\mu}_t(\varphi)$ satisfies equation \eqref{SPDE-CLT}, we differentiate the right hand side of \eqref{eqn-SPDE-solution} w.r.t. $t$, and then integrate on the interval $[0,t]$. The differentiation gives three types of terms. The terms involving the derivative of $\varphi$ give
$\hat{\mu}_t(\varphi')$,
the differentiation of $F^c$ produces  $-\hat{\mu}_t(h\varphi)$, and finally we differentiate w.r.t. the upper bounds of the three integrals, and this produces exactly what is expected, hence the result.
The informal ``strong'' formulation of this SPDE reads:
\begin{equation} \label{eqn-SPDE-strong}
\begin{split}
\partial_t\hat{\mu}_t+\partial_\mfa\hat{\mu}_t&=
-h\hat{\mu}_t+\frac{\partial^2}{\partial t\partial\mfa}W_{rec}(t,\cdot)\,, \\
\hat{\mu}_0 \text{ given by Proposition \ref{prop-initCLT}},&\ \hat{\mu}(t,0)= \widehat{\Upsilon}(t)+\sqrt{\overline{\Upsilon}(t)}\frac{dW_{inf}(t)}{dt}\,.
\end{split}
\end{equation}
This means that $\hat{\mu}$ is a distribution on $\R^2_+$, which satisfies the first line of \eqref{eqn-SPDE-strong} in the sense of distributions in $(0,+\infty)^2$,
 and has traces on the boundaries $t=0$ and $\mfa=0$ specified by the second line of \eqref{eqn-SPDE-strong}. The rigorous meaning of this  is formulated in \eqref{SPDE-CLT}.

Note that $(\hat{\mu}_0,\widehat{\Upsilon}, W_{inf}, W_{rec})$ are jointly Gaussian and $(\hat{\mu}_0,W_{inf}, W_{rec})$ are mutually independent ($\widehat{\Upsilon}$ is given in Theorem \ref{thm-FCLT-SF} below). 

\medskip

We state the following theorem on the convergence of $\hat\mu^N_t$, given the convergence of  $\widehat{\Upsilon}^N \Rightarrow \widehat{\Upsilon}$ in $D$, since the latter has been proved in \cite{PP2020-FCLT-VI} (under a slightly different initial condition). We will state the convergence of the processes $(\hat{S}^N, \widehat{\mathfrak{F}}^N,\widehat\Upsilon^N)$ in Theorem \ref{thm-FCLT-SF} below and that of $(\hat{I}^N, \hat{R}^N)$ as a corollary.  The proof of the following theorem is given in Section \ref{sec-wh-muN-proof}. 

\medskip

\begin{theorem} \label{thm-FCLT-mu}
Given the convergence of $\widehat{\Upsilon}^N \Rightarrow \widehat{\Upsilon}$ in $D$, 
under Assumption \ref{AS-g}, if $F\in C^1$ and $F(\mfa)>0$ for all $\mfa>0$, then
\begin{equation} \label{eqn-wh-muN-conv}
\{\hat\mu^N_t\}_{t\ge 0} \Rightarrow \{\hat\mu_t\}_{t\ge 0} \qinq D(\RR_+,  (H^1(\RR_+)' ) \qasq N \to \infty\,,
\end{equation}
where
 $\hat{\mu}_t$ is specified by  \eqref{eqn-SPDE-solution}, and $\hat{\mu} \in C(\RR_+,  (H^1(\RR_+)' ) $ a.s.
 If,  moreover, $F$ has two derivatives $f$ and $f'$ which are locally bounded, then $\hat{\mu}_t$ given by \eqref{eqn-SPDE-solution} is the unique solution 
of the SPDE  \eqref{SPDE-CLT} satisfying $\hat{\mu} \in C(\RR_+,  (H^1(\RR_+)' ) $ a.s. 
\end{theorem} 


\begin{remark}
The convergence $\widehat{\Upsilon}^N \Rightarrow \widehat{\Upsilon}$ in $D$ will be established below in Section \ref{sec-SF-proof}. 
That proof will rely upon Assumption \ref{AS-lambda}, which is stated below. Hence it turns out that the above theorem is in fact proved 
under that additional assumption. 
\end{remark}

\begin{remark}
In the solution to the SPDE in \eqref{eqn-SPDE-solution}, we define the process $\hat\mu^0_t(\varphi)$: for $t\ge 0$ and $\varphi\in L^2(\R_+)$, 
\begin{align} \label{eqn-hatmu0-def}
\hat\mu^0_t(\varphi)&= \int_0^\infty \varphi(\mfa+t) \frac{F^c(t+\mfa)}{F^c(\mfa)}  \hat\mu_0(d\mfa) \,.
\end{align}
By the representation of $\hat\mu_0$ using the Wiener process $W$ in \eqref{eqn-hat-mu0-Wiener}, we obtain 
\begin{align*}
\hat\mu^0_t(\varphi)
&=  \bar{\mu}_0(\mds)^{1/2} \bigg( \int_0^\infty \varphi(\mfa+t) \frac{F^c(t+\mfa)}{F^c(\mfa)}  \sqrt{g_0(\mfa)} d W(\mfa)     \\
& \qquad \qquad \qquad -   \int_0^\infty \varphi(\mfa+t) \frac{F^c(t+\mfa)}{F^c(\mfa)}   g_0(\mfa) d\mfa   \int_0^\infty \sqrt{g_0(a)} d W(a) \bigg)  \,.
\end{align*}

By well--known properties of the Wiener integral, the process $\{\hat\mu^0_t(\varphi), t\ge0, \varphi\in L^2(\R_+)\}$ is a generalized Gaussian process, with mean zero, and covariance function: for $t, t'\ge 0$ and $\varphi, \psi \in L^2(\R_+)$, 
\begin{align*}
\Cov( \hat\mu^0_t(\varphi), \hat\mu^0_{t'}(\varphi))
&=  \bar{\mu}_0(\mds) \bigg( \int_0^\infty \varphi(\mfa+t)  \psi(\mfa+t') \frac{F^c(t+\mfa)}{F^c(\mfa)} \frac{F^c(t'+\mfa)}{F^c(\mfa)} g_0(\mfa) d\mfa \\
 & \qquad \qquad  -    \int_0^\infty \varphi(\mfa+t) \frac{F^c(t+\mfa)}{F^c(\mfa)}   g_0(\mfa) d\mfa \int_0^\infty \psi(\mfa+t') \frac{F^c(t'+\mfa)}{F^c(\mfa)}  g_0(\mfa) d \mfa \bigg) \,. 
\end{align*} 
In particular,  the variance of the process is given by 
 \begin{align*}
\Var(\hat\mu^0_t(\varphi))  
& = \bar{\mu}_0(\mds) \bigg(   \int_0^\infty \varphi(\mfa+t)^2 \Big( \frac{F^c(t+\mfa)}{F^c(\mfa)} \Big)^2  g_0(\mfa) d \mfa  -  \Big(   \int_0^\infty \varphi(\mfa+t) \frac{F^c(t+\mfa)}{F^c(\mfa)} g_0(\mfa) d\mfa   \Big)^2 \bigg) \,. 
\end{align*}
\end{remark} 

\begin{remark} \label{rem-mu-inf-rec}
Denote the last two terms of  the SPDE solution  $\hat{\mu}_t(\varphi)$ in \eqref{eqn-SPDE-solution}:
\begin{align}
\check{\mu}^{inf}_t(\varphi) &= \int_0^t\varphi(t-\mfa)F^c(t-\mfa) \sqrt{\overline{\Upsilon}(\mfa)}W_{inf}(d\mfa)\,, \label{eqn-check-mu-inf}\\
\check{\mu}^{rec}_t(\varphi) &=\int_0^t\int_0^\infty\varphi(t-s+\mfa)\frac{F^c(t-s+\mfa)}{F^c(\mfa)}W_{rec}(ds,d\mfa)\,. \label{eqn-check-mu-rec}
\end{align}
Note that they are different from $\hat{\mu}^{inf}_t(\varphi) $ and $\hat{\mu}^{rec}_t(\varphi) $ in \eqref{eqn-hat-mu-rep-inf} and \eqref{eqn-hat-mu-rep-rec} in the SPDE  \eqref{eqn-SPDE-hat}.
It is easy to calculate the covariances of these two terms: for $t, t'\ge 0$ and $\varphi, \psi \in L^2(\R_+)$, 
\begin{align*}
\Cov\big(\check{\mu}^{inf}_t(\varphi),\,\check{\mu}^{inf}_{t'}(\psi) \big) &= \int_0^{t\wedge t'}\varphi(t-\mfa)\psi(t'-\mfa) F^c(t-\mfa)F^c(t'-\mfa) \overline{\Upsilon}(\mfa)d\mfa\,,
\end{align*}
and
\begin{align*}
\Cov\big(\check{\mu}^{rec}_t(\varphi) ,\check{\mu}^{rec}_{t'}(\psi) \big) &=  \int_0^{t\wedge t'} \int_0^\infty\varphi(t-s+\mfa)\psi(t'-s+\mfa)  \frac{F^c(t-s+\mfa)}{F^c(\mfa)} \frac{F^c(t'-s+\mfa)}{F^c(\mfa)}  h(\mfa)\bar{\mu}_s(d\mfa)ds \,.
\end{align*}
In particular,
\begin{align}\label{eqn-check-mu-inf-var}
\Var\big(\check{\mu}^{inf}_t(\varphi)\big) = \int_0^t\varphi(t-\mfa)^2 (F^c(t-\mfa))^2 \overline{\Upsilon}(\mfa)d\mfa\,,
\end{align}
and
\begin{align}\label{eqn-check-mu-rec-var}
\Var\big(\check{\mu}^{rec}_t(\varphi)\big) = \int_0^t\int_0^\infty\varphi(t-s+\mfa)^2 \Big( \frac{F^c(t-s+\mfa)}{F^c(\mfa)}\Big)^2  h(\mfa)\bar{\mu}_s(d\mfa)ds \,.
\end{align}
\end{remark}

Before proceeding to specify the limits $(\hat{S}, \widehat{\mathfrak{F}},\widehat\Upsilon)$, we give the following definitions of the driving Gaussian processes.

\begin{definition} \label{def-Gaussian-mu0}
We define the following process $\widehat{\mfF}_{0,1} (t)$:
\begin{align*}
\widehat{\mfF}_{0,1} (t) & :=  \int_0^{\infty} \bar\lambda(\mfa+t)  d \hat{\mu}_0(\mfa) \,.
\end{align*}
By well--known properties of the Wiener integral, the process $\widehat{\mfF}_{0,1} (t) $ is a  Gaussian process (continuous in probability) with mean zero and covariance function: for $t,t'\ge 0$,
\begin{align*}
\Cov\big( \widehat{\mfF}_{0,1} (t) , \widehat{\mfF}_{0,1} (t') \big)
&= \int_0^\infty\bar\lambda(\mfa+t) \bar\lambda(\mfa+t') \bar{\mu}_0(d\mfa)  
 -\int_0^\infty\bar\lambda(\mfa+t)  \bar\mu_0(d\mfa) \int_0^\infty\bar\lambda(\mfa+t')  \bar\mu_0(d\mfa) \,.
\end{align*}

\end{definition}


\begin{definition} \label{def-Gaussian}
Define the following centered Gaussian processes:
\begin{align*}
\widehat{S}_1(t) &:= \int_0^t  \overline{\Upsilon}(s)^{1/2}  W_{inf}(ds) \,, \\
\widehat{\mfF}_1(t) &= \int_0^t \bar\lambda(t-s) \overline{\Upsilon}(s)^{1/2}  W_{inf}(ds) \,.
\end{align*}
\end{definition}

\begin{definition} \label{def-Gaussian2}
We define the continuous Gaussian process
$ \widehat{\mfF}_{0,2}$, 
independent
 of the Gaussian random field $W_{inf}$),  with mean zero and covariance function: 
 for $t,t'\ge 0$,
\begin{align*}
\Cov\big(\widehat{\mfF}_{0,2}(t), \widehat{\mfF}_{0,2}(t)\big) &= \int_0^{\infty} v(\mfa+t, \mfa+t')   \bar{\mu}_0(d\mfa)\,.
\end{align*}
We define another continuous Gaussian process
$ \widehat{\mfF}_{2}$, independent of $ \hat{\mu}_0(\cdot)$  (hence, $ \widehat{\mfF}_{0,1}$) and of $ \widehat{\mfF}_{0,2}$, as well as of $W_{inf}$ and $W_{rec}$,  with mean zero and covariance function: 
 for $t,t'\ge 0$,
\begin{align*}
\Cov\big(\widehat{\mfF}_{2}(t), \widehat{\mfF}_{2}(t')\big) &=  \int_0^{t\wedge t'}  v(t-s, t'-s) \overline{\Upsilon}(s)ds\,. 
\end{align*}
Observe that the PRMs $Q_{\inf}$ and $Q_{rec}$ in \eqref{muN-eq} by construction are independent of the random infectivity functions $\{\lambda_i(\cdot)\}_{i \in \ZZ\setminus\{0\}}$ and  hence the infection durations $\eta_i$ (see the expressions of the corresponding scaled processes $ \widehat{\mfF}^N_{0,2}$ and $ \widehat{\mfF}^N_{2}$  in \eqref{eqn-whFN02} and \eqref{eqn-whFN2}). Therefore we have the independence of their corresponding limits $W_{inf}$ and $W_{rec}$ from the limits $ \widehat{\mfF}_{0,2}$ and $ \widehat{\mfF}_{2}$, which capture the randomness of  $\{\lambda_i(\cdot)\}$. 
\end{definition}

\begin{remark}  \label{rem-hatmu0-mfF}
We discuss the correlations between $ \hat\mu_0$ and the processes $\widehat\mfF_{0,1}$ and $\widehat\mfF_{0,2}$ associated with the initially infected individuals. Recall that
$\widetilde\tau_{j,0} = U_j\eta^0_j\wedge\bar\mfa$, where $\eta^0_j=G(\lambda_{-j})$, for a certain mesurable function $G:D\mapsto\R_+$,
and $U_j\simeq\mathcal{U}(0,1)$ is independent of $\lambda_{-j}$. 
This specifies completely the joint law of $\lambda_{-j}$ and $\widetilde\tau_{j,0}$. 

First, $ \hat\mu_0$ and $\widehat\mfF_{0,1}$ have the following covariance function: for $t, t'\ge 0$,
\begin{align*}
\Cov(\hat\mu_0(0,t],\widehat\mfF_{0,1}(t'))&=\bar I(0)\, \Cov({\bf1}_{(0,t]}(\widetilde\tau_{j,0}),\bar\lambda(\widetilde\tau_{j,0}+t')),\\
&=\bar I(0) \, \bigg(   \int_0^t\bar\lambda(t'+\mfa)\bar{\bar{\mu}}_0(d\mfa)-
\bar{\bar{\mu}}_0((0,t])\int_0^{\bar\mfa}\bar\lambda(t'+\mfa)\bar{\bar{\mu}}_0(d\mfa)\bigg). 
\end{align*}
Next, $\hat{\mu}_0$ and $ \widehat{\mfF}_{0,2}$ have the following covariance function:  for $t,t'\ge 0$, 
\begin{align*}
\Cov\big( \hat\mu^0(\varphi), \widehat{\mfF}_{0,2}(t)\big) 
&= \bar I(0)\, \Cov({\bf1}_{(0,t]}(\widetilde\tau_{j,0}),\lambda_{-j}(\widetilde\tau_{j,0}+t'))\,,
\end{align*} 
and   $ \widehat{\mfF}_{0,1}$ and $ \widehat{\mfF}_{0,2}$ have the following covariance function:  for $t,t'\ge 0$, 
\begin{align*}
\Cov(\widehat\mfF_{0,1}(t),\widehat\mfF_{0,2}(t'))&=\bar I(0) \, \Cov(\bar\lambda(\widetilde\tau_{j,0}+t),\lambda_{-j}(\widetilde\tau_{j,0}+t'))\,.
\end{align*}
Given the joint law of $\lambda_{-j}$ and $\widetilde\tau_{j,0}$, the above two covariances are well defined, although they are not easily calculated explicitly. 
We do not go further into their computation here. 
\end{remark}

\begin{remark}\label{rem-Wrec-mfF}
We next characterize the covariances between $W_{rec}$ and $\widehat\mu_0$, $\widehat\mfF_{0,1}(t)$, $\widehat\mfF_{0,2}(t)$, $\widehat\mfF_{1}$, $\widehat\mfF_{2}$. 
Observe that $W^\varphi_{rec}(t):=\int_0^t\int_0^\infty\varphi(\mfa)W_{rec}(ds,d\mfa)$ is the limit of the scaled process which is the last term in the expression of $\hat{\mu}^{N}_t(\varphi)$ in \eqref{eqn-SPDE-prelimit}. 
Recall the covariance function of $W^\varphi_{rec}(t)$ is given in \eqref{eqn-Wrec-cov}. It is easy to check that 
\begin{align*}
\frac{1}{\sqrt{N}}\int_0^t&\int_0^\infty\int_0^1\varphi(H(\mu^N_{s^-},w)){\bf1}_{v\le\mu^N_{s^-}(h)}\overline{Q}_{rec}(ds,dv,dw)\\
&=
\frac{1}{\sqrt{N}}\left(\sum_{j=1}^{I^N(0)}\varphi(\tilde\tau_{j,0}+\eta^0_j){\bf1}_{\eta^0_j\le t}+\sum_{i=1}^{A^N(t)}\varphi(\eta_i){\bf1}_{\tau^N_i+\eta_i\le t}
-\int_0^t\mu^N_s(h\varphi)ds\right). 
\end{align*}
One can then derive the following covariance functions: for $t, \mfa\ge0$,
\begin{align*}
\Cov(W^\varphi_{rec}(t),\widehat{\mu}_0(0,\mfa])
&=\bar{I}(0)\, Cov\Big({\bf1}_{(0,\mfa]}(\tilde\tau_{j,0}),\varphi(\eta_{-j}){\bf1}_{\eta^0_{j}\le t}-\int_0^{\eta^0_j\wedge t}\varphi(\tilde\tau_{j,0}+s)h(\tilde\tau_{j,0}+s)ds\Big)\,, 
\end{align*}
and for $t,t'\ge 0$, 
\begin{align*}
\Cov(W^\varphi_{rec}(t),\widehat{S}_1(t'))&=0\,,\\
\Cov(W^\varphi_{rec}(t),\widehat\mfF^N_{0,1}(t'))&=\bar{I}(0)\, \Cov\Big(\varphi(\eta_{-j}){\bf1}_{\eta^0_{j}\le t}-\int_0^{\eta^0_j\wedge t}\varphi(\tilde\tau_{j,0}+s)h(\tilde\tau_{j,0}+s)ds,
\bar\lambda(\tilde\tau_{j,0}+t')\Big)\,, \\
\Cov(W^\varphi_{rec}(t),\widehat\mfF^N_{0,2}(t'))&=\bar{I}(0)\, \Cov\Big(\varphi(\eta_{-j}){\bf1}_{\eta^0_{j}\le t} -\int_0^{\eta^0_j\wedge t}\varphi(\tilde\tau_{j,0}+s)h(\tilde\tau_{j,0}+s)ds,(\lambda_{-j}-\bar\lambda)(\tilde\tau_{j,0}+t') \Big) \,,\\
\Cov(W^\varphi_{rec}(t),\widehat\mfF^N_{1}(t'))&= 0\,,\\
\Cov(W^\varphi_{rec}(t),\widehat\mfF^N_{2}(t'))&=\int_0^t\overline\Upsilon(s) \, \Cov\Big(\varphi(\eta_i){\bf1}_{s+\eta_i\le t}
-\int_0^{(s+\eta_i)\wedge t}\varphi(r-s)h(r-s)dr,
\lambda_i(t'-s)\Big)ds\,. 
\end{align*}
\end{remark}

\begin{assumption} \label{AS-lambda}
Let $\lambda(\cdot)$ be a process having the same law of $\{\lambda_{-j}(\cdot)\}_{j=1,\dots, I^N(0)}$ and $\{\lambda_i(\cdot)\}_{i\in \NN}$. 
Assume that there exists a constant $\lambda^*$ such that for each $0<T<\infty$, 
$\sup_{t\in [0,T]} \lambda(t) \le \lambda^*$  almost surely.
Assume that there exist an integer $k$, a random sequence  $0=\zeta^0 \le \zeta^1 \le \cdots \le \zeta^k $ and associated random functions $\lambda^\ell \in C(\RR_+;[0,\lambda^\ast])$, $1\le\ell \le k$, such that 
\begin{align} \label{eqn-lambda-assump}
\lambda(t) = \sum_{\ell=1}^k \lambda^\ell(t) \bone_{[\zeta^{\ell-1},\zeta^\ell)}(t).
\end{align}
We assume moreover that there exists a deterministic nondecreasing function $\phi \in C(\RR_+;\RR_+)$ with $\phi(0)=0$ such that $|\lambda^\ell(t) - \lambda^\ell(s)| \le \phi(t-s)$ 
almost surely for all $t,s \ge 0$ and for all $\ell\ge 1$.

In addition to the conditions on $\lambda(t)$ above,
 the function $\phi$  satisfies that for some $\alpha >1/4$,
\begin{equation}\label{eqn-lambda-inc}
\phi(t)\le C t^\alpha\,.
\end{equation} 
Also, if $F_\ell$ denotes the c.d.f. of the r.v. $\zeta^\ell$, there exist $C'$ and  
 $\rho>0$ such that for any $1\le \ell\le k-1$, $0\le s<t$,
 \begin{equation}\label{hypF}
 F_\ell(t)-F_\ell(s)\le C'(t-s)^\rho\,,
 \end{equation}
 and in addition, for any  $1\le \ell\le l-1$, $r>0$,
 \begin{equation}\label{increments}
 \P(\zeta^{\ell}-\zeta^{\ell-1}\le r|\zeta^{\ell-1})\le C' r^\rho\,.
 \end{equation}
\end{assumption}

%

\medskip

\begin{theorem} \label{thm-FCLT-SF}
Under Assumptions \ref{AS-g} and \ref{AS-lambda}, 
\begin{align} \label{eqn-conv-hatSsI}
(\widehat{S}^N, \widehat{\mfF}^N) \RA (\widehat{S}, \widehat{\mfF}) \qinq D^2 \qasq N \to \infty,
\end{align}
where the limit processes $(\widehat{S}, \widehat{\mfF})$ are the unique  solution to the following stochastic integral equations driven by the continuous Gaussian processes $\widehat{\mu}_0$, $ \widehat{\mfF}_{0,1}$, $ \widehat{\mfF}_{0,2}$, $\widehat{\mfF}_{1}$ and $ \widehat{\mfF}_{2}$: 
\begin{align}
\widehat{S}(t) &=   - \widehat{S}_1(t) - 
\int_0^t \widehat{\Upsilon}(s) ds, \label{eqn-hat-S}\\
\widehat{\mfF}(t) & =
 \int_0^t \bar\lambda(t-s) \widehat{\Upsilon}(s) ds +  \widehat{\mfF}_{0,1}(t)  + \widehat{\mfF}_{0,2}(t) + \widehat{\mfF}_{1}(t) + \widehat{\mfF}_{2}(t)\,,  \label{eqn-hat-mfF}\\
\widehat{\Upsilon}(t) &= \widehat{S}(t) \overline{\mfF}(t) + \bar{S}(t) \widehat{\mfF}(t) \,, \label{eqn-hat-Upsilon} 
\end{align}
where 
$\bar{S}$ and $ \overline{\mfF}$ are the limits in \eqref{eqn-barS} and \eqref{eqn-overline-mfF}, respectively. 
The limit processes  $(\widehat{S}, \widehat{\mfF})$  are Gaussian, and continuous almost surely. 
\end{theorem}

\bigskip

\begin{coro} \label{thm-FCLT-IR}
Under Assumptions \ref{AS-g}  and \ref{AS-lambda}, and assuming that $\widehat{R}^N(0)\Rightarrow \widehat{R}(0)$ as $N\to\infty$, we have 
\begin{align} \label{eqn-conv-hatIR}
\big( \widehat{I}^N, \widehat{R}^N\big) \to \big( \widehat{I}, \widehat{R}\big) \qinq D^2 \qinq N \to\infty,
\end{align}
jointly with the convergence of the processes $\big(\widehat{S}^N, \widehat{\mfF}^N\big)$ in \eqref{eqn-conv-hatSsI}  and  $ \widehat{\mu}^N_t $ in \eqref{eqn-wh-muN-conv},  
where 
\begin{align}
\widehat{I}(t) &=   \int_0^{\infty}  \frac{F^c(t+\mfa)}{F^c(\mfa)}  d \hat{\mu}_0(\mfa) 
+ \int_0^t F^c(t-s) \widehat{\Upsilon}(s) ds +  \widehat{I}_{inf}(t) + \widehat{I}_{rec}(t)\,, \label{eqn-wh-I}   \\
\widehat{R}(t) &= \widehat{R}(0)+   \int_0^{\infty} \Big(1-  \frac{F^c(t+\mfa)}{F^c(\mfa)} \Big)  d \hat{\mu}_0(\mfa) 
+ \int_0^t F(t-s) \widehat{\Upsilon}(s) ds +  \widehat{R}_{0}(t) + \widehat{R}_{1}(t)\,. \label{eqn-wh-R}
\end{align}
Here 
$ \widehat{I}_{inf}(t) =\check{\mu}^{inf}_t(\mds) $ and $ \widehat{I}_{rec}(t)=\check{\mu}^{rec}_t(\mds)$ are independent continuous Gaussian processes, which have covariance functions: for $t, t'\ge 0$,  
\begin{align*}
\Cov\big(\widehat{I}_{inf}(t),\,\widehat{I}_{inf}(t') \big) &= \int_0^{t\wedge t'} F^c(t-\mfa)F^c(t'-\mfa) \overline{\Upsilon}(\mfa)d\mfa\,,
\end{align*}
and
\begin{align*}
\Cov\big(\widehat{I}_{rec}(t) ,\widehat{I}_{rec}(t')\big) &=  \int_0^{t\wedge t'} \int_0^\infty \frac{F^c(t-s+\mfa)}{F^c(\mfa)} \frac{F^c(t'-s+\mfa)}{F^c(\mfa)}  h(\mfa)\bar{\mu}_s(d\mfa)ds \,,
\end{align*}
where $\bar{\mu}_s(d\mfa)$ is the LLN limit appearing in Theorem \ref{thm-FLLN}. 
$ \widehat{R}_{0}(t)$ and $\widehat{R}_{1}(t)$ are independent Gaussian processes with covariance functions: for $t, t'\ge 0$,  
 \begin{equation} \label{eqn-wh-R0-cov}
\Cov\big(\widehat{R}_{0}(t), \widehat{R}_{0}(t')\big) = \int_0^\infty \Big(\frac{F(t\wedge t'+\mfa) - F(\mfa)}{F^c(\mfa)} - \frac{F(t+\mfa) - F(\mfa)}{F^c(\mfa)}\frac{F(t'+\mfa) - F(\mfa)}{F^c(\mfa)} \Big) \bar\mu_0(d\mfa) \,,
\end{equation}
and
\begin{equation} \label{eqn-wh-R1-cov}
\Cov\big(\widehat{R}_{1}(t), \widehat{R}_{1}(t')\big)  = \int_0^{t\wedge t'} F(t\wedge t'-s) \overline\Upsilon(s) ds \,. 
\end{equation}
\end{coro}

\begin{remark} \label{rem-wh-I}
We remark that we have taken a different approach from the previous work  \cite{PP2020-FCLT-VI} to derive the limit $\widehat{I}(t)$, by exploiting the measure-valued process $\hat\mu_t$ in \eqref{eqn-SPDE-solution}. However, by extending the analysis in \cite[Theorem 2.4]{PP2020-FCLT-VI} to take into account the different initial conditions, we obtain the following representation for the limit $\widehat{I}(t)$:
\begin{align}  \label{eqn-wh-I-01} 
\widehat{I}(t) &=   \int_0^{\infty}  \frac{F^c(t+\mfa)}{F^c(\mfa)}  d \hat{\mu}_0(\mfa) 
+ \int_0^t F^c(t-s) \widehat{\Upsilon}(s) ds +  \widehat{I}_{0}(t) + \widehat{I}_{1}(t)\,,
\end{align}
where $\widehat{I}_{0}(t)$ and $ \widehat{I}_{1}(t)$ are independent continuous Gaussian processes with covariance functions:
 for $t, t'\ge 0$,  
\begin{equation} \label{eqn-wh-I0-cov}
\Cov\big(\widehat{I}_{0}(t), \widehat{I}_{0}(t')\big) = 
\int_0^\infty \Big(\frac{F^c(t\vee t'+\mfa)}{F^c(\mfa)} - \frac{F^c(t+\mfa)}{F^c(\mfa)}\frac{F^c(t'+\mfa)}{F^c(\mfa)} \Big)\bar\mu_0(d\mfa) \,,
\end{equation}
and
\begin{equation} \label{eqn-wh-I1-cov}
\Cov\big(\widehat{I}_{1}(t), \widehat{I}_{1}(t')\big )  = \int_0^{t\wedge t'} F^c(t\vee t'-s) \overline\Upsilon(s) ds \,. 
\end{equation}
It can shown that the driving Gaussian processes $\widehat{I}_{inf}(t) + \widehat{I}_{rec}(t)$ and $\widehat{I}_{0}(t) + \widehat{I}_{1}(t)$ have the same law. See the relevant discussions in Section \ref{sec-IR-proof}. 
\end{remark} 



\begin{remark}
We remark that the stochastic epidemic model has essentially three mutually independent sources of randomness: (i) the generation process of new infections, embedded in the Gaussian white noise $W_{inf}$, and also $\widehat{S}_1$ and $\widehat{\mfF}_1$; (ii)  the randomness from 
the random varying infectivity functions $\{\lambda_i(\cdot)\}_{i \in \NN}$ for the newly infected individuals (embedded in $ \widehat{\mfF}_{2}$); 
and (iii) the randomness from  the random varying infectivity functions  $\{\lambda_{-j}(\cdot)\}_{j \in \NN}$ for the initially infected individuals (embedded in $ \widehat{\mfF}_{0,2}$) and  the infection ages of the initially infected individuals (embedded in $ \widehat{\mu}_0(\cdot)$). 
However, we stress that the  generalized Gaussian random field $W_{rec}$ contains both the randomness from (ii) and (iii), particularly, the infection duration variables $\{\eta_i\}_{i\in \NN}$ and $\{\eta^0_j\}_{j \in \NN}$ (see the covariances in Remarks \ref{rem-hatmu0-mfF} and \ref{rem-Wrec-mfF}).
\end{remark}


\medskip

\section{Proof of Theorem \ref{thm-FLLN}} \label{sec-FLLN-proof}
In this section we prove the FLLN in Theorem \ref{thm-FLLN} using the new representation of $\bar{\mu}^N$ in \eqref{eq-mubarN}, and provide a new proof for the uniqueness of the solution to the PDE in  \eqref{eqn-PDE}. 

\begin{proof}[Proof of Theorem  \ref{thm-FLLN}] Most of the Theorem is contained in Theorem 2.1 from \cite{PP2021-PDE}. The convergence $\big(\bar{S}^N,  \overline{\mfF}^N\big) \to \big(\bar{S}, \overline{\mfF}\big)$ is proved under the condition that $\lambda(t) \le \lambda^*$ in \cite{FPP2022-MPMG} without requiring any regularity conditions in Assumption \ref{AS-lambda}.  We take that as given. 
Thus, we also have the convergence 
\begin{align} \label{eqn-UpsilonN-conv} 
\overline\Upsilon^N\to \overline\Upsilon \qinq D 
\end{align}
in probability as $N\to \infty$, where $\overline\Upsilon(t) = \bar{S}(t) \overline{\mfF}(t)$, $t\ge 0$.

We only need to prove the convergence of the measure-valued process 
$\bar{\mu}^N$, and the statement concerning equation \eqref{eqn-PDE}. 

To prove the convergence of $\bar{\mu}^N$ in $D(\RR_+, \sM_{F}(\RR_+))$, it suffices to show the convergence of 
$\bar{\mu}^N(\varphi)$, for any $\varphi\in C_b(\R_+)$, by \cite[Theorem~5.2]{mitoma1983tightness} (in fact, only invoking tightness criterion in  \cite[Theorem~4.1]{mitoma1983tightness}).

We have noted above that the solution of the PDE \eqref{genPDE-weak} is given by the formula \eqref{genexplicit-weak}. We first derive an analogous formula for 
$\bar{\mu}^N_t(\varphi)$. Differentiating \eqref{eq-mubarN} with respect to $t$, we obtain for $\varphi \in C_b(\R_+)$, 
\begin{align*}
\frac{d}{dt}\bar{\mu}_t^N(\varphi)&=\bar{\mu}^N_t(\varphi'-h\varphi)+\varphi(0)\left[\bar{\Upsilon}^N(t)+\frac{1}{N}\frac{d}{dt}\int_0^t\int_0^\infty{\bf1}_{v\le\Upsilon^N(s^-)}\bar{Q}_{inf}(ds,dv)\right]\\
&\quad-\frac{1}{N}\frac{d}{dt}\int_0^t\int_0^\infty\int_0^1\varphi(H(\mu^N_{s^-},w)){\bf1}_{v\le\mu^N_{s^-}(h)}\bar{Q}_{rec}(ds,dv,dw)\,.
\end{align*}

We see that $\bar{\mu}$ satisfies an equation of the type \eqref{genPDE-weak}, but with  $u_0(\mfa)d\mfa$ replaced by $\bar{\mu}_0^N(d\mfa)$, and also
\begin{align*}
\text{the function } k(t) \text{ replaced by }&\bar{\Upsilon}^N(t)+\frac{1}{N}\frac{d}{dt}\int_0^t\int_0^\infty{\bf1}_{v\le\Upsilon^N(s^-)}\bar{Q}_{inf}(ds,dv),\\
\text{and the function } g(t,\mfa) \text{ replaced by }& -\frac{1}{N}\frac{d}{dt}\int_0^t\int_0^\infty\int_0^1\delta_{H(\mu^N_{s^-},w)}(d\mfa){\bf1}_{v\le\mu^N_{s^-}(h)}\bar{Q}_{rec}(ds,dv,dw)\,.
\end{align*}

Hence it is not hard to show that for $\varphi \in C_b(\R_+)$,
\begin{align}\label{explicit-weak-N}
\bar\mu^N_t(\varphi)
&=\int_0^\infty \varphi(\mfa+t)\frac{F^c(t+\mfa)}{F^c(\mfa)}\bar\mu^N_0(d\mfa)+\int_0^t \varphi(t-\mfa)F^c(t-\mfa)\bar{\Upsilon}^N(\mfa)d\mfa\nonumber\\
&\quad+ \frac{1}{N}\int_0^t\int_0^\infty \varphi(t-\mfa)F^c(t-\mfa) {\bf1}_{v\le\Upsilon^N(\mfa^+)}\bar{Q}_{inf}(d\mfa,dv) \\
&\quad-\frac{1}{N}\int_0^t\int_{0}^\infty\int_0^1 \varphi(t-s+H(\mu^N_{s^-},v))\frac{F^c(t-s+H(\mu^N_{s^-},v))}{F^c(H(\mu^N_{s^-},w))}{\bf1}_{v\le\mu^N_{s^-}(h)}\bar{Q}_{rec}(ds,dv,dw)\,.\nonumber
\end{align}

Theorem \ref{thm-FLLN} will follow from the two next Lemmas, the first one says that for any $\varphi\in C_b(\R_+)$, $\bar\mu^N_t(\varphi)\to\bar\mu_t(\varphi)$ in probability as $N\to\infty$, for each fixed $t\ge0$, and the second one that $\bar\mu^N(\varphi)$ is tight in $D(\R_+)$, again for any $\varphi\in C_b(\R_+)$.
Both proofs exploit the formula \eqref{explicit-weak-N}.
\end{proof}

\medskip

\begin{lemma}\label{LLNfixedt}
Assume that the hazard function $h$ is bounded. Then,
as $N\to\infty$, $\bar\mu^N_t(\varphi)\to\bar\mu_t(\varphi)$ in probability, for each $t\ge0$ and $\varphi\in C_b(\R_+)$, where 
\begin{equation}\label{explicit-weak}
\begin{split}
\bar\mu_t(\varphi)=\int_0^\infty \varphi(\mfa+t)\frac{F^c(t+\mfa)}{F^c(\mfa)}\bar\mu_0(d\mfa)+\int_0^t \varphi(t-\mfa)F^c(t-\mfa)\bar{\Upsilon}(\mfa)d\mfa\,.
\end{split}
\end{equation}
\end{lemma}
\begin{proof}
For that sake, we consider each term of the right hand side of \eqref{explicit-weak-N}. 
The convergence of the first term follows readily from the fact that a.s. $\bar{\mu}^N_0\Rightarrow\bar{\mu}_0$ weakly,
and for each $t\ge0$, 
\[\mfa\mapsto \varphi(\mfa+t)\frac{F^c(t+\mfa)}{F^c(\mfa)}\]
is continuous, and  bounded by $\|\varphi\|_\infty$. The convergence of the second term in probability follows from that 
of $\bar{\Upsilon}^N$ to $\bar{\Upsilon}$ in $D$. Finally, the last two terms tend to $0$ in mean square. Indeed, we have
\begin{align*}
\E\left(\left|\frac{1}{N}\int_0^t\int_0^\infty \varphi(t-\mfa)F^c(t-\mfa) {\bf1}_{v\le\Upsilon^N(\mfa^+)}\bar{Q}_{inf}(d\mfa,dv)\right|^2\right)
&=\frac{1}{N}\E\int_0^t\left|\varphi(t-\mfa)F^c(t-\mfa)\right|^2\bar{\Upsilon}^N(\mfa)d\mfa\\
&\le \frac{t\|\varphi\|^2\lambda^\ast}{N},
\end{align*}
which tends to $0$ as $N\to\infty$. Finally, we have 
\begin{align*}
\E&\left(\left|\frac{1}{N}\int_0^t\int_{0}^\infty\int_0^1 \varphi(t-s+H(\mu^N_{s^-},w))\frac{F^c(t-s+H(\mu^N_{s^-},w))}{F^c(H(\mu^N_{s^-},w))}{\bf1}_{v\le\mu^N_{s^-}(h)}\bar{Q}_{rec}(ds,dv,dw)\right|^2\right)\\
&\quad=\frac{1}{N}\E\int_0^t\bar{\mu}^N_s\left(h\varphi(t-s+\cdot)\frac{F^c(t-s+\cdot)}{F^c(\cdot)}\right)ds\\
&\le\frac{t\|h\|_\infty\|\varphi\|_\infty}{N},
\end{align*}
which tends to $0$ as $N\to\infty$. 
\end{proof}

\medskip
We next establish the following tightness result. 
\begin{lemma}\label{LLNtight}
Assume that $h$ is bounded, and that the measure
$\bar{\mu}_0$ has compact support, and that $F^c(\mfa)>0$ for all $\mfa>0$. Then
for each $\varphi\in C_b(\R_+)$, the sequence $\{\bar{\mu}^N_t(\varphi),\ t\ge0\}_{N\ge1}$
is tight in $D$.
\end{lemma}
\begin{proof}
We will in fact show that the sequence $\{\bar{\mu}^N(\varphi)\}_{N\ge1}$ is $C$--tight. This will follow from the fact that
for any $T>0$, we have the following property: For any $\eps,\eta>0$, there exists $\delta\in(0,1)$ and $N_0$ such that
\begin{equation}\label{Ctight}
\P\left(\sup_{0\le t<t'\le T,\ t'-t\le\delta}\left|\bar\mu^N_{t'}(\varphi)-\bar\mu^N_t(\varphi)\right|\ge\eps\right)\le\eta,\quad \forall N\ge N_0\,.
\end{equation}
We first rewrite \eqref{explicit-weak-N} as follows:  for  $\varphi\in C_b(\R_+)$, 
\begin{align}\label{explicit-weak-N2}
\bar\mu^N_t(\varphi)
&=\int_0^\infty \varphi(\mfa+t)\frac{F^c(t+\mfa)}{F^c(\mfa)}\bar\mu^N_0(d\mfa)\nonumber\\
&\quad+ \frac{1}{N}\int_0^t\int_0^\infty \varphi(t-\mfa)F^c(t-\mfa) {\bf1}_{v\le\Upsilon^N(\mfa^+)}{Q}_{inf}(d\mfa,dv) \\
&\quad-\frac{1}{N}\int_0^t\int_{0}^\infty\int_0^1 \varphi(t-s+H(\mu^N_{s^-},w))\frac{F^c(t-s+H(\mu^N_{s^-},w))}{F^c(H(\mu^N_{s^-},w))}{\bf1}_{v\le\mu^N_{s^-}(h)}\bar{Q}_{rec}(ds,dv,dw)\,.\nonumber
\end{align}
We will show that each term on the right of the identity \eqref{explicit-weak-N2} satisfies the property \eqref{Ctight}.
We start with the first term. Recall that $[0,\bar\mfa]$ is the support of $\bar\mu_0$. We note that
\begin{align*}
\sup_{0\le t<t'\le T,\ t'-t\le\delta}&\left|\int_0^{\bar\mfa} \varphi(\mfa+t')\frac{F^c(t'+\mfa)}{F^c(\mfa)}\bar\mu^N_0(d\mfa)-\int_0^{\bar\mfa} \varphi(\mfa+t)\frac{F^c(t+\mfa)}{F^c(\mfa)}\bar\mu^N_0(d\mfa)\right|\\
&\le\frac{1}{F^c(\bar\mfa)}\sup_{0\le t<t'\le T+\bar\mfa,\ t'-t\le\delta}|\varphi(t') F^c(t')-\varphi(t) F^c(t)|,
\end{align*}
which tends to $0$ as $\delta\to0$, since both $\varphi$ and $F^c$ are continuous.

We next consider the second term: 
\begin{align*}
\bigg|\frac{1}{N}&\int_0^{t'}\int_0^\infty \varphi(t'-\mfa)F^c(t'-\mfa) {\bf1}_{v\le\Upsilon^N(\mfa^+)}{Q}_{inf}(d\mfa,dv)\\
&-  \frac{1}{N}\int_0^t\int_0^\infty \varphi(t-\mfa)F^c(t-\mfa) {\bf1}_{v\le\Upsilon^N(\mfa^+)}{Q}_{inf}(d\mfa,dv)\bigg|\\
&\le\frac{1}{N}\int_t^{t'}\int_0^\infty \varphi(t'-\mfa)F^c(t'-\mfa) {\bf1}_{v\le\Upsilon^N(\mfa^+)}{Q}_{inf}(d\mfa,dv)\\
&+\frac{1}{N}\int_0^t\int_0^\infty \left|\varphi(t'-\mfa)F^c(t'-\mfa) -\varphi(t-\mfa)F^c(t-\mfa)\right| {\bf1}_{v\le\Upsilon^N(\mfa^+)}{Q}_{inf}(d\mfa,dv)\,.
\end{align*}
The first term on the last right hand side is bounded by $\|\varphi\|_\infty$  multiplied by
\[\frac{1}{N}\int_t^{t'}\int_0^{N\lambda^\ast}{Q}_{inf}(d\mfa,dv)\,.\]
Note that, as $N\to\infty$, 
\[ \frac{1}{N}\int_0^{t}\int_0^{N\lambda^\ast}{Q}_{inf}(d\mfa,dv)\to \lambda^\ast t,\]
in probability, for any $t\ge0$. Moreover, for any $N$, $t\mapsto\frac{1}{N}\int_0^{t}\int_0^{N\lambda^\ast}{Q}_{inf}(d\mfa,dv)$ is increasing, and the limit is continuous. Hence from the second Dini theorem, the convergence is locally uniform in $t$.
Now we obtain
\begin{align*}
\P&\left(\sup_{0\le t<t'\le T,\ t'-t\le\delta}\frac{1}{N}\int_t^{t'}\int_0^{N\lambda^\ast}{Q}_{inf}(d\mfa,dv)>\eps\right)\\
&\quad\le\P\left(2\sup_{0\le t\le T}\left|\frac{1}{N}\int_0^{t}\int_0^{N\lambda^\ast}{Q}_{inf}(d\mfa,dv)-\lambda^\ast t\right|>\eps/2\right)+{\bf1}_{\lambda^\ast(t'-t)\ge\eps/2}\,.
\end{align*}
We choose $\delta<\eps/2\lambda^\ast$, and deduce from the above that
\[ \P\left(\sup_{0\le t<t'\le T,\ t'-t\le\delta}\frac{1}{N}\int_t^{t'}\int_0^{N\lambda^\ast}{Q}_{inf}(d\mfa,dv)>\eps\right)\to0,\]
as $N\to\infty$.

We next consider the term
\begin{align}\label{term22}
\frac{1}{N}\int_0^t\int_0^\infty \left|\varphi(t'-\mfa)F^c(t'-\mfa) -\varphi(t-\mfa)F^c(t-\mfa)\right| {\bf1}_{u\le\Upsilon^N(\mfa^+)}{Q}_{inf}(d\mfa,dv)\,.
\end{align}
Define
\[\omega_{\varphi F^c}(\delta,T)=\sup_{0\le t<t'\le T,\ t'-t\le\delta}|\varphi(t')F^c(t')-\varphi(t)F^c(t)|\,.\]
We have
\begin{align*}
\sup_{0\le t<t'\le T,\ t'-t\le\delta}\frac{1}{N}&\int_0^t\int_0^\infty \left|\varphi(t'-\mfa)F^c(t'-\mfa) -\varphi(t-\mfa)F^c(t-\mfa)\right| {\bf1}_{u\le\Upsilon^N(\mfa^+)}{Q}_{inf}(d\mfa,dv)\\
&\le \omega_{\varphi F^c}(\delta,T)\frac{Q_{inf}([0,T]\times[0,N\lambda^\ast])}{N} \,. 
\end{align*}
On the last right hand side, the first factor tends to $0$ as $\delta\to0$, while the second factor tends to $T\lambda^\ast$
a.s., as $N\to\infty$.  Hence the term in \eqref{term22} also satisfies \eqref{Ctight}.

We finally consider the last term in \eqref{explicit-weak-N2}. In other words, we need to establish the $C$-tightness of
\begin{align}\label{lastterm}
\nu^N_t(\varphi)&:=\int_0^t\bar\mu^N_s\left(h\varphi(t-s+\cdot)\frac{F^c(t-s+\cdot)}{F^c(\cdot)}\right)ds \nonumber\\
&-\frac{1}{N}\int_0^t\int_{0}^\infty\int_0^1 \varphi(t-s+H(\mu^N_{s^-},w))\frac{F^c(t-s+H(\mu^N_{s^-},w))}{F^c(H(\mu^N_{s^-},w))}{\bf1}_{v\le\mu^N_{s^-}(h)}{Q}_{rec}(ds,dv,dw) \nonumber\\
&=\nu^{N,1}_t(\varphi)+\nu^{N,2}_t(\varphi)\,. 
\end{align}

We first note that, if $t<t'$,
\begin{align*}
\nu^{N,1}_{t'}(\varphi)-\nu^{N,1}_t(\varphi)&=\int_t^{t'}\bar\mu^N_s\left(h\varphi(t'-s+\cdot)\frac{F^c(t'-s+\cdot)}{F^c(\cdot)}\right)ds\\
&\quad+\int_0^t\bar\mu^N_s\left(h\frac{\varphi(t'-s+\cdot)F^c(t'-s+\cdot)-\varphi(t-s+\cdot)F^c(t-s+\cdot)}{F^c(\cdot)}\right)ds,
\end{align*}
hence, provided $t'-t\le\delta$,
\begin{align*}
\big|\nu^{N,1}_{t'}(\varphi)-\nu^{N,1}_t(\varphi)\big|&\le \|h\|_\infty\|\varphi\|_\infty \delta+\frac{\|h\|_\infty}{F^c(t+K)}\omega_{\varphi F^c}(\delta,T+K)\,.
\end{align*}
Hence clearly $\nu^{N,1}_t(\varphi)$ satisfies \eqref{Ctight}.

We finally consider $\nu^{N,2}_t(\varphi)$: 
\begin{align*}
&\big|\nu^{N,2}_{t'}(\varphi)-\nu^{N,2}_t(\varphi)\big| \\
& \le\frac{1}{N}\int_t^{t'}\int_{0}^\infty\int_0^1 \varphi(t'-s+H(\mu^N_{s^-},w))\frac{F^c(t'-s+H(\mu^N_{s^-},w))}{F^c(H(\mu^N_{s^-},w))}{\bf1}_{v\le\mu^N_{s^-}(h)}{Q}_{rec}(ds,dv,dw)\\
& \quad +\!\!\frac{1}{N}\!\!\int_0^{t}\!\!\!\int_{0}^\infty\!\!\!\!\int_0^1\!\!\! \, \frac{|\varphi(t'\!\!-\!\!s\!\!+\!\!H(\mu^N_{s^-},v))F^c(t'\!\!-\!\!s\!\!+\!\!H(\mu^N_{s^-},v))-
\varphi(t\!\!-\!\!s\!\!+\!\!H(\mu^N_{s^-},w))F^c(t\!\!-\!\!s\!\!+\!\!H(\mu^N_{s^-},w))|}{F^c(H(\mu^N_{s^-},w))} \\
& \qquad \qquad \qquad \qquad \times{\bf1}_{v\le\mu^N_{s^-}(h)}{Q}_{rec}(ds,dv,dw)\\
&\le\|\varphi\|_\infty\frac{1}{N}{Q}_{rec}([t,t']\times[0,N\|h\|_\infty]\times[0,1]) \\
& \quad +\frac{\omega_{\varphi F^c}(t'-t,T+\bar\mfa)}{F^c(t+\bar\mfa)}\frac{1}{N}{Q}_{rec}([0,t]\times[0,N\|h\|_\infty]\times[0,1])\\
&=(t'-t)\|\varphi\|_\infty\|h\|_\infty+\omega_{\varphi F^c}(t'-t,T)\frac{t\|h\|_\infty}{F^c(t+\bar\mfa)}\\
&\quad+\|\varphi\|_\infty\frac{1}{N}\bar{Q}_{rec}([t,t']\times[0,N\|h\|_\infty]\times[0,1]) \\
& \quad 
+\frac{\omega_{\varphi F^c}(t'-t,T+\bar\mfa)}{F^c(t+\bar\mfa)}\frac{1}{N}\bar{Q}_{rec}([0,t]\times[0,N\|h\|_\infty]\times[0,1])\,. 
\end{align*}
The sup over $0\le t<t'\le T,\ t'-t\le\delta$ of the first line on the right hand side tends to $0$ as $\delta\to0$, while the term on the second line tends to $0$ in probability as $N\to\infty$, uniformly over $0\le t<t'\le T$, as explained above. This proves that 
$\nu^{N,2}_t(\varphi)$ satisfies \eqref{Ctight}.  The Lemma is established.
\end{proof}

\medskip

We finally establish uniqueness of the solution of the LLN limiting PDE.
\begin{prop}
Assume that $F\in C^1$ and $F(\mfa)>0$ for all $\mfa>0$. Then
the   PDE \eqref{eqn-PDE} has at most one solution in the space $C(\R_+;\mathcal{M}_{F}(\R_+))$. 
\end{prop}
\begin{proof}
We use a duality argument.
Note that if $\mu\in C(\R_+;\mathcal{M}_{F}(\R_+))$
solves the PDE \eqref{eqn-PDE}, then for any 
$\varphi\in C^1_b(\R_+)$, the mapping $t\mapsto \mu_t(\varphi)=\int_0^\infty \varphi(\mfa)\mu(t,d\mfa)$ is differentiable, and 
\begin{align}\label{eq:PDE}
\frac{d}{dt}\mu_t(\varphi)=\varphi(0)\overline{\Upsilon}(t)+\mu_t(\varphi'-h\varphi),\quad  t\ge0\,.
\end{align}
Moreover, 
if $\{\mu_t,\ t\ge0\}$ solves the PDE \eqref{eqn-PDE} and $(t,\mfa)\mapsto\varphi(t,\mfa)=\varphi_t(\mfa)$ is an element of $C^1_c(\R_+^2)$, then 
\begin{align}\label{eq:PDE-t}
\frac{d}{dt}\mu_t(\varphi_t)=\varphi_t(0)\overline{\Upsilon}(t)+\mu_t(\partial_t\varphi_t+\partial_\mfa\varphi_t-h\varphi_t),\quad t\ge0\,.
\end{align}
Suppose now that we have two solutions $\mu_t$ and $\nu_t$ of the PDE \eqref{eqn-PDE} with the same initial condition. Then the difference 
$\Delta\mu_t:=\mu_t-\nu_t$ satisfies $\Delta\mu_0=0$ and for any $\varphi\in C^1_c(\R_+^2)$,
\begin{align}\label{eq:extPDE}
\frac{d}{dt}\Delta\mu_t(\varphi_t)=\Delta\mu_t(\partial_t\varphi_t+\partial_\mfa\varphi_t-h\varphi_t),\quad t\ge0\,.
\end{align}
Now consider the following backward PDE: for $T>0$ and $g\in C_c^1(\R_+)$ arbitrary,
\begin{align}\label{eq:adjointPDE}
\partial_t v_t+\partial_\mfa v_t=hv_t,\ 0\le t\le T,\quad v(T,\mfa)=g(\mfa)\,.
\end{align}
This last equation has the following explicit solution:
\[v(t,\mfa)=\frac{F^c(\mfa+T-t)}{F^c(\mfa)}g(\mfa+T-t)\,.\]
Indeed, this function $v$ satisfies $v(T,\mfa)=g(\mfa)$. Moreover, $\partial_t+\partial_\mfa$ of the factor
$F^c(\mfa+T-t)g(\mfa+T-t)$ vanishes, since it is a function of $\mfa-t$, while
\[(\partial_t+\partial_\mfa)\frac{1}{F^c(\mfa)}
=\partial_\mfa\frac{1}{F^c(\mfa)}=\frac{f(\mfa)}{(F^c(\mfa))^2}
=\frac{h(\mfa)}{F^c(\mfa)}\,.
\]
Hence this  function $v$ satisfies \eqref{eq:adjointPDE}.	Moreover, since in particular $F\in C^1$, 
 $v\in C^1_c(\R_+^2)$.
Combining \eqref{eq:extPDE} and \eqref{eq:adjointPDE}, we deduce that $\frac{d}{dt}\Delta\mu_t(v_t)=0$,
for all $t\in[0,T]$. Consequently,
$\Delta\mu_T(v_T)=\Delta\mu_0(v_0)=0$, i.e. $\Delta\mu_T(g)=0$, for any $g\in C_c^1(\R_+)$. Therefore, 
$\Delta\mu_T=0$. But this true for any $T>0$. Uniqueness of the solution of the PDE \eqref{eqn-PDE} has been established. 
\end{proof}

\medskip

\section{Proof of Theorem \ref{thm-FCLT-mu}} \label{sec-wh-muN-proof}

%

Recall the expressions of $\hat{\mu}^N_t(\varphi)$ in \eqref{eqn-hat-muN-rep}, with $\hat{\mu}^{inf,N}_t(\varphi)$ in \eqref{eqn-hat-muN-rep-inf} and $\hat{\mu}^{rec,N}_t(\varphi)$ in \eqref{eqn-hat-muN-rep-rec}. 
Here we apply the tightness criterion in \cite{mitoma1983tightness}  for stochastic  processes in 
 $D(\R_+;(H^{1}(\R_+))')$.

\begin{lemma} \label{lem-hat-mu0-conv}
Under Assumption \ref{AS-g}, for any $\varphi\in H^1(\RR_+)$, 
\[
\{\hat{\mu}^N_0(\varphi(t+\cdot)),\ t\ge0\} \Rightarrow \{\hat{\mu}_0(\varphi(t+\cdot)),\ t\ge0\} \qinq D
\]
as $N\to\infty$, where $\{\hat\mu_{0}\}$ is given in Proposition \ref{prop-initCLT}. 
\end{lemma}

\begin{proof}
Recall the expression of $ \hat\mu^N_0$ in \eqref{eqn-hat-muN0}. 
In particular, given that $d\bar{\bar{\mu}}_0 (d\mfa) = g_0(\mfa)d\mfa$, we have
\[
 \hat\mu^N_0([0,\mfa])= \frac{1}{\sqrt{N}}\sum_{j=1}^{I^N(0)} \Big( {\bf1}_{\tilde\tau_{j,0} \le \mfa} -\bar{\bar{\mu}}_0([0,\mfa])  \Big)\,. 
\]
Thus we can write
\begin{align*}
\hat{\mu}^N_0(\varphi(t+\cdot)) &= \int_0^\infty \varphi(t+\mfa)  \hat\mu^N_0(d\mfa)   \\
&=  \frac{1}{\sqrt{N}} \sum_{j=1}^{I^N(0)} \Big( \varphi(\tilde\tau_{j,0}+t) \  -  \int_0^\infty \varphi(\mfa+t) \bar{\bar{\mu}}_0(d\mfa) \Big)  \,.
\end{align*}
 Observe that the summation is over a sequence of i.i.d. random variables. Thus, by the CLT, we immediately obtain the convergence of finite dimensional distributions. 
 
 To prove tightness, we  apply Theorem 13.5 in \cite{billingsley1999convergence}, which says in particular that it suffices to show that for any $T>0$, any $t_1<t<t_2\le T$,  $\varphi \in H^1(\RR_+)$, and 
 $N\ge1$,
 \begin{align} \label{eqn-hat-mu0-tight}
 \E\big[ \big|\hat{\mu}^N_0(\varphi(t+\cdot))  -\hat{\mu}^N_0(\varphi(t_1+\cdot)) \big |^2 \big|\hat{\mu}^N_0(\varphi(t_2+\cdot))   - \hat{\mu}^N_0(\varphi(t+\cdot)) \big|^2\big] \le (G(t_2) - G(t_1))^\alpha
 \end{align} 
 for some $\alpha >1$ and some nondecreasing nonnegative  continuous function $G$. From now on, $T$ will be arbitrarily fixed.
We have 
\[
\hat{\mu}^N_0(\varphi(t+\cdot))  -\hat{\mu}^N_0(\varphi(t_1+\cdot)) =  \frac{1}{\sqrt{N}} \sum_{j=1}^{I^N(0)} \big[ (Z_{j,t} - \E[Z_{j,t}] ) - (Z_{j,t_1} - \E[Z_{j,t_1}] )    \big]\,,
\]
 and
 \[
\hat{\mu}^N_0(\varphi(t_2+\cdot))   - \hat{\mu}^N_0(\varphi(t+\cdot)) =  \frac{1}{\sqrt{N}} \sum_{j=1}^{I^N(0)} 
 \big[ (Z_{j,t_2} - \E[Z_{j,t_2}] ) - (Z_{j,t} - \E[Z_{j,t}] )    \big]\,,
 \]
 with
 \[
 Z_{j,t} :=  \varphi(\tilde\tau_{j,0}+t), \quad \E[Z_{j,t}] =  \int_0^\infty \varphi(\mfa+t)  \bar{\bar{\mu}}_0(d\mfa)\,. 
 \]
 Also, for notational convenience, we write $\bar{Z}_{j,t} =Z_{j,t} - \E[Z_{j,t}]$. Note that the stochastic processes $\bar{Z}_{j,\cdot}$ are mutually independent and centered.
 Then, we have
 \begin{align*}
&   \E\big[ \big|\hat{\mu}^N_0(\varphi(t+\cdot))  -\hat{\mu}^N_0(\varphi(t_1+\cdot)) \big |^2 \big|\hat{\mu}^N_0(\varphi(t_2+\cdot))   - \hat{\mu}^N_0(\varphi(t+\cdot)) \big|^2\big] \\
& = \frac{1}{N^2} \E\bigg[\bigg( \sum_{j=1}^{I^N(0)} (\bar{Z}_{j,t}-\bar{Z}_{j,t_1}) \bigg)^2 \bigg( \sum_{j=1}^{I^N(0)} (\bar{Z}_{j,t_2}-\bar{Z}_{j,t}) \bigg)^2 \bigg] \\
& = \frac{1}{N^2} \E\bigg[\bigg( \sum_{j=1}^{I^N(0)} (\bar{Z}_{j,t}-\bar{Z}_{j,t_1})^2 + \sum_{j, j'=1, j'\neq j}^{I^N(0)}(\bar{Z}_{j,t}-\bar{Z}_{j,t_1}) (\bar{Z}_{j',t}-\bar{Z}_{j',t_1})  \bigg)\\
& \qquad \quad \times \bigg( \sum_{j=1}^{I^N(0)} (\bar{Z}_{j,t_2}-\bar{Z}_{j,t})^2 
+  \sum_{j, j'=1, j'\neq j}^{I^N(0)}(\bar{Z}_{j,t_2}-\bar{Z}_{j,t}) (\bar{Z}_{j',t_2}-\bar{Z}_{j',t})   \bigg) \bigg] \\
& = \frac{1}{N^2} \bigg( \sum_{j=1}^{I^N(0)} \E\big[ (\bar{Z}_{j,t}-\bar{Z}_{j,t_1})^2(\bar{Z}_{j,t_2}-\bar{Z}_{j,t})^2  \big]  + \sum_{j, j'=1, j'\neq j}^{I^N(0)} \E\big[ (\bar{Z}_{j,t}-\bar{Z}_{j,t_1})^2\big] \E\big[ (\bar{Z}_{j',t_2}-\bar{Z}_{j',t})^2\big] \\
& \qquad \quad 
+  \sum_{j, j'=1, j'\neq j}^{I^N(0)}  \E\big[ (\bar{Z}_{j,t}-\bar{Z}_{j,t_1})(\bar{Z}_{j,t_2}-\bar{Z}_{j,t})  \big]   \E\big[ (\bar{Z}_{j',t}-\bar{Z}_{j',t_1})(\bar{Z}_{j',t_2}-\bar{Z}_{j',t})  \big]    \bigg)\,.
 \end{align*}
 We next calculate each of these terms. 
 We write
 \[
\Psi_{1}(\mfa):= \varphi(\mfa+t)  -\varphi(\mfa+t_1)\, \quad \text{and} \quad \Psi_{2}(\mfa):= \varphi(\mfa+t_2)  -\varphi(\mfa+t)  \,,
 \]
 and
 observe that for Cauchy--Schwartz's inequality, we deduce that
 \begin{equation}\label{eqn-Psi12-bound}
 \begin{aligned} 
 |\Psi_{1}(\mfa)|&= | \varphi(\mfa+t)  -\varphi(\mfa+t_1) | \\
 &\le \sqrt{ |t-t_1| \int_{t_1+\mfa}^{t+\mfa}|\varphi'(r)|^2dr}\\
 &\le c_\varphi\sqrt{t-t_1},\\
   |\Psi_{2}(\mfa)| &\le c_\varphi \sqrt{t_2-t}\,, 
 \end{aligned}
 \end{equation}
 where $c_\varphi:=\sqrt{\int_0^{\infty}|\varphi'(r)|^2dr}$.
 We have
  \begin{align*}
  \E\big[ (\bar{Z}_{j,t}-\bar{Z}_{j,t_1})^2\big] &= \int_0^\infty  \Psi_1^2(\mfa) \bar{\bar{\mu}}_0(d\mfa)  - \bigg(\int_0^\infty \Psi_1(\mfa) \bar{\bar{\mu}}_0(d\mfa)  \bigg)^2 \\
    & \le  \int_0^\infty  \Psi_1^2(\mfa) \bar{\bar{\mu}}_0(d\mfa)  \le  2 c_\varphi^2|t-t_1| . 
  \end{align*}
  Similarly, the same bound holds for $ \E\big[ (\bar{Z}_{j',t_2}-\bar{Z}_{j',t})^2\big]$. So we get
\begin{align*}
\E\big[ (\bar{Z}_{j,t}-\bar{Z}_{j,t_1})^2\big] \E\big[ (\bar{Z}_{j',t_2}-\bar{Z}_{j',t})^2\big]  \le 4 c_\varphi^4 |t_2-t_1|^2\,. 
\end{align*}

 Next, we have
\begin{align*}
\E\big[ (\bar{Z}_{j,t}-\bar{Z}_{j,t_1})(\bar{Z}_{j,t_2}-\bar{Z}_{j,t})  \big] 
&=  \int_0^\infty  \Psi_1 (\mfa)\Psi_2 (\mfa)\bar{\bar{\mu}}_0(d\mfa)  - \int_0^\infty \Psi_1(\mfa) \bar{\bar{\mu}}_0(d\mfa)  \int_0^\infty \Psi_2(\mfa) \bar{\bar{\mu}}_0(d\mfa) \,. 
\end{align*}
By \eqref{eqn-Psi12-bound}, we obtain that its absolute value is bounded by $ 2 c_\varphi^2|t-t_1|$. 
Similarly we can bound $\E\big[ (\bar{Z}_{j',t}-\bar{Z}_{j',t_1})(\bar{Z}_{j',t_2}-\bar{Z}_{j',t})  \big] $.
So we have that for $j\neq j'$,
\begin{align*}
 \E\big[ (\bar{Z}_{j,t}-\bar{Z}_{j,t_1})(\bar{Z}_{j,t_2}-\bar{Z}_{j,t})  \big]   \E\big[ (\bar{Z}_{j',t}-\bar{Z}_{j',t_1})(\bar{Z}_{j',t_2}-\bar{Z}_{j',t})  \big] \le 4 c_\varphi^4 |t_2-t_1|^2\,. 
\end{align*}

Now we calculate 
\begin{align*}
&  \E\big[ (\bar{Z}_{j,t}-\bar{Z}_{j,t_1})^2(\bar{Z}_{j,t_2}-\bar{Z}_{j,t})^2  \big] = \E\bigg[ \Big( \tilde\Psi_1  - \int_0^\infty \Psi_1 \bar{\bar{\mu}}_0(d\mfa) \Big)^2 \Big( \tilde\Psi_2  - \int_0^\infty \Psi_2 \bar{\bar{\mu}}_0(d\mfa) \Big)^2 \bigg] \,,
\end{align*}
with $\tilde\Psi_1=Z_{j,t} - Z_{j,t_1}$ and $\tilde\Psi_2=Z_{j,t_2} - Z_{j,t}$. 
It is equal to (dropping $\mfa$ in $\Psi_1$ and $\Psi_2$ in the integrands below for brevity)
\begin{align*}
&\int_0^\infty \Psi_1^2 \Psi_2^2 \bar{\bar{\mu}}_0(d\mfa) + \int_0^\infty \Psi_1^2 \bar{\bar{\mu}}_0(d\mfa) \Big(\int_0^\infty \Psi_2 \bar{\bar{\mu}}_0(d\mfa) \Big)^2 + \int_0^\infty \Psi_2^2 \bar{\bar{\mu}}_0(d\mfa) \Big(\int_0^\infty \Psi_1 \bar{\bar{\mu}}_0(d\mfa) \Big)^2 \\
&
+ \int_0^\infty \Psi_1 \Psi_2 \bar{\bar{\mu}}_0(d\mfa) \int_0^\infty \Psi_1 \bar{\bar{\mu}}_0(d\mfa)\int_0^\infty \Psi_2 \bar{\bar{\mu}}_0(d\mfa) \\
& - 2 \int_0^\infty \Psi_1 \Psi_2^2 \bar{\bar{\mu}}_0(d\mfa) \int_0^\infty \Psi_1 \bar{\bar{\mu}}_0(d\mfa)  - 2   \int_0^\infty \Psi_1^2 \Psi_2 \bar{\bar{\mu}}_0(d\mfa) \int_0^\infty \Psi_2 \bar{\bar{\mu}}_0(d\mfa) 
 \\
&  - 3\Big(\int_0^\infty \Psi_1 \bar{\bar{\mu}}_0(d\mfa) \Big)^2 \Big(\int_0^\infty \Psi_2 \bar{\bar{\mu}}_0(d\mfa) \Big)^2 \,.
 \end{align*}
By the bounds of $|\Psi_1|$ and $|\Psi_2|$ in \eqref{eqn-Psi12-bound},  the absolute value of each term is bounded by
$4 c_\varphi^4 |t_2-t_1|^2$.  

 Combining the above estimates, we obtain \eqref{eqn-hat-mu0-tight} holds with $G(t)=C t$, $C$ being equal to some finite factor times $c_\varphi^4$, and 
 $\alpha=2$, and then applying Theorem 13.5 in \cite{billingsley1999convergence}, we conclude the convergence $\hat{\mu}^N_0(\varphi(t+\cdot)) \Rightarrow \hat{\mu}_0(\varphi(t+\cdot))$ in $D$.
\end{proof}

\medskip

\begin{lemma} \label{lem-hat-muUpsilon-conv}
For any $\varphi\in C_b(\RR)$, 
\[
\Big\{\int_0^{t} \varphi(t-s)  \widehat\Upsilon^N(s)ds,\,  t\ge 0\Big\} \Rightarrow \Big\{ \int_0^t \varphi(t-s)   \widehat\Upsilon(s)ds, \, t\ge 0\Big\}
\]
in $C$ as $N\to\infty$.
\end{lemma}

\begin{proof}
By Theorem \ref{thm-FCLT-SF}, we have $
\widehat{\Upsilon}^N(\cdot) \Rightarrow \widehat{\Upsilon}(\cdot) $ in $ D
$
where $ \widehat{\Upsilon}(t)$ is given in \eqref{eqn-hat-Upsilon}, and has paths in $C$ almost surely. 
Applying the continuous mapping theorem to the mapping $x\in D \to \{\int_0^t \varphi(t-s)  x(s) ds\,, t\ge 0\} \in C$, we obtain the convergence as claimed.
\end{proof} 

\medskip

\begin{lemma}\label{lem-conv-hat-inf}
For any $\varphi\in H^1(\R_+)$,  $\hat{\mu}^{inf,N}(\varphi)\Rightarrow\hat{\mu}^{inf}(\varphi)$ in $D$ as $N\to\infty$.
\end{lemma}
\begin{proof}
We define
\begin{align*}
\tilde{\mu}^{inf,N}_t(\varphi)=\frac{1}{\sqrt{N}}\int_0^t\int_0^\infty\varphi(t-s){\bf1}_{v\le N\overline\Upsilon(s)}\bar{Q}_{inf}(ds,dv)\,.
\end{align*}
We first note that
\begin{align*}
\E\left[\left(\hat{\mu}^{inf,N}_t(\varphi)-\tilde{\mu}^{inf,N}_t(\varphi)\right)^2\right]
=\E\int_0^t\varphi^2(t-s)|\overline\Upsilon(s)-\overline\Upsilon^N(s)| ds,\end{align*}
which tends to $0$ as $N\to\infty$.

Next it follows from Corollary 2.9 in \cite{pardoux-PRM-book}
 that the joint cumulants of the random variables 
$(\tilde{\mu}^{inf,N}_{t_1}(\varphi_1),\ldots,\tilde{\mu}^{inf,N}_{t_m}(\varphi_m))$ equals $0$ for $m=1$, and otherwise is given as
\begin{align*}
\kappa_m(\tilde{\mu}^{inf,N}_{t_1}(\varphi_1),\ldots,\tilde{\mu}^{inf,N}_{t_m}(\varphi_m))=N^{1-m/2}\int_0^{t_1\wedge\cdots\wedge t_m}
\varphi_1(t_1-s)\times\cdots\times\varphi_m(t_m-s)\overline\Upsilon(s)ds\,.
\end{align*}
The cumulants of order 1 are $0$, all cumulants of order $m\ge3$ tend to $0$, while for $m=2$,
\begin{align*}
\kappa_2(\tilde{\mu}^{inf,N}_{t_1}(\varphi_1),\tilde{\mu}^{inf,N}_{t_2}(\varphi_2))=\int_0^{t_1\wedge t_2}\varphi_1(t_1-s)\varphi_2(t_2-s)\overline\Upsilon(s)ds\,.
\end{align*}
The above arguments allow us to conclude that the finite dimensional distributions of the generalized random field $\hat{\mu}^{inf,N}_t(\varphi)$ converge towards
those of $\hat{\mu}^{inf}_t(\varphi)$. If remains to establish tightness in $D$ of the sequence $\hat{\mu}^{inf,N}(\varphi)$, which will be a consequence of the next two lemmas.
\end{proof}

\medskip

\begin{lemma}\label{le-tilde-inf-tight}
For any $\varphi\in H^1(\R_+)$, the sequence $\{\tilde{\mu}^{inf,N}(\varphi)\}$ is tight in $D$.
\end{lemma}

\begin{proof}
 For the sake of simplifying our notations, we define, with $\varphi\in C^1_b(\RR_+)$ being fixed, 
  $\xi^N_t:=\tilde{\mu}^{inf,N}_{t}(\varphi)$. 
  We will again exploit Theorem 13.5 in Billingsley, more precisely Billingsley's condition (13.14) in the following special form. For any $T>0$,
  there exists a nondecreasing continuous function $G$ and a real $\alpha>1$ such that for any $t_1<t<t_2\le T$, $N\ge1$,
  \begin{align}\label{3pts}
  \E\left[(\xi^N_t-\xi^N_{t_1})^2(\xi^N_{t_2}-\xi^N_t)^2\right]\le (G(t_2)-G(t_1))^\alpha\,. 
  \end{align}
  For that sake, we will make use of the following formula from 
  Exercise 2.21 in \cite{pardoux-PRM-book}. 
  Let $Q$ be a Poisson Random measure on the measurable space $(E,\mathcal{E})$ with mean measure $\mu$,
  and $\overline{Q}=Q-\mu$ its associated compensated measure. 
  Then for any 
  $f_1,f_2\in L^1(E,\mathcal{E},\mu)\cap L^4(E,\mathcal{E},\mu)$, 
  \begin{align}\label{DP}
   \E\left[\overline{Q}(f_1)^2 \overline{Q}(f_2)^2\right]=\mu(f_1^2f_2^2)+\mu(f_1^2)\mu(f_2^2)+2[\mu(f_1f_2)]^2\,.
 \end{align}
 We shall apply this formula after having identified $f_1$ and $f_2$ such that
 \[ \xi^N_t-\xi^N_{t_1}=\overline{Q}(f_1)\ \text{ and } \xi^N_{t_2}-\xi^N_t=\overline{Q}(f_2)\,,\]
 in the case $E=\RR_+^2$, $\mu(ds,du)=ds du$.
 
 We have 
 \begin{align*}
  \xi^N_t-\xi^N_{t_1}&= \frac{1}{\sqrt{N}}\int_0^t  \int_0^\infty   {\bf1}_{u < N\overline\Upsilon(s)} \varphi(t-s)  \overline{Q}(ds, du) \\
  & \quad - \frac{1}{\sqrt{N}}\int_0^{t_1}  \int_0^\infty   {\bf1}_{u < N\overline\Upsilon(s)} \varphi(t_1-s)  \overline{Q}(ds, du)\\
  &=\overline{Q}(f_1),
 \end{align*}
 where 
 \begin{align*}
 f_1(s,u)&=\frac{1}{\sqrt{N}}{\bf1}_{(t_1,t]}(s) \varphi(t-s) {\bf1}_{u < N\overline\Upsilon(s)}
 +\frac{1}{\sqrt{N}}{\bf1}_{[0,t_1]}(s)\Big(\varphi(t-s)-\varphi(t_1-s)\Big)
{\bf1}_{u < N\overline\Upsilon(s)}\,.
 \end{align*}
 Similarly, 
 \begin{align*}
 \xi^N_{t_2}-\xi^N_t=\overline{Q}(f_2),
 \end{align*}
 with
 \begin{align*}
 f_2(s,u)&=\frac{1}{\sqrt{N}}{\bf1}_{(t,t_2]}(s) \varphi(t_2-s) {\bf1}_{u < N\overline\Upsilon(s)}
+\frac{1}{\sqrt{N}}{\bf1}_{[0,t]}(s)\Big(\varphi(t_2-s)-\varphi(t-s)\Big)
{\bf1}_{u < N\overline\Upsilon(s)}\,.
 \end{align*}
Now we have
\begin{align*}
 f_1(s,u)f_2(s,u)&=\frac{1}{N}{\bf1}_{(t_1,t]}(s)\varphi(t-s)\Big(\varphi(t_2-s)-\varphi(t-s)\Big) {\bf1}_{u < N\overline\Upsilon(s)}\\&\quad
  \\
 &\quad+\frac{1}{N}{\bf1}_{[0,t_1]}(s)\Big(\varphi(t-s)-\varphi(t_1-s)\Big)
 \Big(\varphi(t_2-s)-\varphi(t-s)\Big){\bf1}_{u < N\overline\Upsilon(s)}\,. 
 \end{align*}
 Hence, with $c_\varphi$ denoting the norm of $\varphi$ in $H^1(0,T+K)$, by the same computation as done in the proof of Lemma \ref{lem-hat-mu0-conv},
 \begin{align}
  \mu(f_1^2) &=\int_{t_1}^t \varphi^2(t-s) \overline\Upsilon(s)ds
  +\int_0^{t_1} \Big(\varphi(t-s)-\varphi(t_1-s)\Big)^2\overline\Upsilon(s)ds\nonumber\\
&\le  \lambda^\ast \|\varphi\|_\infty^2(t-t_1)    + \lambda^\ast t_1 c_\varphi^2 (t-t_1)\nonumber\\
&\le C(t_2-t_1), \label{f1^2}
 \end{align} 
where the constant $C$ depends upon $t_2$, $\lambda^\ast$ $\|\varphi\|_\infty$ and $c_\varphi$. We also obtain
\begin{align}\label{f2^2}
  \mu(f_2^2)\le C(t_2-t_1)\,. 
 \end{align}
 We next compute
 \begin{align*}
 \mu(f_1f_2)&=\int_{t_1}^t \varphi(t-s)\Big(\varphi(t_2-s)-\varphi(t-s)\Big)\overline{\Upsilon}(s)ds \\
 &\quad+\int_0^{t_1}\Big(\varphi(t-s)-\varphi(t_1-s)\Big)\Big(\varphi(t_2-s)-\varphi(t-s)\Big)\overline{\Upsilon}(s)ds\,.
 \end{align*}
Thus, we have
 \begin{align}\label{f1f2}
 |\mu(f_1f_2)|\le\lambda^\ast \|\varphi\|_\infty c_\varphi (t_2-t_1)^{3/2}+\lambda^\ast c_\varphi^2(t_2-t_1)\,.
  \end{align}
  
  Finally, we obtain
 \begin{align}
  \mu(f_1^2f_2^2) &=  \frac{1}{N}\int_{t_1}^t\varphi^2(t-s)\Big(\varphi(t_2-s)-\varphi(t-s)\Big)^2\overline{\Upsilon}(s)ds\nonumber\\
  &\quad+\frac{1}{N}\int_0^{t_1} \Big(\varphi(t-s)-\varphi(t_1-s)\Big)^2\Big(\varphi(t_2-s)-\varphi(t-s)\Big)^2
  \overline{\Upsilon}(s)ds\nonumber\\
  &\le \frac{\|\varphi\|_\infty^2c_\varphi^2\lambda^\ast}{N}(t_2-t_1)^2+\frac{c_\varphi^4\lambda^\ast t_1}{N}(t_2-t_1)^2\label{f1f2^2}\,. 
   \end{align}
  Putting together \eqref{DP}, \eqref{f1^2}, \eqref{f2^2}, \eqref{f1f2} and \eqref{f1f2^2},  there exists a constant $C_T$ such that for any $0\le t_1<t<t_2\le T $, $N\ge1$,
   \begin{align}\label{3pts-pb}
  \E\left[(\xi^N_t-\xi^N_{t_1})^2(\xi^N_{t_2}-\xi^N_t)^2\right]\le C_T(t_2-t_1)^2\,. 
  \end{align}
Applying Theorem 13.5 in \cite{billingsley1999convergence}, we obtain the desired result. 
    \end{proof}

\medskip

\begin{lemma}\label{le-diff-inf-tight}
For any $\varphi\in H^1(\RR_+)$,
  the sequence $\{\hat{\mu}^{inf,N}_{\cdot}(\varphi)-\tilde{\mu}^{inf,N}_{\cdot}(\varphi),\ N\ge1\}$ is tight in $D$.
\end{lemma}

\begin{proof}
For the sake of simplifying our notations, we define
\[\Delta^{N}_t:=\hat{\mu}^{inf,N}_{t}(\varphi)-\tilde{\mu}^{inf,N}_{t}(\varphi)\,.\]
We shall exploit the Corollary page 83 of Billingsley \cite{billingsley1999convergence}. More precisely, a consequence of that Corollary  
is that, since $\Delta^{N}_0=0$ for all $N\ge1$, 
taking into account the inequality (12.7) in Billingsley \cite{billingsley1999convergence}, the statement of the lemma will follow from the following 
fact: for any $t>0$, $\ep>0$, as $\delta\to0$,
\begin{align}\label{Bi83}
\limsup_{N\to\infty}\frac{1}{\delta}\P\left(\sup_{0\le r\le \delta} |\Delta^{N}_{t+r}-\Delta^{N}_t|\ge\ep\right)\to0\,.
\end{align}
In order to simplify our notations below, we define
\[ D_N(s,u):={\bf1}_{u\le \Upsilon^N(s^-)}-{\bf1}_{u\le N\overline{\Upsilon}(s)}\,,\]
and we note that
\begin{align*}
 |D_N(s,u)|&={\bf1}_{\Upsilon^N(s^-)\wedge N\overline{\Upsilon}(s)<u\le \Upsilon^N(s^-)\vee N\overline{\Upsilon}(s)},\\
 \int_0^\infty |D_N(s,u)| du&=|\Upsilon^N(s^-) - N\overline{\Upsilon}(s)|,\\
  \frac{1}{\sqrt{N}}\int_0^\infty |D_N(s,u)| du&=|\widehat{\Upsilon}^N(s^-)|,\\
  \frac{1}{N}\int_0^\infty |D_N(s,u)| du&=|\overline{\Upsilon}^N(s^-)-\overline{\Upsilon}(s)|\,.
 \end{align*}
We have
\begin{align*}
\Delta^{N}_t&=\frac{1}{\sqrt{N}}\int_0^t\int_0^\infty\varphi(t-s)D_N(s,u)\overline{Q}(ds,du),\\
\Delta^{N}_{t+r}-\Delta^{N}_t&=\frac{1}{\sqrt{N}}\int_t^{t+r}\int_0^\infty\varphi(t+r-s)D_N(s,u)\overline{Q}(ds,du)\\
&\quad+\frac{1}{\sqrt{N}}\int_0^t\int_0^\infty\Big(\varphi(t+r-s)-\varphi(t-s)\Big)D_N(s,u)
\overline{Q}(ds,d\eta,du)\,.
\end{align*}
Next, we obtain 
\begin{align*}
|\Delta^{N}_{t+r}-\Delta^{N}_t|&\le\frac{1}{\sqrt{N}}\|\varphi\|_\infty\int_t^{t+r}\int_0^\infty|D_N(s,u)|{Q}(ds,du)
+\|\varphi\|_\infty\int_t^{t+r}|\hat{\Upsilon}^N(s)|ds\\
&\quad+\frac{c_\varphi \sqrt{r}}{\sqrt{N}}\int_0^t \int_0^\infty|D_N(s,u)|{Q}(ds,du)
+c_\varphi \sqrt{r}\int_0^t | \widehat{\Upsilon}^N(s)|ds\,.
\end{align*}
As a consequence,
\begin{align*}
\sup_{0\le r\le \delta}|\Delta^{N}_{t+r}-\Delta^{N}_t|&\le
\frac{1}{\sqrt{N}}\|\varphi\|_\infty\int_t^{t+\delta}\int_0^\infty|D_N(s,u)|{Q}(ds,du)
+\|\varphi\|_\infty\int_t^{t+\delta} |\widehat{\Upsilon}^N(s)| ds\\
&\quad+\frac{c_\varphi \sqrt{\delta}}{\sqrt{N}}\int_0^t \int_0^\infty|D_N(s,u)|{Q}(ds,du)
+c_\varphi \sqrt{\delta}\int_0^t | \widehat{\Upsilon}^N(s)|ds\\
&=\frac{1}{\sqrt{N}}\|\varphi\|_\infty\int_t^{t+\delta}\int_0^\infty|D_N(s,u)|\overline{Q}(ds,du)
+2\|\varphi\|_\infty\int_t^{t+\delta} |\widehat{\Upsilon}^N(s)| ds\\
&\quad+\frac{c_\varphi \sqrt{\delta}}{\sqrt{N}}\int_0^t \int_0^\infty|D_N(s,u)|\overline{Q}(ds,du)
+2c_\varphi \sqrt{\delta}\int_0^t | \widehat{\Upsilon}^N(s)|ds\,. 
\end{align*}
We have proved that
\begin{align*}
\sup_{0\le r\le \delta}|\Delta^{N}_{t+r}-\Delta^{N}_t|&\le A^N_{t,\delta}+2B^N_{t,\delta}\,,
\end{align*}
where
\begin{align*}
A^N_{t,\delta}&=\frac{1}{\sqrt{N}}\|\varphi\|_\infty\int_t^{t+\delta}\int_0^\infty|D_N(s,u)|\overline{Q}(ds,du)
+\frac{c_\varphi \sqrt{\delta}}{\sqrt{N}}\int_0^t \int_0^\infty|D_N(s,u)|\overline{Q}(ds,du),\\
B^N_{t,\delta}&=\|\varphi\|_\infty\int_t^{t+\delta} |\widehat{\Upsilon}^N(s)| ds+c_\varphi \sqrt{\delta}\int_0^t | \widehat{\Upsilon}^N(s)|ds\,. 
\end{align*}
Clearly, in order to establish \eqref{Bi83},  it suffices to establish the  following two facts: for any $t>0$, $\ep>0$, as 
$\delta\to0$,
\begin{align}
\limsup_{N\to\infty}\frac{1}{\delta}\P(A^N_{t,\delta}>\ep)\to0,\label{83A}\\
\limsup_{N\to\infty}\frac{1}{\delta}\P(B^N_{t,\delta}>\ep)\to0\,.\label{83B}
\end{align}
We first establish \eqref{83B}. We note
that
\begin{align*}
(B^N_{t,\delta})^2\le C(\|\varphi\|^2_\infty+c_\varphi^2)\delta\times\int_0^{t+\delta} |\widehat{\Upsilon}^N(s)|^2ds\,.
\end{align*}
So we have
\begin{align*}
\limsup_{N\to\infty}\frac{1}{\delta}\P(B^N_{t,\delta}>\ep)&\le\frac{1}{\delta}\P\left(\int_0^{t+\delta} |\widehat{\Upsilon}(s)|^2ds
\ge\frac{\ep}{C(\|\varphi\|^2_\infty+c_\varphi^2)\delta}\right)\\
&\le\frac{C^2(\|\varphi\|^2_\infty+c_\varphi^2)^2\delta}{\ep^2}\E\left[\left(\int_0^{t+1} |\widehat{\Upsilon}(s)|^2ds\right)^2\right],
\end{align*}
which goes to $0$ as $\delta\to0$. Hence,  \eqref{83B} follows.

We finally establish \eqref{83A}. For that sake, we first estimate the second moment of $A^N_{t,\delta}$: 
\begin{align*}
\E[(A^N_{t,\delta})^2]&= \E\Big[  \|\varphi\|_\infty^2\int_t^{t+\delta} |\overline{\Upsilon}^N(s)-\overline{\Upsilon}(s)|ds
+c_\varphi^2\delta\int_0^t|\overline{\Upsilon}^N(s)-\overline{\Upsilon}(s)|ds\Big]\\
&\le \E\Big[\sup_{0\le s\le t+\delta}|\overline{\Upsilon}^N(s)-\overline{\Upsilon}(s)| \Big] (\|\varphi\|_\infty^2+tc_\varphi^2)\delta\,. 
\end{align*}
Finally, we obtain
\begin{align*}
\frac{1}{\delta}\P(A^N_{t,\delta}>\ep)&\le\frac{\E[(A^N_{t,\delta})^2]}{\delta\ep^2}\\
&\le\frac{C}{\ep^2}\E\Big[\sup_{0\le s\le t+\delta}|\overline{\Upsilon}^N(s)-\overline{\Upsilon}(s)| \Big] \,.
\end{align*}
We know that $\overline{\Upsilon}^N(s)\to\overline{\Upsilon}(s)$ in probability locally uniformly in $s$, and $|\overline{\Upsilon}^N(s)|\le\lambda^\ast$, hence the lim sup as $N\to\infty$ of the above right hand side is $0$ for any $\delta>0$.
\eqref{83A} has been established.
\end{proof}

\medskip 

We next prove that  $\hat{\mu}^{rec,N}(\varphi)\Rightarrow\hat{\mu}^{rec}(\varphi)$ in $D$. We shall treat 
$\hat{\mu}^{rec,N}$ similarly as $\hat{\mu}^{inf,N}$.
Our aim is to establish the following lemma. 

\begin{lemma}\label{lem-conv-hat-rec}
For any $\varphi\in H^1(\R_+)$,  $\hat{\mu}^{rec,N}(\varphi)\Rightarrow\hat{\mu}^{rec}(\varphi)$ in $D$ as $N\to\infty$.
\end{lemma}

Let us first establish the convergence of finite-dimensional distributions, and then tightness in $D$. 

\begin{lemma} \label{lem-conv-hat-rec-fdd}
For any $\varphi\in H^1(\R_+)$ and any $0\le t_1<t_2<\cdots<t_k$, as $N\to\infty$, 
\[(\hat{\mu}_{t_1}^{rec,N}(\varphi),\ldots,\hat{\mu}_{t_k}^{rec,N}(\varphi))\Rightarrow(\hat{\mu}_{t_1}^{rec}(\varphi),\ldots,\hat{\mu}_{t_k}^{rec}(\varphi))\,.\]
\end{lemma}
\begin{proof}
We define
\begin{align*}
\tilde{\mu}^{rec,N}_t(\varphi)=\frac{1}{\sqrt{N}}\int_0^t\int_0^\infty\int_0^1\varphi(t-s+H(\bar\mu_s,w)){\bf1}_{v\le N\bar\mu_s(h)}\overline{Q}_{rec}(ds,dv,dw)\,.
\end{align*}
The proof will be divided in two steps.

{\bf Step 1}
We first want to show that for any $t>0$,
 $\hat{\mu}_{t}^{rec,N}(\varphi)-\tilde{\mu}^{rec,N}_t(\varphi)\to0$ in mean square, as $N\to\infty$.
We have 
\begin{align*}
&\hat{\mu}_{t}^{rec,N}(\varphi)-\tilde{\mu}^{rec,N}_t(\varphi)\\
&=\frac{1}{\sqrt{N}}\int_0^t\int_0^\infty\int_0^1\left[\varphi(t-s+H(\bar\mu^N_{s^-},v))-\varphi(t-s+H(\bar{\mu}_s,v))\right]{\bf1}_{u\le N\bar{\mu}_s(h)}\overline{Q}(ds,du,dv)\\
&\quad+\frac{1}{\sqrt{N}}\int_0^t\int_0^\infty\int_0^1\varphi(t-s+H(\bar\mu^N_{s^-},v))\left[{\bf1}_{u\le N \bar\mu^N_{s^-}(h)}-{\bf1}_{u\le N\bar{\mu}_s(h)}\right]\overline{Q}(ds,du,dv)\,,
\end{align*}
and
\begin{align*}
\E\left[|\hat{\mu}_{t}^{rec,N}(\varphi)-\tilde{\mu}^{rec,N}_t(\varphi)|^2\right]&\le
2\E\int_0^t\int_0^1\left[\varphi(t-s+H(\bar\mu^N_s,v))-\varphi(t-s+H(\bar{\mu}_s,v))\right]^2\bar{\mu}_s(h) dv ds\\
&\quad+2\E\int_0^t\int_0^\infty\varphi^2(t-s+r)h(r)\bar{\mu}^N_s(dr)\frac{|\bar{\mu}^N_s(h)-\bar{\mu}_s(h)|}{\bar{\mu}^N_s(h)}ds\,. 
\end{align*}
It is plain that the second term in the last right hand side is bounded by 
\[ 2\|\varphi\|^2_\infty\E\int_0^t|\bar{\mu}^N_s(h)-\bar{\mu}_s(h)|ds\,,\]
which tends to $0$, as $N\to\infty$.

We now consider the first term of the above right hand side.
We have
\begin{align} \label{eqn-mu-rec-diff1}
\E\int_0^t\int_0^1&\left[\varphi(t-s+H(\bar\mu^N_s,v))-\varphi(t-s+H(\bar{\mu}_s,v))\right]^2\bar{\mu}_s(h) dv ds \nonumber\\
&=\E\int_0^t\int_0^1\varphi^2(t-s+H(\bar\mu^N_s,v))\bar{\mu}_s(h) dv ds+ \int_0^t\int_0^1\varphi^2(t-s+H(\bar{\mu}_s,v))\bar{\mu}_s(h) dv ds \nonumber\\
&\quad-2\E\int_0^t\int_0^1\varphi(t-s+H(\bar\mu^N_s,v))\varphi(t-s+H(\bar{\mu}_s,v))\bar{\mu}_s(h) dv ds\\
&=\E\int_0^t\int_0^\infty\varphi^2(t-s+r)h(r)\bar{\mu}^N_s(dr)\frac{\bar{\mu}_s(h)}{\bar{\mu}_s^N(h)}ds
+\int_0^t\int_0^\infty\varphi^2(t-s+r)h(r)\bar{\mu}_s(dr)ds \nonumber\\
&\quad-2\E\int_0^t\int_0^\infty\varphi(t-s+H(\bar\mu^N_s,G_s(r))\varphi(t-s+r)h(r)\bar\mu_s(dr), \nonumber
\end{align}
where we have used $G_s(a) = \frac{\bar\mu_s(h {\bf1}_{[0,a]})}{\bar\mu_s(h)}$, 
hence $H(\bar\mu_s, v) = G^{-1}_s(v)$ and $H(\bar\mu_s,G_s(r)) =r$,
and we have done the change of variables $v=G_s(r)$, hence $dv=\frac{h(r)}{\bar\mu_s(h)}\bar\mu_s(dr)$.
It remains to show that, as $N\to\infty$,
\begin{align*}
\E\int_0^t\int_0^\infty\varphi^2(t-s+r)h(r)\bar{\mu}^N_s(dr)\frac{\bar{\mu}_s(h)}{\bar{\mu}_s^N(h)}ds&\to\int_0^t\int_0^\infty\varphi^2(t-s+r)h(r)\bar{\mu}_s(dr)ds,\\
\E\int_0^t\int_0^\infty\varphi(t-s+H(\bar\mu^N_s,G_s(r))\varphi(t-s+r)h(r)\bar\mu_s(dr)&\to\int_0^t\int_0^\infty\varphi^2(t-s+r)h(r)\bar{\mu}_s(dr)ds\,.
\end{align*}
Since $$\frac{\bar\mu_s(h)}{\bar\mu^N_s(h)}\int_0^\infty\varphi^2(t-s+r)h(r)\bar{\mu}^N_s(dr)\le\|\varphi\|_\infty^2\bar\mu_s(h)$$ for all $s\in[0,t]$  and $h$ is bounded, the first convergence follows from Lebesgue's dominated convergence theorem and the facts that $\bar\mu^N_s\Rightarrow\bar\mu_s$ for each $s$, and $h\in C_b(\R_+)$. 

The second convergence will follow from the fact that, as $N\to\infty$,
$H(\bar\mu^N_s,G_s(r))\to r$, $\bar\mu_s$ almost a.e., which follows from the fact that, given that $h\in C_b(\R_+)$, 
$\bar\mu_s^N(h{\bf1}_{[0,r]})\to\bar\mu_s(h{\bf1}_{[0,r]})$ for $\bar\mu_s$ almost every $r$.

 {\bf Step 2} We prove that for each $k\ge1$, $0\le t_1<\cdots<t_k$, and 
 $\varphi_1, \dots, \varphi_k\in H^1(\R_+)$,
 \[ (\tilde{\mu}^{rec,N}_{t_1}(\varphi_1),\ldots,\tilde{\mu}^{rec,N}_{t_k}(\varphi_k))\Rightarrow(\hat{\mu}^{rec}_{t_1}(\varphi_1),\ldots,\hat{\mu}^{rec}_{t_k}(\varphi_k)), \qinq \RR^k.\]
Applying Corollary 2.9 in \cite{pardoux-PRM-book},
we obtain that the joint cumulant of the variables  
$(\tilde{\mu}^{rec,N}_{t_1}(\varphi_1),\ldots, \\ \tilde{\mu}^{rec,N}_{t_k}(\varphi_k))$ equals $0$ for $k=1$, and otherwise is given as
\begin{align*}
& \kappa_k(\tilde{\mu}^{rec,N}_{t_1}(\varphi_1),\ldots,\tilde{\mu}^{rec,N}_{t_k}(\varphi_k)) \\
& =N^{1-k/2}\int_0^{t_1\wedge\cdots\wedge t_k} \int_0^\infty
\varphi_1(t_1-s + r)\times\cdots\times\varphi_k(t_k-s+ r) h(r)\bar\mu_s(dr) ds\,.
\end{align*}
The cumulants of order 1 are $0$, all cumulants of order $k\ge3$ tend to $0$, while for $k=2$,
\begin{align*}
\kappa_2(\tilde{\mu}^{rec,N}_{t_1}(\varphi_1),\tilde{\mu}^{rec,N}_{t_2}(\varphi_2))=\int_0^{t_1\wedge t_2}\int_0^\infty \varphi_1(t_1-s+r)\varphi_2(t_2-s+r) h(r)\bar\mu_s(dr) ds\,.
\end{align*}
Thus, we can conclude that the finite dimensional distributions of  $\hat{\mu}^{rec,N}_t(\varphi)$ converge towards
those of $\hat{\mu}^{rec}_t(\varphi)$.
\end{proof}

We will next prove tightness in $D$ of the sequence  $\{\hat{\mu}_{\cdot}^{rec,N}(\varphi)\},\ N\ge1\}$. But before doing so, let us establish a useful 
estimate of the fourth moment of our integrals w.r.t. $\overline{Q}_{rec}$.

\begin{lemma}\label{le:4thmoment}
Let the integrand $g(s,v,w)$ be $\mathcal{F}_s$--predictable and such that for some $T>0$,
\[ \E\int_0^T\int_0^\infty\int_0^1g^4(s,v,w)dsdvdw<\infty\,. \]
 Then for any $0\le t\le T$,
\begin{align}\label{4thmoment}
 \E\left[\left(\int_0^t\int_0^\infty\int_0^1g(s,v,w)\overline{Q}_{rec}(ds,dv,dw)\right)^4\right]&\le 7 e^{4t}\E\int_0^t\int_0^\infty\int_0^1 g^4(s,v,w)dsdvdw\,.
 \end{align}
\end{lemma}
\begin{remark}\label{remg}
We could prove the Lemma under the weaker assumption 
\[\int_0^T\int_0^\infty\int_0^1g^4(s,v,w)dsdvdw<\infty\ \text{ a.s.}\]
However, the result is useful only in the case where the right hand side of \eqref{4thmoment} is finite, hence also the left hand side, in which case the second moment  of the stochastic integral of $g$ w.r.t. $\overline{Q}_{rec}$ is finite, which implies our assumption.

We shall apply this lemma with some $g$'s which satisfy
\[ |g(s,v,w]\le C{\bf1}_{v\le Z_s},\]
where $C$ is a constant, and $Z_t$ is predictable and satisfies $\E\int_0^T Z_sds<\infty$ a.s., for all $T>0$, which implies the assumption of the lemma.
\end{remark}

\begin{proof}
For any $n\ge1$, let 
\begin{align*}
g_n(s,v,w)&=g(s,v,w)\wedge n\vee(-n){\bf1}_{v\le n}\,,\\
X_n(t)&=\int_0^t\int_0^\infty\int_0^1g_n(s,v,w)\overline{Q}_{rec}(ds,dv,dw),\\
X(t)&=\int_0^t\int_0^\infty\int_0^1g(s,v,w)\overline{Q}_{rec}(ds,dv,dw)\,.
\end{align*}
$t\mapsto X_n(t)$ has bounded variations and $X_n(t)$ has finite moments of any order. 
It follows from elementary computations that
\begin{align*}
|X_{n}(t)|^4&=4\int_0^t\int_0^\infty\int_0^1X_n^3({s^-})g_n(s,v,w)\overline{Q}_{rec}(ds,dv,dw)\\&\quad
+\int_0^t\int_0^\infty\int_0^1\left[(X_n({s^-})+g_n(s,v,w))^4-X_n^4({s^-})-4X_n^3({s^-})g_n(s,v,w)\right]Q_{rec}(ds,dv,dw)\,.\end{align*}
We then deduce that
\begin{align*}
\E(|X_n(t)|^4)&=\E\int_0^{t\wedge\tau_n}\int_0^\infty\int_0^1[6X_n^2(s)g_n^2(s,v,w)+4X_n(s)g_n^3(s,v,w)+g_n^4(s,v,w)]dsdvdw\\
&\le 4\E\int_0^{t}|X_n(s)|^4ds+7\E\int_0^t\int_0^\infty\int_0^1 g_n^4(s,v,w)dsdvdw,
\end{align*}
where we have used the two Young's inequalities $6X_n^2(s)g_n^2(s,v,w)\le  3X_n^4(s)+3g_n^4(s,v,w)$ and
$4X_n(s)g_n^3(s,v,w)\le X_n^4(s)+3g_n^4(s,v,w)$.
Now we deduce from Gronwall's Lemma that 
\begin{align*}
\E(|X_n(t)|^4)&\le7 e^{4t}\E\int_0^t\int_0^\infty\int_0^1 g_n^4(s,v,w)dsdvdw\\
&\le7 e^{4t}\E\int_0^t\int_0^\infty\int_0^1 g^4(s,v,w)dsdvdw\,. 
\end{align*}
The result follows from Fatou's Lemma, since $X_n(t)\to X(t)$ in mean square, as $n\to\infty$.
\end{proof}

\medskip

We now establish the tightness result.
\begin{lemma} \label{lem-tilde-mu-rec-tight}
For any $\varphi\in H^1(\R_+)$, the sequence $\{\hat{\mu}_{\cdot}^{rec,N}(\varphi)\}$ is tight in $D$.
\end{lemma}

\begin{proof}
We will apply the tightness criterion \eqref{3pts} to $\hat{\mu}_{t}^{rec,N}(\varphi)$ with the help of Lemma \ref{le:4thmoment}. For any $t_1\le t \le t_2\le T$, we can write
\[
\hat{\mu}_{t}^{rec,N}(\varphi) - \hat{\mu}_{t_1}^{rec,N}(\varphi) = \overline{Q}_{rec}(f_1), \qandq 
\hat{\mu}_{t_2}^{rec,N}(\varphi) - \hat{\mu}_{t}^{rec,N}(\varphi) = \overline{Q}_{rec}(f_2), 
\]
where
\begin{align*}
f_1(s,v,w) &= \frac{1}{\sqrt{N}} {\bf1}_{(t_1,t]}(s) \varphi(t-s+H(\bar\mu^N_{s^-},w)){\bf1}_{v\le \mu^N_{s^-}(h)}  \\
& \quad + \frac{1}{\sqrt{N}} {\bf1}_{[0,t_1]}(s) \Big(\varphi(t-s+H(\bar\mu^N_{s^-},w)) - \varphi(t_1-s+H(\bar\mu^N_{s^-},w)) \Big) {\bf1}_{v\le \mu^N_{s^-}(h)} \\
&=f_{1,1}(s,v,w)+f_{1,2}(s,v,w),
\end{align*}
and
\begin{align*}
f_2(s,v,w) &= \frac{1}{\sqrt{N}} {\bf1}_{(t,t_2]}(s) \varphi(t_2-s+H(\bar\mu^N_{s^-},w)){\bf1}_{v\le \mu^N_{s^-}(h)}  \\
& \quad + \frac{1}{\sqrt{N}} {\bf1}_{[0,t]}(s) \Big(\varphi(t_2-s+H(\bar\mu^N_{s^-},w)) - \varphi(t-s+H(\bar\mu^N_{s^-},w)) \Big) {\bf1}_{v\le \mu^N_{s^-}(h)}  \\
&=f_{2,1}(s,v,w)+f_{2,2}(s,v,w)\,.
\end{align*}
Below we shall apply Lemma \ref{le:4thmoment} for estimating the 4th moment of the four random variables
$\overline{Q}_{rec}(f_{1,1})$, $\overline{Q}_{rec}(f_{1,2})$, $\overline{Q}_{rec}(f_{2,1})$ and 
$\overline{Q}_{rec}(f_{2,2})$. Clearly $f°_{1,1}$, $f_{1,2}$, $f_{2,1}$ and $f_{2,2}$ satisfy the assumption in Remark
\ref{remg}. Note that in those integrals the same time $t$ (or $t_1$, or $t_2$) appears
 both as the upper bound of the integral, and in the integrand. This does not prevent us from using Lemma \ref{le:4thmoment}. Indeed, consider e.g. the term 
\[\overline{Q}_{rec}(f_{1,1})=\int_{t_1}^t\int_0^\infty\int_0^1\varphi(t-s+H(\bar\mu^N_{s^-},w)){\bf1}_{v\le \mu^N_{s^-}(h)}\overline{Q}_{rec}(ds,dv,dw)\,.\]
We consider the process
\[\int_0^r\int_0^\infty\int_0^1{\bf1}_{s>t_1}\varphi(t-s+H(\bar\mu^N_{s^-},w)){\bf1}_{v\le \mu^N_{s^-}(h)}\overline{Q}_{rec}(ds,dv,dw),\ r\in[0,t]\,.\]
This is really of the form 
\[ \int_0^r\int_0^\infty\int_0^1g(s,v,w)\overline{Q}_{rec}(ds,dv,dw),\]
to which we can apply It\^o's formula and the arguments in Lemma  \ref{le:4thmoment}, which allows us to conclude the result at $r=t$.

We want to estimate
\begin{align*}
\E\left[\overline{Q}_{rec}(f_1)^2\overline{Q}_{rec}(f_2)^2\right]&\le2\E\left[\overline{Q}_{rec}(f_{1})^2\overline{Q}_{rec}(f_{2,1})^2\right]\\&\quad+4\E\left[\overline{Q}_{rec}(f_{1,1})^2\overline{Q}_{rec}(f_{2,2})^2\right]+4\E\left[\overline{Q}_{rec}(f_{1,2})^2\overline{Q}_{rec}(f_{2,2})^2\right]\,. 
\end{align*}
We first estimate the first term on the last right--hand side: 
\begin{align*}
\E\left[\overline{Q}_{rec}(f_{1})^2\overline{Q}_{rec}(f_{2,1})^2\right]&=
\E\left[\overline{Q}_{rec}(f_{1})^2\E^{\mathcal{F}_t}\left\{\overline{Q}_{rec}(f_{2,1})^2\right\}\right]\\
&\le\|\varphi\|_\infty^2\E\left[\overline{Q}_{rec}(f_{1})^2\int_t^{t_2}\bar{\mu}^N_s(h)ds\right]\\
&\le\|\varphi\|_\infty^2\|h\|_\infty(t_2-t)\E\left[\overline{Q}_{rec}(f_{1})^2\right]\,, 
\end{align*}
where we have used the fact that $\bar{\mu}^N_s({\bf1})\le 1$ a.s. Now
\begin{align*}
\E\left[\overline{Q}_{rec}(f_{1})^2\right]&= \E\left[\overline{Q}_{rec}(f_{1,1})^2\right]+\E\left[\overline{Q}_{rec}(f_{1,2})^2\right]\\
&\le \|\varphi\|_\infty^2\|h\|_\infty(t-t_1)+\int_0^{t_1}\int_{\R_+}|\varphi(t-s+r-\varphi(t_1-s+r)|^2h(r)\bar{\mu}^N_s(dr)ds\\
&\le[\|\varphi \|_\infty^2+t_1 c_\varphi^2]\|h\|_\infty (t-t_1)\,. 
\end{align*}
We have shown that
\[ \E\left[\overline{Q}_{rec}(f_{1})^2\overline{Q}_{rec}(f_{2,1})^2\right]\le C(t_2-t_1)^2\,.\]
 Consider now the second term:
\begin{align*}
\E\left[\overline{Q}_{rec}(f_{1,1})^2\overline{Q}_{rec}(f_{2,2})^2\right]&\le
\sqrt{\E\left[\overline{Q}_{rec}(f_{1,1})^4\right]}\sqrt{\E\left[\overline{Q}_{rec}(f_{2,2})^4\right]}\\
&\le C_T\|h\|_\infty \frac{c_\varphi^2\|\varphi\|_\infty^2}{N}(t_2-t_1)^{3/2} \,. 
\end{align*}
Finally the third term  is bounded as follows:
\begin{align*}
\E\left[\overline{Q}_{rec}(f_{1,2})^2\overline{Q}_{rec}(f_{2,2})^2\right]&\le C_T\|h\|_\infty\frac{c_\varphi^4}{N}
(t_2-t_1)^2 \,.
\end{align*}
We now conclude the proof exactly as that of Lemma \ref{le-tilde-inf-tight}, in the sense that we verify Billingsley's condition (13.14)
with again $G(t)=t$ and this time $\alpha=3/2$.
\end{proof} 


%

\bigskip

\begin{proof} [{\bf Completing the proof for the convergence of $\hat{\mu}^N_t$  in Theorem \ref{thm-FCLT-mu}}]

We start with \eqref{eqn-hat-muN-rep},  and by taking the derivative with respect to $t$, we obtain
\begin{equation} \label{eqn-SPDE-prelimit} 
\begin{aligned}
\partial_t \hat{\mu}^N_t(\varphi) &=  \hat{\mu}^N_t(\varphi'-h\varphi)+\varphi(0)\left[\widehat{\Upsilon}^N_t 
+\frac{1}{\sqrt{N}}\frac{d}{dt}\int_0^t\int_0^\infty{\bf1}_{v\le N \overline{\Upsilon}^N(s^-)}\overline{Q}_{inf}(ds,dv)\right]  \\
& \quad - \frac{1}{\sqrt{N}}\frac{d}{dt}\int_0^t\int_0^\infty\int_0^1\varphi(H(\bar{\mu}^N_{s^-},w)){\bf1}_{v\le N\bar{\mu}^N_{s^-}(h)}\overline{Q}_{rec}(ds,dv,dw)\,.
\end{aligned} 
\end{equation}
Similar to \eqref{eqn-SPDE-solution}, we obtain that the solution to \eqref{eqn-SPDE-prelimit} can be written as 
\begin{align}\label{explicitN}
\hat{\mu}^N_t(\varphi) &= \int_0^\infty \varphi(t+\mfa) \frac{F^c(t+\mfa)}{F^c(\mfa)} \hat{\mu}^N_0(d \mfa) + \int_0^t \varphi(t-\mfa) F^c(t-\mfa)  \widehat{\Upsilon}^N(\mfa) d\mfa \nonumber\\
& \quad + \int_0^t \varphi(t-\mfa) F^c(t-\mfa)  \frac{1}{\sqrt{N}}\int_0^\infty{\bf1}_{v\le N \bar{\Upsilon}^N(\mfa^-)}\overline{Q}_{inf}(d\mfa,dv) \\
& \quad -\frac{1}{\sqrt{N}} \int_0^t \int_0^\infty \int_0^1 \varphi(t-s+ H(\bar{\mu}^N_{s^-},w))  \frac{F^c(t-s + H(\bar{\mu}^N_{s^-},w)) }{F^c(H(\bar{\mu}^N_{s^-},w))} {\bf1}_{v\le N\bar{\mu}^N_{s^-}(h)}\overline{Q}_{rec}(ds,dv,dw)\,. \nonumber
\end{align}
The first, third and fourth terms on the right hand side correspond to the limits $\hat\mu_t^0(\varphi)$, $\check{\mu}^{inf}_t(\varphi) $ and $\check{\mu}^{rec}_t(\varphi) $  defined in \eqref{eqn-hatmu0-def}, \eqref{eqn-check-mu-inf}, \eqref{eqn-check-mu-rec} respectively. For convenience, we denote these terms as $\hat\mu_t^{0,N}(\varphi)$, $\check{\mu}^{inf,N}_t(\varphi) $ and $\check{\mu}^{rec,N}_t(\varphi)$. To prove the convergence $\hat{\mu}^N_t(\varphi)  \Rightarrow \hat{\mu}_t(\varphi) $ in $D$ for the expression of $\hat{\mu}_t(\varphi)$ in \eqref{eqn-SPDE-solution}, we proceed in the following two steps: 

{\it Step (i)}:  define the processes
\begin{align}
\check{\check{\mu}}^{inf,N}_t(\varphi)  &=\int_0^t \varphi(t-\mfa) F^c(t-\mfa)  \frac{1}{\sqrt{N}}\int_0^\infty{\bf1}_{v\le N \bar{\Upsilon}(\mfa^-)}\overline{Q}_{inf}(d\mfa,dv) \,,\\
\check{\check{\mu}}^{rec,N}_t(\varphi)&=\frac{1}{\sqrt{N}}  \int_0^t \int_0^\infty \int_0^1 \varphi(t-s+ H(\bar{\mu}_{s^-},w))  \frac{F^c(t-s + H(\bar{\mu}_{s^-},w)) }{F^c(H(\bar{\mu}_{s^-},w))} {\bf1}_{v\le N\bar{\mu}_{s^-}(h)}\overline{Q}_{rec}(ds,dv,dw)\,,
\end{align}
and show the joint convergence of 
\[
\Big( \hat\mu_t^{0,N}(\varphi),  \int_0^t \varphi(t-\mfa) F^c(t-\mfa)  \widehat{\Upsilon}^N(\mfa) d\mfa, \check{\check{\mu}}^{inf,N}_t(\varphi),\check{\check{\mu}}^{rec,N}_t(\varphi)\Big )
\]
to 
\[
\Big( \hat\mu_t^{0}(\varphi),  \int_0^t \varphi(t-\mfa) F^c(t-\mfa)  \widehat{\Upsilon}(\mfa) d\mfa,  \check{\mu}^{inf}_t(\varphi), \check{\mu}^{rec}_t(\varphi)\Big )
\]
in $D^4$ as $N\to\infty$, and 

{\it Step (ii)}:  show that in probability, 
\begin{equation} \label{eqn-chb-muN-inf-diff}
\check{\mu}^{inf,N}_t(\varphi) - \check{\check{\mu}}^{inf,N}_t(\varphi) \to 0 \qinq D\,,
\end{equation}
and
\begin{equation} \label{eqn-chb-muN-rec-diff}
 \check{\mu}^{rec,N}_t(\varphi) - \check{\check{\mu}}^{rec,N}_t(\varphi) \to 0 \qinq D\,. 
\end{equation}

We prove the claim in the first step now. 
By a similar argument of Lemma \ref{lem-hat-mu0-conv}, we obtain $\hat\mu_t^{0,N}(\varphi) \to \hat\mu_t^{0}(\varphi)$ in $D$. 
Given the convergence of $ \widehat{\Upsilon}^N\Rightarrow  \widehat{\Upsilon}$ in $D$, by the continuous mapping theorem, we obtain the convergence 
\[
 \int_0^t \varphi(t-\mfa) F^c(t-\mfa)  \widehat{\Upsilon}^N(\mfa) d\mfa \Rightarrow 
  \int_0^t \varphi(t-\mfa) F^c(t-\mfa)  \widehat{\Upsilon}(\mfa) d\mfa \qinq D\,.
\]
 By a similar argument to that in Lemma \ref{lem-conv-hat-inf}, we obtain $\check{\check{\mu}}^{inf,N}_t(\varphi)\Rightarrow\check{\mu}^{inf}_t(\varphi)$ in $D$. By modifying the arguments for the proof of the convergence of $\tilde{\mu}^{rec,N}_t(\varphi)$ in the proof of Lemma \ref{lem-conv-hat-rec-fdd} and  \ref{lem-tilde-mu-rec-tight}, we obtain $\check{\check{\mu}}^{rec,N}_t(\varphi)\Rightarrow\check{\mu}^{rec}_t(\varphi)$ in $D$.
 Then the joint convergence follows from the independence of the driving random quantities $\hat\mu^N_0(d\mfa)$, $\overline{Q}_{inf}$ and $\overline{Q}_{rec}$ in the three terms, given the convergence of $ \widehat{\Upsilon}^N\Rightarrow  \widehat{\Upsilon}$ in $D$. 

We move to prove the claim in the second step. The claim in \eqref{eqn-chb-muN-inf-diff} follows from a similar argument as in Lemma \ref{le-diff-inf-tight} and that in \eqref{eqn-chb-muN-rec-diff} follows from slightly modifying the arguments in the proofs of Lemmas \ref{lem-conv-hat-rec-fdd} and \ref{lem-tilde-mu-rec-tight}.

Hence we have shown that 
$\hat{\mu}^N_t(\varphi)  \Rightarrow \hat{\mu}_t(\varphi) $ in $D$, where  $\hat{\mu}_t(\varphi)$ is given by \eqref{eqn-SPDE-solution}. 
It remains to show uniqueness of the solution of the SPDE \eqref{SPDE-CLT}, in order to conclude that the formula \eqref{eqn-SPDE-solution} 
can be identified with the solution of \eqref{SPDE-CLT}.
\end{proof}

\smallskip

\begin{proof} [{\bf Proof for the uniqueness of the SPDE solution to $\hat{\mu}_t$  in Theorem \ref{thm-FCLT-mu}}]
We will use the same argument as at the end of section \ref{sec-FLLN-proof}, exploiting duality with the same backward PDE. However, the regularities of both the forward and the backward PDE  are different.

We have that $\hat{\mu}\in L^2_{loc}(\R_+; (H(\R_+))')$.
Suppose that equation \eqref{SPDE-CLT} has more than one solution with that regularity. Then the difference of two solutions solves the PDE:
\begin{equation}\label{PDE0}
\langle u_t,\varphi\rangle=\int_0^t \langle u_s,\varphi'-h\varphi\rangle ds\,,
 \end{equation}
 for any $\varphi\in H^2_c(\R_+)$.
 
 Suppose now that $T>0$ and $\varphi$ is a function of $t$ and $\mfa$ such that 
 \begin{equation}\label{regul}
  \varphi,\, \partial_t\varphi,\, \partial_\mfa\varphi\in L^\infty([0,T];H^1_c(\R_+))\,.
  \end{equation}
 Then it is not hard to see that \eqref{PDE0} becomes
 \begin{equation}\label{duality}
  \langle u_t,\varphi_t\rangle =\int_0^t \langle u_s,\partial_t\varphi_s+\partial_\mfa\varphi_s-h\varphi_s \rangle ds\,, \quad 0<t<T\,.
 \end{equation}
 
 Consider again the adjoint backward PDE \eqref{eq:adjointPDE}, whose solution is still
\[ v(t,\mfa)=\frac{F^c(\mfa+T-t)}{F^c(\mfa)}g(\mfa+T-t)\,.\]

Assuming that $g\in C^2_c(\R_+)$ and $F$ has two derivatives $f$ and $f'$ which are locally bounded, that $F^c(\mfa)>0$ for all $\mfa>0$, we have that
$\mfa\mapsto v(t,\mfa)$ has compact support for all $t\in[0,T]$ and satisfies $v,\partial_\mfa v, \partial_tv, hv \in L^2([0,T];H^{1}_c(\R_+))$, and $\partial_tv+\partial_\mfa v-hv=0$, hence from \eqref{duality},  recalling $v(T,\mfa)=g(\mfa)$ in \eqref{eq:adjointPDE}, we have $\langle u_T,v_T\rangle  = \langle u_T,g\rangle=0$, and this holds true for any $g\in C^2_c(\R_+)$, hence $u_T=0$, for all $T\ge0$, from which the claimed uniqueness follows.
\end{proof}

\medskip

\section{On the convergence of $(\widehat{S}^N, \widehat{\mfF}^N)$} \label{sec-SF-proof}

The proof for the convergence of $(\widehat{S}^N, \widehat{\mfF}^N)$ has been established in \cite{PP2020-FCLT-VI}, with a different initial condition (without tracking the infection age of of the initially infected individuals). Here we only provide a sketch and highlight the differences. 

We first write  $\widehat{S}^N$ and $\widehat{\mfF}^N$ as the following: 
\begin{align*}
\widehat{S}^N(t) & =- \widehat{I}^N(0) - \widehat{S}^N_1(t)  -\int_0^t \widehat{\Upsilon}^N(s) ds,\\
\widehat{\mfF}^N(t) &= \int_0^{\infty} \bar\lambda(\mfa+t)  d \hat{\mu}^N_0(\mfa)   
+  \int_0^t  \bar\lambda(t-s) \widehat{\Upsilon}^N(s) ds  +  \widehat{\mfF}^N_{0,1}(t)+ \widehat{\mfF}^N_{0,2}(t) + \widehat{\mfF}^N_{1}(t) + \widehat{\mfF}^N_{2}(t)\,,  \\
\widehat{\Upsilon}^N(t) &=  \widehat{S}^N(s) \overline{\mfF}^N(s) + \bar{S}^N(s) \widehat{\mfF}^N(s)  
\end{align*}
where
\begin{align}
\hat{S}_1^N(t) & :=  \frac{1}{\sqrt{N}} \int_0^t \int_0^\infty \bone_{v \le \Upsilon^N(s^-) } \overline{Q}_{inf}(ds,dv), \nonumber \\
  \widehat{\mfF}^N_{0,1}(t) &:=  \int_0^{\infty} \bar\lambda(\mfa+t)   \hat{\mu}^N_0(d\mfa) \nonumber \\
  &=   \frac{1}{\sqrt{N}} \sum_{j=1}^{I^N(0)} \Big( \bar\lambda(\tilde\tau_{j,0}+t) \  -  \int_0^\infty \bar\lambda(\mfa+t) \bar{\bar{\mu}}_0(d\mfa) \Big)\,, \nonumber \\
\widehat{\mfF}^N_{0,2}(t) &:= \frac{1}{\sqrt{N}} \sum_{j=1}^{I^N(0)} \left( \lambda_{-j}( \tilde{\tau}_{j,0}+ t ) - \bar\lambda( \tilde{\tau}_{j,0}+ t ) \right) , \label{eqn-whFN02}\\
 \widehat{\mfF}^N_{1}(t) &:= \frac{1}{\sqrt{N}} \bigg( \sum_{i=1}^{A^N(t)}  \bar\lambda(t- \tau^N_i) -  \int_0^t \bar\lambda(t-s)  \Upsilon^N(s) ds \bigg) \nonumber\\
 &=  \int_0^t \bar\lambda(t-s) d \hat{S}_1^N(s) , \nonumber \\
\widehat{\mfF}^N_{2}(t) &:= \frac{1}{\sqrt{N}} \sum_{i=1}^{A^N(t)} \left( \lambda_i(t- \tau^N_i)
 - \bar\lambda(t- \tau^N_i)   \right).  \label{eqn-whFN2} 
\end{align}


Observe that $\hat{S}_1^N(t)$ is a square-integrable martingale with respect to the filtration $\sF^N= \{\sF^N(t): t\ge0\}$ where $
\sF^N(t) :=  \sigma\big\{S^N(0), I^N(0),\, \tilde{\tau}_{j}, j =1,\dots, I^N(0)\big\}  \vee \big\{ \lambda_i(\cdot)_{i\in \ZZ \setminus\{0\}}\big \}\vee\big\{Q_{inf}(s,v), s \le t, v \in \RR_+\big\} 
$
with the quadratic variation  $\langle  \hat{S}_1^N \rangle(t) =   \int_0^t \overline\Upsilon^N(s)ds $ for $t \ge 0$. 
Given the convergence of $\overline\Upsilon^N \to \overline \Upsilon$ in $D$ in probability as $N\to\infty$,
 we obtain $\hat{S}_1^N\Rightarrow \hat{S}_1$ in $D$, where $\hat{S}_1(t)$ is given in Definition \ref{def-Gaussian}.  
This can be done by establishing the convergence of finite-dimensional distributions using the cumulants formula as in the proof of Lemma \ref{lem-conv-hat-inf} and then a similar tightness argument in the proof of Lemma \ref{le-tilde-inf-tight}. It can be also proved as done in  Section 3.6 of \cite{PP2020-FCLT-VI}.

The convergence of $ \widehat{\mfF}^N_{0,1} \Rightarrow  \widehat{\mfF}_{0,1}$ in $D$ can be proved in the same way as in Lemma \ref{lem-hat-mu0-conv}. Note that this requires that condition \eqref{eqn-lambda-inc} in Assumption \ref{AS-lambda}, that is, the function $\bar\lambda$ is H{\"o}lder with coefficient $\alpha>1/4$.
The convergences of $\widehat{\mfF}^N_{1}(t)$ and $\widehat{\mfF}^N_{2}(t)$ follow from the same arguments as in  Lemmas 3.5 and 3.6 of \cite{PP2020-FCLT-VI}. It only remains to prove the convergence of $\widehat{\mfF}^N_{0,2}$.

We will need the following bounds on the increments of the infectivity functions (which is Lemma 3.4 in \cite{PP2020-FCLT-VI}).

 \begin{lemma} \label{lem-barlambda-inc-bound}
For $t\ge s\ge 0$,
\begin{align*}
  \big|\lambda_i(t) - \lambda_i(s) \big| \le \phi(t-s) + \lambda^*  \sum_{\ell=1}^k \bone_{s < \zeta_i^\ell \le t}\,, \text { and}
 \end{align*}
 \begin{align*}
 |\bar\lambda(t) -\bar \lambda(s)| \le \phi(t-s) + \lambda^* \sum_{\ell=1}^k (F_\ell(t) - F_\ell(s))\,. 
 \end{align*}

 \end{lemma}

\begin{lemma}
Under Assumptions \ref{AS-g} and \ref{AS-lambda}, 
\begin{align}
\widehat{\mfF}^N_{0,2}\RA \widehat{\mfF}_{0,2} \qinq D \qasq N \to \infty,
\end{align}
where  the limit $ \widehat{\mfF}_{0,2}$ is a continuous Gaussian process as given in Definition \ref{def-Gaussian2}. 
\end{lemma}

\begin{proof}
We first prove the convergence of finite-dimensional distributions, that is, for any $l\ge 1$,
\begin{equation} \label{eqn-hat-sI-01-jontconv}
\widehat{\mfF}^N_{0,2}(t_1), \dots, \widehat{\mfF}^N_{0,2}(t_l)) \RA 
(\widehat{\mfF}_{0,2}(t_1), \dots, \widehat{\mfF}_{0,2}(t_l)) \qinq \RR^l \qasq N \to \infty.
\end{equation}
We start with $l=1$, and consider the convergence of $\widehat{\mfF}^N_{0,2}(t)\RA \widehat{\mfF}_{0,2}(t))$ in $\RR$ as $N\to \infty$. 
By the continuity theorem, it suffices to show the characteristic function of $\widehat{\mfF}^N_{0,2}(t)$ converges to that of $\widehat{\mfF}_{0,2}(t)$, denoted by $\varphi^{\mfF,N}_{0,2}(\theta)$ and  $\varphi^{\mfF}_{0,2}(\theta)$, respectively.  Let $\sG^{N}_0=\sigma\big(\tilde\tau_{j,0}, j=1,\dots, I^N(0)\big)$. 
We have
\begin{align*}
\varphi^{\mfF,N}_{0,2}(\theta) &= \E\big[\exp\big(i \theta \widehat{\mfF}^N_{0,2}(t) \big)\big] = \E\big[\E\big[\exp\big(i \theta \widehat{\mfF}^N_{0,2}(t) \big)| \sG^{N}_0 \big] \big]  \\
& =  \E\Bigg[E\Bigg[ \prod_{j=1}^{I^N(0)}\exp\bigg(i \theta \frac{1}{\sqrt{N}} \left( \lambda_{-j}( \tilde{\tau}_{j,0}+ t ) - \bar\lambda( \tilde{\tau}_{j,0}+ t ) \right)  \bigg)\bigg| \sG^{N}_0\Bigg] \Bigg] \\
& =  \E\Bigg[\prod_{j=1}^{I^N(0)} \Bigg(  1- \frac{\theta^2}{2N} v(\tilde{\tau}_{j,0}+ t ) + o(N^{-1}) \Bigg) \Bigg],
\end{align*}
and
\begin{align*}
\varphi^{\mfF}_{0,2}(\theta) &=  \E\big[\exp\big(i \theta \widehat{\mfF}_{0,2}(t) \big)\big]= \exp\bigg(- \frac{\theta^2}{2}\int_0^{\infty} v(\mfa+t) \bar{\mu}_0(d\mfa)\bigg).
\end{align*}
Thus,
\begin{align*}
& \big|\varphi^{\mfF,N}_{0,2}(\theta) - \varphi^{\mfF}_{0,2}(\theta)  \big| \\
& \le  \E\Bigg[ \Bigg|\prod_{j=1}^{I^N(0)} \Bigg(  1- \frac{\theta^2}{2N} v(\tilde{\tau}_{j,0}+ t ) + o(N^{-1}) \Bigg)  -  \prod_{j=1}^{I^N(0)} \exp\Bigg( - \frac{\theta^2}{2N} v(\tilde{\tau}_{j,0}+t)  \Bigg) \Bigg|  \Bigg] \\
& \quad +  \Bigg| \E \Bigg[\exp\Bigg( - \frac{\theta^2}{2}  \int_0^{\infty} v(\mfa+t) d \bar{\mu}^N_0(d\mfa) \Bigg) \Bigg] - \exp\bigg(- \frac{\theta^2}{2}\int_0^{\infty} v(\mfa+t) d \bar{\mu}_0(d\mfa)\bigg) \Bigg|. 
\end{align*}
We have
 \begin{align*}
 \prod_{j=1}^{I^N(0)} \Bigg(  1- \frac{\theta^2}{2N} v(\tilde{\tau}_{j,0}+ t ) + o(N^{-1}) \Bigg)&=\prod_{j=1}^{I^N(0)}\exp\Bigg( - \frac{\theta^2}{2N} v(\tilde{\tau}_{j,0}+t)  +o(N^{-1})\Bigg)\\
& =\exp\Bigg( -\sum_{j=1}^{I^N(0)}\frac{\theta^2}{2N} v(\tilde{\tau}_{j,0}+t)+o(1)\Bigg)\,,
 \end{align*}
 hence the first term tends to $0$, as $N\to\infty$.

 The second term convergence to zero by the convergence $\bar{\mu}^N_0(\cdot) \to \bar{\mu}_0(\cdot)$ in probability as $N\to \infty$ and the boundedness of the variance function $v(\cdot)$. 
 Thus we have shown that 
 \begin{equation} \label{eqn-char-sI-conv}
 \big|\varphi^{\mfF,N}_{0,2}(\theta) - \varphi^{\mfF}_{0,2}(\theta)  \big| \to 0 \qasq N \to \infty, 
 \end{equation} 
  and thus,  
 $\widehat{\mfF}^N_{0,2}(t)\RA \widehat{\mfF}_{0,2}(t)$ as $N\to\infty$. 

A straightforward generalization implies the convergence of finite dimensional distributions in \eqref{eqn-hat-sI-01-jontconv}.  For instance, for $l=2$, we have
\begin{align*}
&  \E\big[\exp\big(i \theta_1 \widehat{\mfF}^N_{0,2}(t_1) + i \theta_2 \widehat{\mfF}^N_{0,2}(t_2)  \big)\big] \\
& = \E\Bigg[\prod_{j=1}^{I^N(0)} \Bigg(  1- \frac{1}{2N}  \Big( \theta_1^2 v(\tilde{\tau}_{j,0}+ t_1) +  \theta_2^2 v(\tilde{\tau}_{j,0}+ t_2) + 2 \theta_1 \theta_2 v(\tilde{\tau}_{j,0}+ t_1, \tilde{\tau}_{j,0}+ t_2)  \Big)  + o(N^{-1}) \Bigg) \Bigg],
\end{align*}
and
\begin{align*}
&  \E\big[\exp\big(i \theta_1 \widehat{\mfF}_{0,2}(t_1) + i \theta_2 \widehat{\mfF}_{0,2}(t_2)  \big)\big] \\
& =  \exp\bigg(- \frac{1}{2}\int_0^{\infty} \big(\theta_1^2 v(y+t_1) + \theta_2^2 v(y+t_2) +2\theta_1\theta_2 v(y+t_1, y+t_2) \big)    \bar{\mu}_0(dy)\bigg).
\end{align*}
Then the claim follows similarly. 

We next prove that for $0 <r<s<t$, 
\begin{align*}
\P\Big(\big| \widehat{\mfF}^N_{0,2}(r) -  \widehat{\mfF}^N_{0,2}(s) \big| \wedge\big| \widehat{\mfF}^N_{0,2}(s) -  \widehat{\mfF}^N_{0,2}(t) \big| \ge \ep \Big) \le  \frac{1}{\ep^4} (G(t) - G(r))^2
\end{align*}
for some nondecreasing and continuous function $G$. 
It suffices to show that 
\begin{equation} \label{eqn-hat-sI-inc-4-p}
\E\Big[\big| \widehat{\mfF}^N_{0,2}(t) -  \widehat{\mfF}^N_{0,2}(s) \big|^4  \Big]  \le   (G(t) - G(s))^2. 
\end{equation}
Let $\tilde{\lambda}_{-j}(t) := \lambda_{-j}(t) - \bar\lambda(t)$. 
We have
\begin{align} \label{eqn-hat-sI-inc-4-p1}
& \E\Big[\big| \widehat{\mfF}^N_{0,2}(t) -  \widehat{\mfF}^N_{0,2}(s) \big|^4  \Big] \non \\
& = \E\Bigg[\bigg|  \frac{1}{\sqrt{N}} \sum_{j=1}^{I^N(0)} \left( \tilde\lambda_{-j}( \tilde{\tau}_{j,0}+ t )  -  \tilde\lambda_{-j} ( \tilde{\tau}_{j,0}+ s ) \right)  \bigg|^4   \Bigg] \non \\
& \le \E\Bigg[ \frac{1}{N^2} \sum_{j=1}^{I^N(0)}  \left( \tilde\lambda_{-j}( \tilde{\tau}_{j,0}+ t ) - \tilde\lambda_{-j}( \tilde{\tau}_{j,0}+ s )  \right)^4   \Bigg] \non \\
& \quad +\E\Bigg[ \frac{6}{N^2} \sum_{j,j'=1,j\neq j'}^{I^N(0)} \left( \tilde\lambda_{-j}( \tilde{\tau}_{j,0}+ t ) -\tilde\lambda_{-j}( \tilde{\tau}_{j,0}+ s )  \right)^2   \left( \tilde\lambda_{-j'}( \tilde{\tau}_{j',0}+ t ) -\tilde\lambda_{-j'}( \tilde{\tau}_{j',0}+ s ) \right)^2  \Bigg] \non\\
& \le \E\Bigg[ \frac{3}{N^2} \Bigg( \sum_{j=1}^{I^N(0)} \left( \tilde\lambda_{-j}( \tilde{\tau}_{j,0}+ t ) -\tilde\lambda_{-j}( \tilde{\tau}_{j,0}+ s )  \right)^2  \Bigg)^2  \Bigg] \,. 
\end{align}

By Lemma \ref{lem-barlambda-inc-bound}, we have
\begin{align*}
| \tilde\lambda_{-j}( t ) -\tilde\lambda_{-j}(  s ) |^2
& \le 2 | \lambda_{-j}( t ) -\lambda_{-j}(  s ) |^2 +2 | \bar\lambda( t ) -\bar\lambda(  s ) |^2 \\
& \le  8 \phi(t-s)^2 + 4 (\lambda^*)^2  \bigg(  \sum_{\ell=1}^k \bone_{s < \zeta_{-j}^\ell \le t} \bigg)^2 + 4 (\lambda^*)^2 \bigg(  \sum_{\ell=1}^k (F_\ell(t) - F_\ell(s)) \bigg)^2.  
\end{align*}
Thus, 
\begin{align*}
& \frac{1}{N^2} \E\Bigg[  \Bigg( \sum_{j=1}^{I^N(0)} \left( \tilde\lambda_{-j}( \tilde{\tau}_{j,0}+ t ) -\tilde\lambda_{-j}( \tilde{\tau}_{j,0}+ s )  \right)^2  \Bigg)^2  \Bigg] \\
& \le \frac{1}{N^2} \E\Bigg[  \Bigg(
8 \phi(t-s)^2 I^N(0) +  \sum_{j=1}^{I^N(0)} \bigg(  4 (\lambda^*)^2  \bigg(  \sum_{\ell=1}^k \bone_{\tilde{\tau}_{j,0}+s < \zeta_{-j}^\ell \le \tilde{\tau}_{j,0}+t} \bigg)^2  \\
& \qquad \qquad + 4 (\lambda^*)^2 \bigg(  \sum_{\ell=1}^k (F_\ell(\tilde{\tau}_{j,0}+t) - F_\ell(\tilde{\tau}_{j}+s)) \bigg)^2 \bigg)  \Bigg)^2  \Bigg] \\
&\le  \frac{128}{N^2}  \phi(t-s)^4 I^N(0)^2   \\
& \qquad +  \frac{32 (\lambda^*)^4 }{N^2}  \E\Bigg[  \Bigg(  \sum_{j=1}^{I^N(0)} \bigg(  \bigg(  \sum_{\ell=1}^k \bone_{\tilde{\tau}_{j,0}+s < \zeta_{-j}^\ell \le \tilde{\tau}_{j,0}+t} \bigg)^2  + \bigg(  \sum_{\ell=1}^k (F_\ell(\tilde{\tau}_{j,0}+t) - F_\ell(\tilde{\tau}_{j}+s)) \bigg)^2 \bigg)  \Bigg)^2  \Bigg] \\
& \le  128  \phi(t-s)^4 \bar{I}^N(0)^2   \\
& \qquad +  \frac{64 (\lambda^*)^4 }{N^2}  \E\Bigg[  I^N(0)  \sum_{j=1}^{I^N(0)} \bigg(  \bigg(  \sum_{\ell=1}^k \bone_{\tilde{\tau}_{j,0}+s < \zeta_{-j}^\ell \le \tilde{\tau}_{j,0}+t} \bigg)^4  + \bigg(  \sum_{\ell=1}^k (F_\ell(\tilde{\tau}_{j,0}+t) - F_\ell(\tilde{\tau}_{j}+s)) \bigg)^4 \bigg)   \Bigg] \\
& \le 128  \phi(t-s)^4 +   64 (\lambda^*)^4  \E\Bigg[ \frac{1}{N}  \sum_{j=1}^{I^N(0)}   \bigg(  \sum_{\ell=1}^k \bone_{\tilde{\tau}_{j,0}+s < \zeta_{-j}^\ell \le \tilde{\tau}_{j,0}+t} \bigg)^4  \Bigg] \\
& \qquad +  64 (\lambda^*)^4   \E\Bigg[  \int_0^{\infty} \bigg(  \sum_{\ell=1}^k (F_\ell(y+t) - F_\ell(y+s)) \bigg)^4    \bar{\mu}^N_0(dy) \Bigg]\,.
\end{align*}
Here the first expectation is bounded by
\begin{align*}
&     \E\Bigg[ \frac{1}{N} \sum_{j=1}^{I^N(0)}   \bigg(  \sum_{\ell=1}^k \bone_{ \tilde{\tau}_{j,0}+s < \zeta_{-j}^\ell \le  \tilde{\tau}_{j,0}+ t}  +  14\sum_{\ell\neq \ell'} \bone_{ \tilde{\tau}_{j,0}+s < \zeta_{-j}^\ell \le  \tilde{\tau}_{j,0}+ t} \bone_{ \tilde{\tau}_{j,0}+s < \zeta_{-j}^{\ell'} \le  \tilde{\tau}_{j,0}+ t}  \bigg)    \Bigg] \non \\
& =  \E\Bigg[ \frac{1}{N}  \sum_{j=1}^{I^N(0)}   \bigg(  \sum_{\ell=1}^k ( F_\ell(\tilde{\tau}_{j,0}+t) - F_\ell( \tilde{\tau}_{j,0}+s)  \non \\
& \qquad \qquad  \qquad  + 14 \sum_{\ell\neq\ell'} ( F_\ell(\tilde{\tau}_{j,0}+t) - F_\ell( \tilde{\tau}_{j,0}+s) ( F_{\ell'}(\tilde{\tau}_{j,0}+t) - F_{\ell'}( \tilde{\tau}_{j,0}+s) \bigg)  \Bigg] \non \\
& =  \E\Bigg[ \int_0^{\infty}   \bigg(  \sum_{\ell=1}^k ( F_\ell(y+t) - F_\ell( y+s)  \non   \\
& \qquad \qquad  \qquad  + 14 \sum_{\ell\neq\ell'} ( F_\ell(y+t) - F_\ell(y+s) ( F_{\ell'}(y+t) - F_{\ell'}( y+s) \bigg)   \bar{\mu}^N_0(dy) \Bigg]. 
\end{align*}

%

Thus, combining the above, we obtain that there exist some $C',C''$ such that 
\begin{align} \label{eqn-hat-sI-inc-4-p7}
& \E\Big[\big| \widehat{\mfF}^N_{0,2}(t) -  \widehat{\mfF}^N_{0,2}(s) \big|^4  \Big] \non \\
&\le C'\frac{1}{N}\phi(t-s)^4 + C' \frac{1}{N}  \E\bigg[ \int_0^{\infty}   \bigg(  \sum_{\ell=1}^k ( F_\ell(y+t) - F_\ell( y+s)) \bigg)  \bar{\mu}^N_0(dy) \bigg] \non \\
& \quad + C'' \phi(t-s)^4 + C'' \E\bigg[ \int_0^{\infty}   \bigg(  \sum_{\ell=1}^k ( F_\ell(y+t) - F_\ell( y+s)) \bigg)  \bar{\mu}^N_0(dy) \bigg]. 
\end{align}
Under Assumptions \ref{AS-lambda},  we have $\phi(t-s)^4 \le C^4 |t-s|^{4\alpha}$ and if $F_\ell$ satisfies the H{\"o}lder continuity condition, 
\begin{align*}
& \E\bigg[ \int_0^{\infty}   ( F_\ell(y+t) - F_\ell( y+s)) \bar{\mu}^N_0(dy) \bigg]  \le C(t-s)^{1/2+\theta} \E[\bar{\mu}^N_0(\mds)] \le  C(t-s)^{1/2+\theta}. 
\end{align*}
and if $F_\ell$ satisfies the discrete condition, say $F_\ell= \sum_i a_i^\ell \bone_{t \ge t_i^\ell}$ for $\sum_i a_i^\ell=1$ and $t_i^\ell < t_{i+1}^\ell$, then
\begin{align*}
 \E\bigg[ \int_0^{\infty}   ( F_\ell(y+t) - F_\ell( y+s)) \bar{\mu}^N_0(dy) \bigg]   
&= \E\bigg[ \int_0^{\infty}  \sum_i a_i^\ell  \bone_{s+y < t_i^\ell \le t+y}  \bar{\mu}^N_0(dy) \bigg]  \\
& \le  C(t-s)  \E[\bar{\mu}^N_0(\mds)] \le  C (t-s) 
\end{align*}
for some constant $C>0$. 
Thus, using the above estimates,  by Theorem 4.1 in \cite{PP2020-FCLT-VI}, we can conclude that \eqref{eqn-hat-sI-inc-4-p} holds.  
\end{proof}

\medskip

 \section{On the convergence of $(\widehat{I}^N, \widehat{R}^N)$} \label{sec-IR-proof}
 
 By the first expression in  \eqref{eqn-In-rep}, the convergence of $\widehat{I}^N_t$ follows from directly from that of $\hat\mu^N_t$ and hence the expression of the limit  $\widehat{I}_t$ in \eqref{eqn-wh-I} from the limit  $\hat\mu_t$ in \eqref{eqn-SPDE-solution}.
  As discussed in Remark \ref{rem-wh-I}, the limit $\widehat{I}_t$ has an equivalent in distribution expression as given in \eqref{eqn-wh-I-01}. It can be derived from an alternative decomposition from the second expression of $\hat{\mu}^N_t(\varphi) $ in \eqref{eqn-wh-muN-def} and hence, an analogous decomposition for  $I^N_t$ in \eqref{eqn-In-rep}.

The CLT-scaled process $\hat{\mu}^N_t(\varphi) $ in \eqref{eqn-wh-muN-def} can be written as 
\begin{align} \label{eqn-wh-mu-rep01}
\hat{\mu}^N_t(\varphi) =
\int_0^\infty \varphi(\mfa+t) \frac{F^c(t+\mfa)}{F^c(\mfa)}  \hat\mu^N_0(d\mfa)  + 
  \int_0^t \varphi(t-s) F^c(t-s)  \widehat\Upsilon^N(s)ds +   \hat{\mu}^{N,0}_t(\varphi) + \hat{\mu}^{N,1}_t(\varphi) \,,
\end{align}
where
\begin{equation}\label{eqn-wh-mu0-rep}
 \hat{\mu}^{N,0}_t(\varphi)  = \frac{1}{\sqrt{N}}\sum_{j=1}^{I^N(0)}\Big({\bf1}_{\eta^0_{j} >t} - 
 \frac{F^c(t+\tilde{\tau}_{j,0})}{F^c(\tilde{\tau}_{j,0})}  \Big) \varphi(\tilde{\tau}_{j,0}+t)\,,
\end{equation}
and
\begin{equation}\label{eqn-wh-mu1-rep}
 \hat{\mu}^{N,1}_t(\varphi) = 
\frac{1}{\sqrt{N}}\int_0^t  \int_0^\infty \int_0^\infty {\bf1}_{\eta>t-s} {\bf1}_{u < \Upsilon^N(s^-)} \varphi(t-s)  \overline{Q}(ds,d \eta, du) \,.
\end{equation}
Using the fact that $\widehat{I}^N(t) =\hat{\mu}^N_t(\mds)$, we obtain 
\begin{align} \label{eqn-wh-I-rep01}
\widehat{I}^N_t =
\int_0^\infty \frac{F^c(t+\mfa)}{F^c(\mfa)}  \hat\mu^N_0(d\mfa)  + 
  \int_0^t F^c(t-s)  \widehat\Upsilon^N(s)ds +   \widehat{I}^{N,0}_t +  \widehat{I}^{N,1}_t \,,
\end{align}
where
\begin{equation}\label{eqn-wh-In0-rep}
 \widehat{I}^{N,0}_t = \frac{1}{\sqrt{N}}\sum_{j=1}^{I^N(0)}\Big({\bf1}_{\eta^0_{j} >t} - 
 \frac{F^c(t+\tilde{\tau}_{j,0})}{F^c(\tilde{\tau}_{j,0})}  \Big) \,,
\end{equation}
and
\begin{equation}\label{eqn-wh-In1-rep}
 \widehat{I}^{N,1}_t= 
\frac{1}{\sqrt{N}}\int_0^t  \int_0^\infty \int_0^\infty {\bf1}_{\eta>t-s} {\bf1}_{u < \Upsilon^N(s^-)}   \overline{Q}(ds,d \eta, du) \,.
\end{equation}
It is an easy extension to the proof of   \cite[Theorem 2.4]{PP2020-FCLT-VI} in order to prove the weak convergence of $\widehat{I}^N_t$ to $\widehat{I}_t$ in \eqref{eqn-wh-I-01}. More generally, it can be also shown that for any $\varphi\in H^1(\RR_+)$, $\hat{\mu}^N_t(\varphi) \Rightarrow \hat{\mu}_t(\varphi) $ in $D$ where 
\begin{align} \label{eqn-wh-mu-01}
\hat{\mu}_t(\varphi) =
\int_0^\infty \varphi(\mfa+t) \frac{F^c(t+\mfa)}{F^c(\mfa)}  \hat\mu_0(d\mfa)  + 
  \int_0^t \varphi(t-s) F^c(t-s)  \widehat\Upsilon(s)ds +   \hat{\mu}^{0}_t(\varphi) +  \hat{\mu}^{1}_t(\varphi) \,,
\end{align}
with $\hat{\mu}^{0}_t(\varphi)$ and $\hat{\mu}^{1}_t(\varphi)$ being independent and continuous centered Gaussian processes whose covariance functions are given as follows:
 for $t,t'\ge 0$ and $\varphi, \psi \in H^1(\R_+)$, 
\begin{equation} \label{eqn-wh-mu0-cov}
\Cov(\hat{\mu}^{0}_t(\varphi), \hat{\mu}^{0}_{t'}(\psi) ) = 
\int_0^\infty \Big(\frac{F^c(t\vee t'+\mfa)}{F^c(\mfa)} - \frac{F^c(t+\mfa)}{F^c(\mfa)}\frac{F^c(t'+\mfa)}{F^c(\mfa)} \Big) \varphi(\mfa+t) 
\psi(\mfa+t')\bar\mu_0(d\mfa) \,,
\end{equation}
and
\begin{equation} \label{eqn-wh-mu1-cov}
\Cov(\hat{\mu}^{1}_t(\varphi), \hat{\mu}^{1}_{t'}(\psi) )  = \int_0^{t\wedge t'} F^c(t\vee t'-s) \overline\Upsilon(s) \varphi(t-s)\psi(t'-s)ds \,. 
\end{equation}

It turns out that the last two terms $\check{\mu}^{inf}_t(\varphi) + \check{\mu}^{rec}_t(\varphi)$ of  $\hat\mu_t$ in \eqref{eqn-SPDE-solution} and 
 $\hat{\mu}^{0}_t(\varphi) +  \hat{\mu}^{1}_t(\varphi)$ in \eqref{eqn-wh-mu-01} have the same law. 
We verify that the variances of these expressions are equal (the covariance functions can be also easily checked), and hence provide a justification for the claim in Remark  \ref{rem-wh-I}. 
Recall the variance formulas for $\check{\mu}^{inf}_t(\varphi)$ and $\check{\mu}^{rec}_t(\varphi)$ in \eqref{eqn-check-mu-inf-var}  and \eqref{eqn-check-mu-rec-var}, respectively and the expression of $  \bar\mu_t(d\mfa) $ in \eqref{eqn-barmu}. Then we can write 
\begin{align*}
\Var(\check{\mu}^{rec}_t(\varphi))
&=  \int_0^t\int_0^\infty\varphi(t-s+\mfa)^2 \Big( \frac{F^c(t-s+\mfa)}{F^c(\mfa)}\Big)^2  h(\mfa)  \Big( {\bf1}_{\mfa<s}  F^c(\mfa)  \overline\Upsilon(s-\mfa)d \mfa  \\
& \qquad \qquad \qquad  + {\bf1}_{\mfa \ge s}  \frac{F^c(\mfa)}{F^c(\mfa-s)}  \bar\mu_0(d\mfa -s)\Big)  ds\,.
\end{align*}
Then the first term is equal to
  \begin{align*}
&  \int_0^t\int_0^s\varphi(t-s+\mfa)^2 \Big( \frac{F^c(t-s+\mfa)}{F^c(\mfa)}\Big)^2    f(\mfa)     \overline\Upsilon(s-\mfa)d \mfa  ds \\
& = \int_0^t \int_\mfa^t \varphi(t-s+\mfa)^2 \Big( \frac{F^c(t-s+\mfa)}{F^c(\mfa)}\Big)^2    f(\mfa)     \overline\Upsilon(s-\mfa)  ds d \mfa \\
& =  \int_0^t \int_\mfa^t \varphi(t-s+\mfa)^2 F^c(t-s+\mfa)^2    \overline\Upsilon(s-\mfa)  ds  \Big(\frac{1}{F^c(\mfa)}\Big)'d\mfa    \\
&= \int_0^t \int_0^{t-\mfa} \varphi(t-v)^2 F^c(t-v)^2    \overline\Upsilon(v)  dv  \Big(\frac{1}{F^c(\mfa)}\Big)'d\mfa    \\
&= \int_0^{t-\mfa} \varphi(t-v)^2 F^c(t-v)^2    \overline\Upsilon(v)  dv \Big(\frac{1}{F^c(\mfa)}\Big) \Big|_{\mfa=0}^t  \\
& \qquad -  \int_0^t   \Big(\frac{1}{F^c(\mfa)}\Big) d_\mfa \int_0^{t-\mfa} \varphi(t-v)^2 F^c(t-v)^2   \overline\Upsilon(v) dv\\
&=  - \int_0^{t} \varphi(t-v)^2 F^c(t-v)^2    \overline\Upsilon(v) dv + 
 \int_0^t   \varphi(\mfa)^2 F^c(\mfa)  \overline\Upsilon(t-\mfa) d\mfa \,. 
\end{align*}
So summing this with  $\check{\mu}^{inf}_t(\varphi)$  in \eqref{eqn-check-mu-inf-var}  will give us 
$\Var(  \hat{\mu}^{1}_t(\varphi) ) = \int_0^{t} \varphi(t-s)^2 F^c(t-s) \overline\Upsilon(s) ds$.
Next, the second term in the expression above for $\Var(\check{\mu}^{rec}_t(\varphi))$ is equal to 
  \begin{align*}
&\int_0^t\int_s^\infty\varphi(t-s+\mfa)^2 \Big( \frac{F^c(t-s+\mfa)}{F^c(\mfa)}\Big)^2  \frac{f(\mfa)}{F^c(\mfa-s)}  \bar\mu_0(d\mfa -s)  ds \\
&=  \int_0^t \int_0^\infty \varphi(t+v)^2 \Big( \frac{F^c(t+v)}{F^c(v+s)}\Big)^2  \frac{f(v+s)}{F^c(v)}  \bar\mu_0(dv)  ds \\
&= \int_0^\infty \int_0^t \varphi(t+v)^2 \frac{F^c(t+v)^2}{F^c(v)}   \frac{f(v+s)}{F^c(v+s)^2}  ds   \bar\mu_0(dv) \\
&= \int_0^\infty  \varphi(t+v)^2 \frac{F^c(t+v)^2}{F^c(v)}  \int_0^t  \frac{f(v+s)}{F^c(v+s)^2}  ds   \bar\mu_0(dv) \\
&= \int_0^\infty  \varphi(t+v)^2 \frac{F^c(t+v)^2}{F^c(v)}  \int_0^t   \Big(  \frac{1}{F^c(v+s)}\Big)_s'ds \bar\mu_0(dv) \\
&= \int_0^\infty  \varphi(t+v)^2 \frac{F^c(t+v)^2}{F^c(v)}  \Big(   \frac{1}{F^c(v+t)} -  \frac{1}{F^c(v)} \Big) \bar\mu_0(dv) \\
&=  \int_0^\infty  \varphi(t+v)^2 \frac{F^c(t+v)}{F^c(v)}  \Big(   1-  \frac{F^c(t+v)}{F^c(v)} \Big) \bar\mu_0(dv) 
\end{align*}
which is exactly the expression of $\Var( \hat{\mu}^{0}_t(\varphi) )$. This proves the claim above. 

\smallskip

Now for the process $\widehat{R}^N(t)$, from \eqref{eqn-Rn-rep}, we obtain 
\begin{align*} 
\widehat{R}^N_t =\widehat{R}^N_0 + 
\int_0^\infty \Big(1-\frac{F^c(t+\mfa)}{F^c(\mfa)}\Big)  \hat\mu^N_0(d\mfa)  + 
  \int_0^t F(t-s)  \widehat\Upsilon^N(s)ds +   \widehat{R}^{N,0}_t +  \widehat{R}^{N,1}_t \,,
\end{align*}
where
\begin{equation*} 
 \widehat{R}^{N,0}_t = \frac{1}{\sqrt{N}}\sum_{j=1}^{I^N(0)}\Big({\bf1}_{\eta^0_{j} \le t} - \Big(1- \frac{F^c(t+\tilde{\tau}_{j,0})}{F^c(\tilde{\tau}_{j,0})} \Big) \Big) 
\end{equation*}
and
\begin{equation*} 
 \widehat{R}^{N,1}_t= 
\frac{1}{\sqrt{N}}\int_0^t  \int_0^\infty \int_0^\infty {\bf1}_{\eta\le t-s} {\bf1}_{u < \Upsilon^N(s^-)}   \overline{Q}(ds,d \eta, du) 
\end{equation*}
Then  a slight modification of the proof of   \cite[Theorem 2.4]{PP2020-FCLT-VI} shows the weak convergence of $\widehat{R}^N_t$ to $\widehat{R}_t$ in \eqref{eqn-wh-R}. 

%
%
%
\medskip

\paragraph{\bf Declarations}
The authors have no conflicts of interest to declare that are relevant to the content of this article.

\bibliographystyle{abbrv}
\bibliography{Epidemic-Age-FCLT}

\end{document}